\documentclass[12pt]{amsart}
\usepackage{amsmath, amsthm, amsfonts, amssymb, color}
\usepackage{mathrsfs}
\usepackage{amsfonts, amsmath}
 \usepackage{amsmath,amstext,amsthm,amssymb,amsxtra}
 \usepackage{txfonts} 
 \usepackage[colorlinks, citecolor=blue,pagebackref,hypertexnames=false]{hyperref}
 \allowdisplaybreaks
 \usepackage{pgf,tikz}

\begin{document}

 \baselineskip 16.6pt
\hfuzz=6pt

\widowpenalty=10000

\renewcommand{\theequation}
{\thesection.\arabic{equation}}

\def\SL{\sqrt H}

\newcommand{\mar}[1]{{\marginpar{\sffamily{\scriptsize
        #1}}}}

\newcommand{\as}[1]{{\mar{AS:#1}}}

\numberwithin{equation}{section}

\newtheorem{theorem}{Theorem}[section]
\newtheorem{prop}[theorem]{Proposition}
\newtheorem{cor}[theorem]{Corollary}
\newtheorem{definition}[theorem]{Definition}
\newtheorem{lemma}[theorem]{Lemma}
\newtheorem{claim}[theorem]{Claim}
\newtheorem{condition}[theorem]{Condition}
\newtheorem{question}[theorem]{Question}
\newtheorem{conjecture}{Conjecture}
\newtheorem*{theorem*}{Theorem}

\theoremstyle{remark}
\newtheorem{example}[theorem]{Example}
\newtheorem{remark}[theorem]{Remark}
\newtheorem{remark*}{Remark}

\newcommand{\cA}{\mathcal{A}}
\newcommand{\cB}{\mathcal{B}}
\newcommand{\cC}{\mathcal{C}}
\newcommand{\cD}{\mathcal{D}}
\newcommand{\cF}{\mathcal{F}}
\newcommand{\cH}{\mathcal{H}}
\newcommand{\cK}{\mathcal{K}}
\newcommand{\cL}{\mathcal{L}}
\newcommand{\cM}{\mathcal{M}}
\newcommand{\cN}{\mathcal{N}}
\newcommand{\cP}{\mathcal{P}}
\newcommand{\cS}{{\mathcal{S}}}
\newcommand{\cU}{\mathcal{U}}
\newcommand{\cX}{\mathcal{X}}
\newcommand{\dsone}{\mathds{1}}
\newcommand{\spa}{\mathrm{span}}
\newcommand{\dom}{\mathrm{Dom}}
\newcommand{\sot}{\mathrm{SOT}}
\newcommand{\sr}{\mathrm{SR}}
\newcommand{\loc}{\mathrm{loc}}
\newcommand{\bB}{\mathbb{B}}
\newcommand{\bC}{\mathbb{C}}
\newcommand{\bD}{\mathbb{D}}
\newcommand{\bE}{\mathbb{E}}
\newcommand{\bF}{\mathbb{F}}
\newcommand{\bG}{\mathbb{G}}
\newcommand{\bN}{\mathbb{N}}
\newcommand{\bR}{\mathbb{R}}
\newcommand{\bZ}{\mathbb{Z}}
\newcommand{\bT}{\mathbb{T}}
\newcommand{\di}{\mathrm{div}}
\newcommand{\ri}{\mathrm{i}}

\newcommand{\fU}{\mathfrak{U}}
\newcommand{\fC}{\mathfrak{C}}
\newcommand{\ad}{\mathrm{ad}}

\newcommand{\ctimes}{\rtimes_\theta}
\newcommand{\mtx}[4]{\left(\begin{array}{cc}#1&#2\\#3&#4\end{array} \right)}
\newcommand{\dvarphi}[2]{D^{(#1)}(\varphi_{#2})}
\newcommand{\xpsi}[2]{\psi^{[#1]}_{#2}(\omega)}
\newcommand{\xPsi}[2]{\Psi^{[#1]}_{#2}}
\newcommand{\dt}[1]{\partial_t^{(#1)}}
\newcommand{\vN}{vN(G)}
\newcommand{\add}[1]{\quad \text{ #1 } \quad}
\newcommand{\xp}{\rtimes_\theta }
\newcommand{\crossa}{\cM \rtimes_\theta G}
\newcommand{\uu}[2]{u_{#1}^{(#2)}}

\newcommand{\Lp}{{L_p(vN(G))}}
\newcommand{\Lplcr}[1]{{L_p(#1;\ell_2^{cr})}}
\newcommand{\Lplc}[1]{{L_p(#1;\ell_2^{c})}}
\newcommand{\Lplr}[1]{{L_p(#1;\ell_2^{r})}}
\newcommand{\Lpinfty}[1]{{L_p(#1;\ell_\infty)}}
\newcommand{\Lpli}[1]{{L_p(#1;\ell_1)}}
\newcommand{\Lpco}[1]{{L_p(#1;c_0)}}
\newcommand{\Lpcor}[1]{{L_p(#1;c_0)}}
\newcommand{\LpR}{{L_p(\bR^d;L_p(\cN))}}

\newcommand\hDelta{{\bf L}}
\def\RN {\mathbb{R}^n}
\renewcommand\Re{\operatorname{Re}}
\renewcommand\Im{\operatorname{Im}}

\newcommand{\mc}{\mathcal}
\newcommand\D{\mathcal{D}}
\def\hs{\hspace{0.33cm}}
\newcommand{\la}{\alpha}
\def \l {\alpha}
\def\ls{\lesssim}
\def\su{{\sum_{i\in\bN}}}
\def\lz{\lambda}
\newcommand{\eps}{\varepsilon}
\newcommand{\pl}{\partial}
\newcommand{\supp}{{\rm supp}{\hspace{.05cm}}}
\newcommand{\x}{\times}
\newcommand{\lag}{\langle}
\newcommand{\rag}{\rangle}
\newcommand\wrt{\,{\rm d}}
\newcommand{\botimes}{\bar{\otimes}}
\def\bN{{\mathbb N}}
\def\bx{{\mathbb X}}
\def\fz{\infty}
\def\r{\right}
\def\lf{\left}
\def\cm{{\mathcal M}}
\def\cs{{\mathcal S}}
\def\bR{{\mathbb R}}
\def\lm{\mathcal{N}}
\def\rd{\mathbb{R}^d_{\theta}}
\def\rii{\mathbb{R}^2_{\theta}}
\def\rn{{{\bR}^n}}
\def\bZ{{\mathbb Z}}
\def\cl{{\mathcal L}}
\def\cq{{\mathcal Q}}
\def\cd{{\mathcal D}}
\def\lt{\lambda_{\theta}}
\def\wt{W_{\theta}}
\def\schwartz{\mathcal{S}(\mathbb{R}^d)}
\def\temper{\mathcal{S}'(\mathbb{R}^d)}
\def\schwartzt{\mathcal{S}(\mathbb{R}_{\theta}^d)}
\def\tempert{\mathcal{S}'(\mathbb{R}_{\theta}^d)}
\def\sotto{\overset{\sot}{\rightarrow}}
\def\srto{\overset{\sr}{\rightarrow}}
\def\wto{\overset{w(L_p)}{\rightarrow}}
\theoremstyle{assumption}
\newtheorem{assumption}[theorem]{Assumption}

\title[The nonlinear estimates on quantum Besov spaces]{The nonlinear estimates on quantum Besov spaces}

\author{Deyu Chen}
\address{Deyu Chen, Institute for Advanced Study in Mathematics, Harbin Institute of Technology, Harbin 150001, China}
\email{1201200317@stu.hit.edu.cn}

\author{Guixiang Hong}
\address{Guixiang Hong, Institute for Advanced Study in Mathematics, Harbin Institute of Technology, Harbin 150001, China}
\email{gxhong@hit.edu.cn}

  \date{\today}

\begin{abstract}
The superposition operators have been widely studied in nonlinear analysis, which are essential for the well-posedness theory of nonlinear equations. In this paper, we investigate the boundedness estimates of superposition operators with non-smooth symbols on quantum Besov spaces, which significantly generalize  McDonald's results \cite{McNLE} for infinitely differentiable symbols and have rich applications in the well-posedness theory of noncommutative PDEs.  The ingredients in the proof involve a novel quantum chain rule and nonlinear interpolation. As a byproduct, we prove the equivalence of the two descriptions of quantum Besov spaces, resolving the conjecture proposed in \cite[Remark 3.16]{McNLE}.
\end{abstract}

\maketitle

\tableofcontents\newpage

\section{Introduction}
\setcounter{equation}{0}

As the noncommutative deformation of Euclidean spaces $\bR^d$, quantum Euclidean spaces $\rd$ serve as prototypical examples of noncommutative non-compact manifolds. The original definition of quantum Euclidean spaces was introduced by Groenewold \cite{Groenewold} and Moyal \cite{Moyal}, who regarded them as operator formulations of classical Euclidean spaces but endowed with a noncommutative coordinate system $\{x_j\}_{j=1}^d$ satisfying $[x_i, x_j] = \mathrm{i} \theta_{i, j}$. In this framework, the classical product is replaced by the Moyal product. In recent research, many authors have also considered the quantum Euclidean spaces as the von Neumann algebra generated by the twisted left-regular representation on $L_2(\bR^d),$ see for example \cite{GJP2021}.

\subsection{Background}

\subsubsection{PDE theory on quantum Euclidean spaces}
The quantum Euclidean spaces appeared frequently in the mathematical physics literature, including in instantons theory \cite{NekrasovSchwarz}, noncommutative string theory \cite{SW1999}, and noncommutative field theory \cite{DouglasNekrasov}, and also play important roles in the noncommutative harmonic analysis in a series of works \cite{GJM,Hong,HLW2023,hlw25,LSZ,LMSZ,MSX}. From the quantum machine theory, Seiberg and Witten \cite{SW1999} developed the noncommutative gauge theory, which has led to various studies of partial differential equations (abbreviated as PDEs) on quantum Euclidean spaces, such the Maxwell equations \cite{GJPP2001}, Hamiltonian dynamics \cite{BSS2014,Ma2018}, continuity equations \cite{JPP} and Euler equations \cite{DG2016}. Besides these fluid equations, Hamanaka and Toda \cite{ht02,ht03,ht06} also derived from the noncommutative Yang-Mills theory and Lax representation the quantization of many integrable systems, such as Burgers equations, KdV equations, mKdV equations and nonlinear Schr{\"o}dinger equations. Although many authors studied these noncommutative equations from different perspectives, there are nearly no result on the theory of well-posedness of these PDEs  due to the bad behaviors of Moyal product on function spaces. These challenges strongly motivate us to study the well-posedness theory of PDEs on quantum Euclidean spaces.


Fortunately, exploiting the noncommutative harmonic analysis technique, there have been some progress on the well-posedness theory of noncommutative PDEs. As some direct applications of noncommutative harmonic analysis, Gonz\'{a}les-P\'{e}rez,
Junge and Parcet \cite{GJP2021} studied the $L_p$-regularity of linear elliptic pseudodifferential
equations; Fan, Hong and Wang \cite{FHW25} obtained the sharp endpoint $L_p$ estimates of the free quantum Schr\"odinger equations, and Hong, Lai and Wang \cite{hlw25} also established a local smoothing estimate of the free wave equations on 2D quantum Euclidean space. Beyond theses linear estimates, McDonald \cite{McNLE} investigated the well-posedness of Allen-Cahn equations $$\partial_tu=\Delta u+F(u)$$ with infinitely differentiable symbols $F$ on quantum Besov space, as well as nonlinear Schr{\"o}dinger equations and Navier-Stokes equations, utilizing the nonlinear estimates of superposition operators with infinitely differentiable symbols. Ruzhansky et al. \cite{RST,RST25,MRST} developed some nonlinear estimates or Sobolev type inequalities to study nonlinear evolution equations on quantum Euclidean spaces in an abstract way. In our previous paper \cite{CHWW}, the authors also investigated the nonlinear estimate of bilinear form $B(u,v)=\mathbb P((u\cdot \nabla)v)$ with Leray projection $\mathbb P$ by a transference technique, which yields the quantum analogues of Ladyzhenskaya \cite{l69} and Kato's \cite{k84} well-known results for Navier-Stokes equations. 

The nonlinear terms in the aforementioned PDEs are not particularly difficult to deal with. However, there exist many nonlinear PDEs with more complicated nonlinear terms. For instance, in classical model of Allen-Cahn equations, the symbols $F$ are not expected to be infinitely differentiable, or even high-order differentiable (cf. e.g. \cite{Bartels,Lunardi}). Hence it is natural to ask whether the infinitely differentiable condition in \cite[Theorem 7.4]{McNLE} can be reduced.
In this paper, we aim to address such problems in the noncommutative setting by developing the nonlinear estimates of superposition operators with non-smooth symbols, which have been widely studied in the classical theory of nonlinear analysis and PDEs.

\subsubsection{Nonlinear superposition operators}
Let us first recall the theory of superposition operators (also called Nemytskij operators). Recalling in the classical case, given a Borel function $F:\bR\to \bR,$ we can define the related superposition operator $$T_F:u\mapsto F(u),$$ where $u$ belongs to some real-valued Banach function spaces $X$ on $\bR^d$ and $F(u)(x)=F(u(x))$ for any $x\in \bR^d.$ In the noncommutative setting, we will define $T_F:u\mapsto F(u)$ as the Borel functional calculus when $u$ is self-adjoint. We refer to $F$ as the symbol of $T_F$. To the end, we sometimes abuse the notation $F$ itself to denote the superposition $T_F$ whenever no confusion arises. In this paper, we mainly concern the following properties of superposition operators, which have various applications in the well-posedness theory of nonlinear PDEs (cf. e.g. \cite{D2022,MZ2004,RS1996}):
\begin{enumerate}
    \item[\rm{(i)}] Closedness: $T_F(X)\subset X.$
    \item[\rm{(ii)}] Boundedness: For every $M>0,$ there exists a constant $K_M$ such that $\|T_F(u)\|_{X}\le K_M\|u\|_X$ for any $u\in B_X(0,M).$
    \item[\rm{(iii)}]Local Lipschitzness: For every $M>0,$ there exists $K_M$ such that $\|T_F(u)-T_F(v)\|_{X}\le K_M\|u-v\|_X$ for any $u,v\in B_X(0,M),$
\end{enumerate}
where $B_X(0,M)$ denotes the open ball $\{u\in X:\|u\|_X< M\}.$ We can regulate the boundedness and local Lipschitzness to the following estimates:
\begin{align}
 \label{boundedness}   & \|T_F(u)\|_{X}\le C(F,\|u\|_{X},X)\|u\|_{X}.\\
 \label{local Lip}   & \|T_F(u)-T_F(v)\|_{X}\le C(F,X,\|u\|_{X},\|v\|_{X})\|u-v\|_{X},
\end{align}
The inequalities \eqref{boundedness} and \eqref{local Lip} are referred to as the {\it boundedness estimate} and the {\it local Lipschitz estimate} of $T_F$, respectively, which serve as the main analytical tools in our study of the well-posedness of nonlinear PDEs

There have been many well-known results concerning the referred mapping properties and associated nonlinear estimates in various function spaces, such as Lebesgue spaces \cite{AZ1990}, Sobolev spaces \cite{Bourdaud1991} and Tribel-Lizorkin spaces \cite{BMS2008}. In particular, for the Besov spaces $B_{p,q}^s(\bR^d),s>\frac dp,$ as well as  $B_{p,q}^s(\bR^d)\cap L_\infty(\bR^d),s>0,$ Peetre \cite{Peetre1970} obtained the first result in 1970, who proved that $T_F$ is closed on Besov spaces when $F$ is infinitely differentiable and $F(0)=0.$
Subsequently, the related boundedness estimates were established by Meyer \cite{Meyer1981} and Runst \cite{Runst1985}, utilizing Meyer's decomposition and the mapping properties of pseudodifferential operators on Besov spaces.
Hereafter, Runst \cite{Runst1986} studied the case when $F$ belongs in the H{\"o}lder space $C^r(\bR),r>s.$ Exploiting Taylor expansion of $T_F(f),f\in B_{p,q}^s(\bR^d)\cap L_\fz(\bR^d)$ and Fefferman-Peetre-Stein maximal inequality, he obtained an optimal boundedness estimate. This approach can also be applied to the unbounded case $s\le\frac dp,$ see \cite{Sickel1989}.
Another useful technique was developed by Bourdaud and collaborators \cite{BK1995,BM1991}, who employed the difference characterization
\begin{align}\label{classical eq norm}
    \left \| u \right \|_{B_{p,q}^s}\sim_{s,m,N,p,q} \sum_{i=1}^d\left \| u \right \|_p+\lf(\int_0^{\infty}\lf(t^{-s+N}\sup_{|h|\le t}\lf\|\Delta^m(h,\partial_i^Nu)\r\|_p^q\r) \frac{dt}{t}\r)^{1/q}
\end{align}
of Besov spaces to investigate the Besov spaces of order $0<s<1+\frac 1p$ in the case $F(\cdot)=|\cdot|$ or more generally, $F\in \mathrm{Lip}(\bR).$  Here $\mathrm{Lip}(\bR)$ denotes all the Lipschitz functions on $\bR$ with norm $$\|f\|_{\mathrm{Lip}}:=\sup_{x\neq y}\frac{|f(x)-f(y)|}{|x-y|}.$$
It is also interesting to study the necessary condition for the mapping properties of $T_F.$ We do not introduce details here and some relevant results can be founded in \cite{Bourdaud1992,Bourdaud1993,RS1996,Sickel1997} and references therein.

In the noncommutative setting, such as a von Neumann algebra $\cM,$ it is difficult to deduce analogous nonlinear estimates even for the noncommutative $L_p$ space. For instance, it is straightforward to derive the Lipschitz estimates 
\begin{align}\label{Lip}
    \|F(u)-F(v)\|_p\le \|F\|_{\mathrm{Lip}}\|u-v\|_p
\end{align}
for $u,v\in L_p(\bR^d)$ by the triangle inequality, while in the noncommutative case when $u,v\in L_p(\cM),$ this became a well-known conjecture proposed by Daletskii and Krein \cite{DK1956}. The main difficulty came from the lack of the notions of "points". Moreover, unlike the classical case, \eqref{Lip} does not hold for $p=1,\fz$ (cf. \cite{Kato}). To address these challenges,
Daletskii and Krein \cite{DK1956}, as well as Birman and Solomyak \cite{Birman-Solomyak-I,Birman-Solomyak-II,Birman-Solomyak-III}, developed an important technique known as double operator integral. Exploiting this technique, Potapov and Sukochev \cite{PS2011} eventually proved that \eqref{Lip} is valid for $1 < p < \infty$ up to a constant depending on $p$, while Peller \cite{Peller1985} obtained a weaker estimate \begin{align}\label{Beses}
    \|F(u)-F(v)\|_p\lesssim_p \|F\|_{\tilde B_{\fz,1}^1}\|u-v\|_p
\end{align}
for $1\le p\le \fz.$ 

We now turn our attention to more intricate operator spaces, such as Besov spaces on quantum Euclidean spaces (abbreviated as quantum Besov spaces). Motivated by the previous works, McDonald \cite{McNLE} utilized the double operator integral technique to generalize Meyer's decomposition \cite{Meyer1981} and Runst's results \cite{Runst1985} to the setting of quantum Besov spaces, which yields the nonlinear estimates \eqref{boundedness} and \eqref{local Lip} when $F$ is smooth. Notably, this represents the first instance in which the nonlinear estimates for noncommutative superposition operators acting on operator spaces with nonzero regularity are investigated. However, for superposition operators with non-smooth symbols, McDonald's method requires a high order differentiability on the symbol $F,$ which seems far away from Runst's \cite{Runst1986} optimal results. Addressing this limitation is a central aim of the present paper, in which we attempt to extend nonlinear estimates to a broader class of symbols in the noncommutative framework.


\subsection{Main results}
In this paper, we establish the nonlinear estimates \eqref{boundedness} for the nonlinear superposition operators with non-smooth symbols, acting on the quantum Besov spaces. These results not only generalize Runst \cite{Runst1986} and Bourdaud's results \cite{Bourdaud1993}, but also improve McDonald's \cite{McNLE} boundedness estimates.
 In what follows, let $B_{p,q}^s(\rd)$ denote the quantum Besov spaces. For $n\in \bN^+,$ let $F^{(n)}$ denote the $n$th derivative of $F$. If $n=0,$ set $F^{(0)}:=F.$ Let
$C^{n}(\bR)$ denote all the $n$th order continuously differentiable complex-valued functions $F$ on $\bR$ and $C_b^n(\bR)$ denote the subspace of $C^n(\bR)$ consisting of $F\in C^n(\bR)$ such that $F^{(k)},k=0,\dots,n$ are bounded, equipped with the norm $$\|F\|_{C^n_b}:=\sup_{0\le k\le n} \lf\|F^{(k)}\r\|_{\fz},$$   Let $\tilde{B}_{\infty,1}^{n}(\bR)$ denote the modified Besov space on $\bR,$ and its definition can be seen in Section \ref{s3}. For $s>0,$ we set $[s]$ to be the integer part of $s$ and $\{s\}=s-[s].$ Moreover, we set $\left \lceil s \right \rceil:=[s]+1.$ 

The first result concerns the global boundedness estimates of $T_F$ on $B_{p,q}^s(\rd),0<s<1,$ which generalizes the sufficient part of Bourdaud's result \cite{Bourdaud1993}. Regarding the necessity part, no noncommutative result is known in the literature.
\begin{theorem}\label{s smaller than 1}
    Let $1\le p,q\le\infty,0<s<1$ and $u\in B_{p,q}^s(\rd)$ be self-adjoint. If $F(0)=0$ and one of the following assumptions holds:
\begin{enumerate}
    \item[\rm{(A)}]$F$ is locally Lipschitz and $u\in L_\infty(\rd);$ 
    \item[\rm{(B)}]$F$ is Lipschitz,
\end{enumerate}
then for $1<p<\infty$ we have \begin{align}
    \|F(u)\|_{B_{p,q}^s}\lesssim_{s,p,q} \|F\|_{\mathrm{Lip}}^{\|u\|_\fz,\,\mathrm{loc}}\|u\|_{B_{p,q}^s};
\end{align}
if $F(0)=0$ and one of the following assumptions holds:
\begin{enumerate}
    \item[\rm{(C)}]$F \in_{\loc} \tilde{B}_{\infty,1}^{1}(\mathbb{R})$ and $u\in L_\fz(\rd);$ 
    \item[\rm{(D)}]$F \in \tilde{B}_{\infty,1}^{1}(\mathbb{R}),$ 
\end{enumerate}
then for $1\le p\le\infty$ we have \begin{align}
    \|F(u)\|_{B_{p,q}^s}\lesssim_{s,p,q} \|F\|_{\tilde{B}_{\infty,1}^{1}}^{\|u\|_\fz,\,\mathrm{loc}}\|u\|_{B_{p,q}^s}.
\end{align}
\end{theorem}

\begin{remark}
Here we explain the meaning of the notions $\in_{\mathrm{loc}}$ and $\|\cdot\|_{X}^{\|u\|_\fz,\,\mathrm{loc}}.$ For a Borel function $F:\bR\to \bR$ and a function space $X(\bR)$ on $\bR,$ we say that $F$ belongs to $X(\bR)$ locally if $F\phi\in X(\bR)$ for every $\phi\in C_c^\infty(\bR).$ We simply write this as $F\in_{\loc}X(\bR).$  Fix a bump function $\phi$ such that $\phi=1$ on $[-1,1],\phi=0$ outside of $[-2,2]$ and $0\le \phi\le 1$ elsewhere. Set $\phi_{M}(\cdot):=\phi(\frac{\cdot}{M}),M>0$ then we define
 \begin{align*}
\|F\|_{X}^{M,\,\mathrm{loc}}:=\|F\phi_{M}\|_{X},
 \end{align*}
 with the convention that $\|F\|_{X}^{M,\,\mathrm{loc}}=\|F\|_X$ for every $M\in\bR_+\cup \{\infty\}$ when $F\in X(\bR).$ 
\end{remark}
  
For $s\ge 1$ and $s>\frac dp,T_F$ may fail to be globally bounded. Instead, we deduce the relevant local boundedness estimates of $T_F$ on $B_{p,q}^s(\rd),s>\frac dp,$ which generalize Runst's main result \cite[Theorem 3]{Runst1986} (see also \cite[Theorem 1]{Sickel1996}) and follows from Theorem \ref{s smaller than 1} when $0<s<1$.
\begin{theorem}\label{main re 1}
    Let $1\le p,q\le\infty,s>\frac dp$ and $u\in B_{p,q}^s(\rd)$ self-adjoint. If $F\in C^{\left \lceil s \right \rceil}(\bR)$ and $F(0)=0,$ then there exists a positive non-decreasing continuous function $h:\bR_+\to\bR_+$ such that for $1<p<\infty,$ we have \begin{align}\label{bound es} \|F(u)\|_{B_{p,q}^s}\le  h\lf(\|u\|_{B_{p,q}^s}\r)\|u\|_{B_{p,q}^s},
    \end{align}
    where for some $C(s,p,q)>0$ we have \begin{align}\label{hdepend}
      h(t)\lesssim_{s,p,q} \|F\|_{X}^{C(s,p,q)t,\,\loc}\lf(1+(2t)^{s}\r), \ \ \ \forall t\in\bR_+.
 \end{align}
    If $F \in_{\loc} \tilde{B}_{\infty,1}^{1}(\mathbb{R})\cap \tilde{B}_{\infty,1}^{\left \lceil s \right \rceil}(\mathbb{R})$ and $F(0)=0,$ then for $1\le p\le\infty$ we have \eqref{bound es}. 
In particular, if $s\notin \bN,$ we have
\begin{align}\label{bound es 1}
\|F(u)\|_{B_{p,q}^s}\lesssim_{s,p,q}\|F\|_{X}^{\|u\|_\fz,\,\loc}\lf(1+\|u\|_{B_{p,q}^s}^{[s]}\r)\|u\|_{B_{p,q}^s}.
\end{align}

\end{theorem}

 We should point out that this result has some difference from the classical case. First, the classical result only requires $F\in C^r(\bR),r>s,$ whereas our result requires  $F\in C^{\lceil s\rceil}(\bR)$. Second, in the noncommutative setting, the conditions imposed on $F$ for $1<p< \infty$ and for $p = 1, \infty$ differ substantially. This distinction originates from the boundedness properties of operator integrals on noncommutative $L_p$ spaces, which will be discussed in Section \ref{s3}.

When $s\le \frac dp,$ the lack of embedding $B_{p,q}^s(\rd)\hookrightarrow L_\infty(\rd)$ forces us to consider the boundedness estimate of $T_F$ on the multiplication algebra $B_{p,q}^s(\rd)\cap L_\infty(\rd).$ With more regularity assumption on $F,$ we obtain the following local boundedness estimate of $T_F$ on $B_{p,q}^s(\rd)\cap L_\infty(\rd),s>0,$ which follows from Theorem \ref{main re 1} when $s>\frac dp$. Let $W_n(\bR)$ denote the Wiener space, where the definition can also be seen in Section \ref{s3}. 
\begin{theorem}\label{main re 4}
   Let $1\le p,q\le\infty,s>0$ and $u\in B_{p,q}^s(\rd)\cap L_\fz(\rd)$ self-adjoint, if $F\in_{\mathrm{loc}} W_0(\bR)\cap W_{\lceil s\rceil}(\bR)$ and $F(0)=0,$ then we have \begin{align}\label{bound es 4}
\|F(u)\|_{B_{p,q}^s}\lesssim_{s,p,\|u\|_\fz}\|F\|_{W_0\cap W_{\lceil s\rceil}}^{\|u\|_\fz,\,\mathrm{loc}} \|u\|_{B_{p,q}^s}.
    \end{align}
\end{theorem}

\begin{remark}
    Theorem \ref{main re 4} requires more regularity assumption on $F$ than Theorem \ref{main re 1}. For example, when $\lceil s\rceil=2k+1$ for some $k\in \bN^+$ and $F=|x|x^{2k+1},$ we then have $F\in C^{2k+1}(\bR),$ whereas $F\notin_{\mathrm{loc}} W_0(\bR)\cap W_{2k+1}(\bR)$ due to the singularity of the Hilbert transform on $L_1(\bR).$ Unfortunately, we do not know how to remove this gap.
\end{remark}

In this article, for the proof of main results, we will deeply explore the multiple operator integral theory introduced in \cite{ACDS2009, CMS2021, Peller2006, PSS2013}, which can be viewed as a multivariate extension of double operator integral. Specifically, we will first present two equivalent difference characterizations of quantum Besov spaces in Theorem \ref{equiv ch} and Theorem \ref{Laf}, which is well-known in the classical case $\theta=0$. This resolves the problem posed by McDonald in \cite[Remark 3.16]{McNLE}. Subsequently, we will establish a quantum chain rule formula in Theorem \ref{nc chain rule} that is essential to the proof of nonlinear estimates, and its proof turns out quite technical. More precisely, we need to pick up $p_k$'s carefully from the underlying assumption so that the Sobolev/Besov embeddings and the boundedness of multiple operator integrals can be applied simultaneously. The  quantum chain rule formula and its proof constitute the main novel part of the present paper. By leveraging the difference characterizations of quantum Besov spaces and the quantum chain rule and its proof, one can demonstrate the main results in the case $s\notin \bN.$ However, the case $s\in \bN$ cannot be concluded in a similar way since the difference characterizations are not helpful any more. We will exploit the nonlinear interpolation (see e.g. \cite[Theorem 2]{Tartar}, \cite[Theorem 1]{M1989}). In order to apply the nonlinear interpolation which requires sharper endpoint estimates than the linear case, we improve significantly the techniques appearing in the proof of both quantum chain rule and the case $s\notin \bN$, see the arguments in Proposition \ref{main re 2} and Lemma \ref{gap} for details. 

\begin{remark}
 In the classical case, Runst's results \cite{Runst1985,Runst1986} also hold for $0<p,q\le \infty$ with $s>\frac{d}{\min(p,q,1)}-d.$  It is therefore natural to ask whether Theorem \ref{main re 1} can be extended to the case $0<p,q\le\infty.$  Note that in the noncommutative setting, Ricard \cite{Ricard} established the H{\"o}lder estimate for $|\cdot|^\theta$ with $0<\theta<1$ and $0<p<\infty$ on the Schatten classes. McDonald and Sukochev \cite{MS} later improved his results to a wider function class by developing the double operator integral technique for the case $0<p<1,$ which was generalized by Casper and Huisman \cite{CH} to the multiple operator integral. In light of these advances, we believe that our main results can also be extended to the case $0<p,q\le\infty,$ which we plan to investigate in further works.  
\end{remark}

Finally, we present a significant application of our results to the well-posedness of noncommutative partial differential equations (PDEs), specifically focusing on the Allen-Cahn equations.
The Allen-Cahn equation is defined by the following Cauchy problem:
\begin{align}\label{Allen Cahn}
    \partial_t u=\Delta u+F(u), \ \ \ u(0)=u_0
\end{align}
and its integral form is
\begin{align}\label{integral AC}
    u(t)=e^{t\Delta}u_0+\int_0^t e^{(t-\tau)\Delta}F(u(s))\,d \tau,
\end{align}
where $F$ is a real-valued Borel function. We say that $u$ is a {\it strong solution} if $u$ satisfies \eqref{Allen Cahn} and $\partial_t u$ exists in the Fr{\'e}chet sense and a {\it mild solution} of \eqref{Allen Cahn} if $u$ satisfies \eqref{integral AC}. 
The well-posedness theory of Allen-Cahn equations is crucial within the classical theory of PDEs, since this class of equations includes many well-known nonlinear parabolic equations. In particular, for their well-posedness on Besov spaces, Miao and Zhang \cite{MZ2004} considered the well-posedness of more general semilinear parabolic equation $$\partial_t u=\Delta u+F, \ \ \ u(0)=u_0$$ in the Besov space $B_{p,2}^s(\bR^d)$ with $s\ge \frac dp,1<p<\infty,$ where $F=F(u)$ (Allen-Cahn equation) or $m(D)F(u)$ with some homogeneous
pseudodifferential operator $m(D),$ such as $\mathbb P\di(u\otimes u)$ (Navier-Stokes equation) and $(\vec{a}\cdot \di)(|u|^\alpha u)$ (convection-diffusion equation). In this article, we aim to investigate the well-posedness of the noncommutative Allen-Cahn equation in $B_{p,2}^s(\rd)$ in the case $s>\frac dp,1<p<\infty$. 

We adopt the following notations.
For a Banach space $X$ and a time interval $I,$ let $C(I;X)$ denotes all continuous functions on $I$ with value in $X$, equipped with the norm $$\|f\|_{C(I;X)}:=\sup_{t\in I}\|f(t)\|_{X}.$$ For $u\in C(I;X),$ let $u(t)$ denote the value of $u$ at time $t$ and $\partial_t u$ denote the Fr{\'e}chet derivative of $u$ with respect to $t,$ that is $\partial_t u$ exists in the sense that $$\lim_{h\to 0}\lf\|\frac{u(t+h)-u(t)}{h}-\partial_t u(t)\r\|_{X}=0.$$ Define $C^1(I;X)$ as the space containing all $u\in C(I;X)$ such that $\partial_t u\in C(I;X).$ Let $F(u)$ denote the operator-valued function $t\mapsto F(u(t)).$ Then exploiting Theorem \ref{s smaller than 1} and Theorem \ref{main re 1}, we proved the following well-posedness result.
\begin{theorem}\label{posedness}
        For $n\in \bN^+, 1<p<\infty,\frac dp<s\le n$ and $u_0\in B_{p,2}^s(\rd)$ self-adjoint, if $F \in C^{n}(\bR)$ and $F(0)=0,$ then  
        there exists a unique self-adjoint mild solution $u$ of \eqref{Allen Cahn} and a maximal existence time $T_{u_0}$ such that
        $$
           u \in C([0,T_{u_0});B_{p,2}^s(\rd))\cap \bigcap_{0<\alpha<n+1} C((0,T_{u_0});B^\alpha_{p,2}(\rd)).
        $$
        If $T_{u_0}<\infty,$ then $$\limsup_{t\nearrow T_{u_0}}\|u(t)\|_{B_{p,2}^s}=\fz.$$ Moreover, we also have the following conclusions:
        \begin{enumerate}
            \item[\rm{(i)}]if $n\ge 2,$ then $u$ is indeed a strong solution such that $$u\in C^1((0,T_{u_0});L_p(\rd)),$$
            \item[\rm{(ii)}] if $F$ is also Lipschitz, then $T_{u_0}=\infty.$
        \end{enumerate}
        
    \end{theorem}







\bigskip

\noindent {\bf Notations.}
Conventionally, we set $\bN:=\{0,\,1,\, 2,\,\ldots\},\bN^+:=\bN\setminus \{0\}$
and $\bR_+:=(0,\infty)$. Throughout the whole paper, we denote by $C$ a positive constant which is independent of the main parameters, but it may
vary from line to line. We use $A\lesssim B$ to denote the statement that $A\leq CB$ for some constant $C>0$, and $A\thicksim B$ to denote the statement that $A\lesssim B$ and $B\lesssim A$.  
For any $1\leq p\leq\infty$, we denote by $p'$ the conjugate of $p$, which satisfies $\frac{1}{p}+\frac{1}{p'}=1$.
 For a measurable set $E \subset\bR^d$, let $\chi_E$ denote its characteristic function.
 For two Banach spaces $X,Y$, let $\mathcal{B}(X,Y)$ denote the space of all bounded linear operators from $X$ to $Y$.
 If $X=Y$, then we simply write $\mathcal{B}(X):=\mathcal{B}(X,X)$. 
\bigskip




\section{Preliminaries} \label{s2}
\setcounter{equation}{0}
In this section, we recall some basic definitions and properties of quantum Euclidean spaces and function spaces on them. For a more detailed and comprehensive introduction to the theory of quantum Euclidean spaces, see \cite{Hong,LMSZ}.

\subsection{Quantum Euclidean spaces}\label{ch2.1}
As in \cite{HLW2023,LSZ,MSX}, the quantum Euclidean spaces $\rd$ can be defined as follows.
\begin{definition}
Let $\theta$ be a $d\times d$ real antisymmetric matrix and $t\in\mathbb{R}^d$. Consider the unitary
operator $\lambda_\theta(t)$ on $L_2(\mathbb{R}^d):$
\begin{align*}
		(\lambda_\theta(t)f)(s):=e^{-\frac{\mathrm{i}}2\lag t,\theta s\rag}f(s-t),\quad f\in{L_2(\mathbb{R}^d)},\, s\in\mathbb{R}^d.
\end{align*}
Here $\ri:=\sqrt{-1}$ and $\lag\cdot,\cdot\rag$ denotes the standard inner product on $\bR^d.$
Subsequently, the quantum Euclidean space $\rd$ is defined as the closed subalgebra of $\cB(L_2(\mathbb{R}^d))$ under the strong operator topology (abbreviated as $\mathrm{SOT}$) generated by
the family $\{\lambda_\theta(t)\}_{t\in\mathbb{R}^d}$.
\end{definition}

\begin{remark}
    Note that when $\theta=0,\lambda_\theta(t)$ becomes the translation operator $$\tau_t:f\mapsto f(\cdot-t)$$ and $\rd$ is $\ast$-isomorphic to $\bR^d.$ Moreover, let $\theta=\hbar \begin{pmatrix}
 0 & I_n\\
 -I_n & 0
\end{pmatrix},$ where $I_n$ denotes the $n\times n$ unit matrix and $\hbar$ denotes the Planck's constant, then $\rd$ is $\ast$-isomorphic to the phase space $\cB(L_2(\bR^n))$ under the map $$\lambda_\theta(t_1,t_2)\mapsto e^{\frac{\ri\hbar}{2}\lag t_1,t_2\rag}\tau_{\hbar t_1}\exp(t_2),\ \ \ t_1,t_2\in\bR^n.$$ Here for $t\in\bR^n$ we use $\exp(t)$ to denote the function $$\exp(t):\bR^n\to \bC,s\mapsto e^{\ri\lag t,s\rag}.$$ These relations indicate the physical significance of quantum Euclidean spaces.
\end{remark}


Now we introduce the Weyl quantization on $\rd.$ Given $f \in L_1(\mathbb{R}^d)$, one defines its Weyl quantization $\lt(f): L_2(\mathbb{R}^d)\rightarrow L_2(\mathbb{R}^d)$ as
\begin{equation}\label{defU}
 \lt(f)( g):=\int_{\mathbb{R}^d}f(t)(\lambda_\theta(t)g)\,dt
\end{equation}
for $g\in L_2(\mathbb{R}^d)$.
This $L_2(\mathbb{R}^d)$-valued integral is convergent in the Bochner sense. Then one may show that $\lt(f)\in \rd$ and $\|\lt(f)\|\le \|f\|_1.$ For an explicit proof, see \cite[Lemma 2.3]{MSX}.  



In what follows, we normalize
the Fourier transform of a reasonable function $f$ as
$$\mathcal{F}(f)(\xi):=\hat{f}(\xi) := \int_{\mathbb{R}^d}
f(t)e^{-\mathrm{i}(t,\,\xi)}\,dt, \ \ \  \forall\,\xi\in\mathbb{R}^d,$$
and its inverse Fourier transform as
$$\mathcal{F}^{-1}(f)(t):=\check{f}(t):=(2\pi)^{-d}\int_{\mathbb{R}^d}
f(\xi)e^{\mathrm{i}(\xi,\,t)}\,d\xi, \ \ \  \forall\,t\in\mathbb{R}^d.$$
The Schwartz class on $\rd,$ denoted by $\cS(\rd),$ is defined as the image of the Schwartz class $\mathcal{S}(\mathbb{R}^d)$  on $\mathbb R^d$ under the Weyl quantization $\lt.$   Note that $\lt:\schwartz \to \schwartzt$ is bijective (cf. e.g. \cite[Subsection 2.2.3]{MSX}), hence $\schwartzt$ is a Fr{\'e}chet topological space induced by $\lt.$  The topological dual of $\mathcal{S}(\rd)$ is denoted by $\cS'(\rd).$ Since $\lt$ is bijective, it extends to a topological isomorphism from $\mathcal{S}'(\mathbb{R}^d)$  to $\mathcal{S}'(\rd)$ in the following way: for $T\in\mathcal{S}^\prime(\mathbb{R}^d)$,
$$ (\lt(T),\lt(f))=(T,\tilde{f}), \ \ \ f\in \mathcal{S}(\mathbb{R}^d),$$ 
where $\tilde{f}(\cdot)=f(-\cdot).$

If $x=\lt(f)\in \mathcal{S}(\rd)$ for some $f\in \mathcal{S}(\mathbb{R}^d),$ define $\tau_\theta$ as the linear functional $\tau_{\theta}(x):=f(0)$ on $\cS(\rd),$ then $\tau_{\theta}$ extends to a $n.s.f.$ trace on $\rd$ (cf. \cite[Proposition 1.1]{GJP2021}). As for a general von Neumann algebra, let $L_p(\rd):=L_p(\rd,\tau_{\theta})$ denote the noncommutative $L_p$ spaces associated with $(\rd, \tau_{\theta})$ with norm $\|\cdot\|_p,$ and $L_\infty(\rd):=\rd$ with norm $\|\cdot\|_\infty=\|\cdot\|.$ The Schwartz class $ \mathcal{S}(\rd)$ is dense in $L_p(\rd)$ for $1\leq p<\infty$ in the norm topology, and dense in $L_\infty(\rd)$ in the weak-$*$ topology, for which we refer to \cite{GJP2021, MSX}. It is also straightforward to see that $L_p(\rd)\subset \tempert.$ 

By direct computations, we have the following identities for $f,g\in \schwartz$:
\begin{align*}
    &\tau_{\theta}(\lt(f)\lt(g))=\int_{\bR^d}f(-t)g(t)\,dt.\\
    &\lt(f)^*=\lt(\overline{f(-\cdot)}).\\
\end{align*}
Then the $L_2(\rd)$-inner product $\lag f,g\rag$ can be expressed by $$\tau_{\theta}(\lt(f)^*\lt(g))=\int_{\bR^d}\overline{f(t)}g(t)\,dt,  \ \ \ \text{for\:all\;}f,g\in \mathcal{S}(\mathbb{R}^d).$$ 
Combined with complex interpolation, we deduce the quantum Euclidean analogue of the Hausdorff-Young inequality and Plancherel inequality.
\begin{lemma}\label{Hauss}
Let $f\in \cS(\bR^d),$ then when $1\le p<2,$
$$\|\lt(f)\|_{L_{p'}(\rd)}\leq \|f\|_{L_p(\bR^d)}$$
and when $p=2,$ $$\|\lt(f)\|_{L_{2}(\rd)}=\|f\|_{L_2(\bR^d)}$$
Therefore, $\lt$ extends to a contraction from $L_p(\bR^d)$ to $L_{p'}(\rd)$ when $1\le p<2$ and
a unitary isomorphism from $L_2(\bR^d)$ to $L_2(\rd)$.
\end{lemma}

\subsection{Differential calculus on quantum Euclidean spaces}
There are two equivalent ways to define the partial derivatives in quantum Euclidean spaces, as discussed respectively in \cite{LMSZ,McNLE} and \cite{LSZ,MSX}: one through the translation semigroup and the other through the commutator. In this article, we will focus solely on the former one.
 Let $T_s$ be the unique $*$-automorphism of $L_\infty(\bR^d_\theta)$ which acts on $\{\lt(t):t\in\bR^d\}$ by 
    \begin{equation*}
        T_s(\lt(t)) = e^{\ri\lag t,s\rag}\lt(t),\quad t,s \in \bR^d.
    \end{equation*}
    Then $T_s$ extends to a unitary operator on $L_\fz(\rd)$ such that $T_s(x)=\tau_s x\tau_{-s}.$ We refer to $T_s$ as the {\it translation operator} on $\rd.$ 

\begin{prop}\label{tran prop}
    For $x\in L_p(\rd),$ if $1\le p<\infty,$ then $\|T_sx-x\|_p\to 0$ when $s\to 0$. If $p=\infty,$ then $T_sx\sotto x$ when $s\to 0$. Here the notion $\sotto$ refers to the convergence in $\mathrm{SOT}$.
\end{prop}
\begin{proof}
    The proof for the case $1\le p<\infty$ can be seen in \cite[Theorem 3.6]{MSX}. In the case $p=\infty,$ since for any $\ \xi \in L_2(\rd),$ $$T_s(x)(\xi)=(\tau_s x)\,(\xi)(\cdot+s), $$ we then have \begin{align*}
       \lim_{s\to 0} \|T_s(x)(\xi)-x(\xi)\|_2 &=\lim_{s\to 0}\,\|(\tau_s x) \,(\xi)(\cdot+s)-(\tau_sx)(\xi) \|_2+\|(\tau_sx)(\xi)-x(\xi) \|_2\\
        &\le \|x\|_{\infty}\lim_{s\to 0}\|\,\xi(\cdot+s)-\xi\|_2+\lim_{s\to 0}\|x (\xi)\,(\cdot-s)-x(\xi) \|_2\\
        &=0
    \end{align*}
    from the fact that $\xi,x(\xi) \in L_2(\bR^d)$.
\end{proof}

Due to Proposition \ref{tran prop}, $\{T_{ze_j}\}_{z\ge 0}$ is a $C_0$ semigroup on $\rd,$ where $e_j$ denotes the $j$th standard basis on $\bR^d$ for $j=1,\dots,d.$ The partial derivative $\partial_j$ on $\rd$ will be defined as the infinitesimal generator of the $C_0$ semigroup $\{T_{ze_j}\}_{z\ge 0}.$ Equivalently, let $\partial_j^z$ denote the divided difference operator on $\rd:$
\begin{equation*}
    \partial_j^z x = \frac{T_{ze_j}x-x}{z},
\end{equation*}
    then $x\in \dom(\partial_j)$ if and only if the limit
    $\partial_j x = \lim_{z\to 0}\partial_j^z x$
exists in $\sot$.
For a multi-index $\alpha \in \bN^d$, we define
    \begin{equation*}
        \partial^\alpha = \partial_1^{\alpha_1}\cdots\partial_d^{\alpha_d}.
    \end{equation*}  
 Furthermore, we define the Laplace operator $$\Delta=\sum_{j=1}^d\partial_j^2$$ as in the classical case.
 We say that an element $x$ is {\it smooth} if for every $\alpha\in \bN^d,x\in\dom(\partial^\alpha).$
Now let $\mathrm{x}_j$ denote the generator of the unitary group $\{\exp(ze_j)\}_{z\in \bR}$ on $L_\infty(\bR^d),$ then for any $x=\lambda_\theta(f)\in \cS(\rd),$ we can readily confirm that $x$ is smooth and
    \begin{align*}
  & \partial_j x = \lt(\ri \mathrm{x}_jf),\;j=1,\ldots,d\\
  & T_s x = \lt(\exp(s)f),\;j=1,\ldots,d,
    \end{align*}
By the above definitions and duality, we can also extend $\partial^\alpha$ to $\tempert.$ 
\begin{definition} For any multi-index $\alpha\in\bN^d$ and $w\in\cs'(\rd)$,
we define the distributional derivative $\partial^{\alpha}$ and translation $T_s$ to every distribution $w\in\cs'(\rd)$ as
$$(\partial^{\alpha}w,x)=(-1)^{|\alpha|}( w,\partial^{\alpha}x),
\ \ \   x\in\cs(\rd)$$
and $$( T_sw,x)= (w,T_{-s}x),\ \ \   x\in\cs(\rd).$$
\end{definition}

\begin{remark}
    Exploiting the identity $T_s(x)=\tau_sx\tau_{-s},$ we immediately have $\partial_j x=[\mathbf D_j,x],$ where $\mathbf D_j$ is the $j$th classical partial derivative. This coincides with the definition of $\partial_j$ in \cite{LSZ,MSX}. 
\end{remark}

Now we introduce the quantum Sobolev space on $\rd,$ which was introduced in \cite{LMSZ,MSX}.
\begin{definition}
For $m\in \bN^+$ and $1\le p\le \fz,$ the quantum Sobolev space $W_p^m(\rd)$ is defined as the space of all $x\in\tempert$ such that $\partial^{\alpha}x\in L_p(\rd)$ whenever $|\alpha|\le m,$ equipped with the norm
 \begin{align*}
     \|x\|_{W^{m}_p (\mathbb{R}^{d}_\theta)}:=\sum\limits_{ |\alpha|\le m} \|\partial^{\alpha}x\|_{L_p(\mathbb{R}^{d}_\theta)}.
 \end{align*}
\end{definition}


For the element $x\in W_p^1(\rd),$ we have the following approximation property.
\begin{prop}\label{partial prop}
    For $x\in W_p^1(\rd),$ if $1\le p<\infty,$ then $\|\partial_j^zx-\partial_j x\|_p\to 0$ when $z\to 0$. If $p=\infty,$ then $\partial_j^zx\sotto\partial_j x$ when $z\to 0$.
\end{prop}
\begin{proof}
For $1\le p<\infty,$ suppose $x=\lt(f)\in W_p^1(\rd)$ for some $f\in \temper,$ then for any $y=\lt(g)\in \schwartzt,$ we have
    \begin{align*}
    ( \partial_j^zx-\partial_jx,y)&=
        \lf( \lt\lf(\lf(\frac{\exp(ze_j)-1}{z}-\ri\mathrm x_j \r)f\r),\lt(g) \r)\\
        &=\lf( f,\lf(\frac{\exp(ze_j)-1}{z}-\ri\mathrm x_j \r)\tilde g\r).
    \end{align*}
    Since $$\lf(\frac{\exp(ze_j)-1}{z}-\ri\mathrm x_j \r)\tilde g \to 0$$ in the topology of $\cS(\bR^d),$ we have  \begin{align*}
        \lim_{z\to 0}( \partial_j^zx-\partial_jx,y) =0,
    \end{align*}
Hence it follows that 
   \begin{align*} 
       \partial_j^z x=\frac{1}{z}\int_0^z T_{se_j}\partial_j x\, ds
   \end{align*}
   in distributional sense, which implies that for $1\le p<\infty,$ we have
   \begin{align*}
       \|\partial_j^z x-\partial_j x\|_p\le \frac{1}{z}\int_0^z \|T_{se_j}\partial_j x-\partial_j x\|_p\, ds.
   \end{align*}
Similarly, for $p=\infty$ and every $\xi\in L_2(\bR^d),$ we have
\begin{align*}
       \lf\|\lf(\partial_j^h x-\partial_j x\r)(\xi)\r\|_2\le \frac{1}{z}\int_0^z \lf\|\lf(T_{se_j}\partial_j x-\partial_j x\r)(\xi)\r\|_2\, ds.
   \end{align*}
By Proposition \ref{tran prop}, we conclude the proof.

\end{proof}

\subsection{Fourier multiplier on quantum Euclidean spaces}

For a reasonable symbol $m,$ the Fourier multiplier $m(D)$ on quantum Euclidean spaces is defined on $\schwartzt$ by $$m(D)\lt(f)=\lt(mf), \ \ \ f\in \cS(\bR^d).$$   A direct computation shows that $m(D)$ can be seen as a convolution type operator with kernel $\check{m}:$ $$m(D)x=\check m\ast x,$$ where as in \cite[Section 3.2]{MSX}, the convolution $\check m \ast x$ is expressed by the operator-valued integral $$\check m \ast x=\int_{\bR^d}\check m(s)T_{-s}x\,ds.$$

Now we introduce the noncommutative Young's convolution inequality. For the proof, see e.g. \cite[Lemma 2.7]{CHWW}.
\begin{lemma}
    Let $1\leq p,q,r \leq \infty$ obey the relation
        \[
            \frac{1}{r}+1 = \frac{1}{p}+\frac{1}{q}.
        \]
        Then for $x\in L_p(\rd)$ and $m$ satisfying $\check{m}\in L_q(\bR^d),$ we have 
        \begin{align}\label{Young convolution}
            \|\check m\ast x\|_r\le \|\check{m}\|_q\|x\|_p.
        \end{align}
\end{lemma}


Now we introduce the dilation on quantum Euclidean spaces. Recall that for measurable functions $f$ on $\bR^d,$ the dilation $\delta_\varepsilon$ is defined as $\delta_\varepsilon(f)(\xi)=f(\varepsilon\xi)$ for any $\xi\in\bR^d$ and $\varepsilon>0$. For the quantum Euclidean space $\rd$, we define the corresponding dilation via the Weyl quantization. Specifically, given $x=\lt(f)\in \tempert,$ we set $\delta_\varepsilon(x):=\lambda_{\varepsilon^{-2}\theta}(\delta_{\varepsilon}(f)).$
The $L_p$ boundedness of dilation is proved in \cite[Proposition 3.1]{HLW2023} and \cite[Proposition 3.7]{CHWW}. 
\begin{prop}\label{prop:dilation}
Let $1\leq p\leq\infty$ and $\varepsilon>0$. Then for any $x\in L_p(\rd)$, we have
$$\|\delta_\varepsilon(x)\|_{L_p(\bR_{\varepsilon^{-2}\theta}^d)}=\varepsilon^{-\frac dp}\|x\|_{L_p(\rd)}.$$
\end{prop}

 Let $M_p(\rd)$ denote the space of symbols $m$ such that $m(D)\in \cB(L_p(\rd)),$ equipped with the norm $$\|m\|_{M_p}:=\|m(D)\|_{L_p(\rd)\to L_p(\rd)}.$$ 
The multiplier $m(D)$ with $m\in M_p(\rd)$ is said to be a $L_p$ multiplier, which is an important object in classical harmonic analysis when $\theta=0.$ Exploiting Proposition \ref{prop:dilation}, we can show the following rescaling property.
\begin{prop}\label{rescaling property}
    For $\varepsilon\in \mathbb{R}\setminus \{0\},$ if $\delta_{\varepsilon^{-1}} m\in M_p(\mathbb{R}_{\varepsilon^{2}\theta}^d),$  then $m\in M_p(\mathbb{R}_{\theta}^d)$ and $$\left \|\delta_{\varepsilon^{-1}} m \right \|_{M_p(\mathbb{R}_{\varepsilon^{2}\theta}^d)}=\left \|  m\right \|_{M_p(\mathbb{R}_{\theta}^d)}.$$ 
\end{prop}
\begin{proof}
    For simplicity, set $\tilde{m}:=\delta_{\varepsilon^{-1}}m.$ Suppose that $x=\lt(f)\in L_p(\rd)$ for some $f\in\temper,$ then by the definition of $\delta_\varepsilon,$ we have
    $$x=\delta_{\varepsilon}(\lambda_{\varepsilon^2\theta}(\delta_{\varepsilon^{-1}}(f))).$$ Hence by Proposition \ref{prop:dilation}, we have
    \begin{align*}
         \left \| m(D)x \right \|_{L_p(\mathbb{R}_{\theta}^d)}&=\|\lt(mf)\|_{L_p(\mathbb{R}_{\theta}^d)}\\
         &= \left \| \delta_{\varepsilon}(\lambda_{\varepsilon^2\theta}(\tilde{m}\delta_{\varepsilon^{-1}}(f)))\right \|_{L_p(\mathbb{R}_{\theta}^d)}\\
        &=\varepsilon^{-\frac dp}\left \| \tilde{m}(D)(\lambda_{\varepsilon^2\theta}(\delta_{\varepsilon^{-1}}(f))) \right \|_{L_p(\mathbb{R}_{\varepsilon^{2}\theta}^d)}\\
        &\le \varepsilon^{-\frac dp}\left \| \tilde{m} \right \|_{M_p(\mathbb{R}_{\varepsilon^2\theta}^d)}\left \| \lambda_{\varepsilon^2\theta}(\delta_{\varepsilon^{-1}}(f)) \right \|_{L_p(\mathbb{R}_{\varepsilon^{2}\theta}^d)}\\
        &=\varepsilon^{-\frac dp}\left \| \tilde{m} \right \|_{M_p(\mathbb{R}_{\varepsilon^2\theta}^d)}\left \| \delta_{\varepsilon^{-1}}(x )\right \|_{L_p(\mathbb{R}_{\varepsilon^2\theta}^d)}\\
        &=\left \| \tilde{m} \right \|_{M_p(\mathbb{R}_{\varepsilon^2\theta}^d)}\left \| x \right \|_{L_p(\mathbb{R}_{\theta}^d)}.
    \end{align*}
    We then have $m\in M_p(\rd)$ and $\|m\|_{M_p(\rd)}\le \left \|\tilde m\right \|_{M_p(\mathbb{R}_{\varepsilon^{2}\theta}^d)}.$ The converse direction is similar.
\end{proof}

\section{Quantum Besov spaces}\label{s2.5}
In this section, we will introduce Besov spaces on quantum Euclidean spaces defined in \cite{CHWW,Hong,McNLE}, and establish several fundamental properties, including difference characterizations of quantum Besov spaces.
\subsection{Characterizations of quantum Besov spaces}
\subsubsection{Littlewood-Paley characterization}
As in the classical case, the quantum  Besov spaces are defined via the Littlewood-Paley multipliers.
Let $\varphi$ be a smooth radial function on $\bR^d$ with $\supp \varphi\subset B(0,2)\setminus B(0,\frac 12)$ and $$\varphi(\xi)+\varphi(\frac{\xi}{2})=1, \ \ \  \xi \in B(0,2)\setminus B(0,1).$$
For any $k\in\bZ$, define
$$\varphi_k(\xi):=\varphi(2^{-k}\xi),\ \ \ k\in\bZ.$$
Then we have
$$\sum_{k\in\bZ}\varphi_k(\xi)=1,\ \ \ \xi\in\bR^d\setminus\{0\}.$$
Since $\supp \varphi\subset B(0,2)\setminus B(0,\frac12),\supp \varphi_k\subset B(0,2^{k+1})\setminus B(0,2^{k-1}),k\in\bZ$. The functions $\varphi_k,k\in \bZ$ are called the {\it homogeneous Littlewood-Paley functions} on $\bR^d.$

Let $\phi_k=\varphi_k$ for $k\in\bN^+$ and $\phi_0=R+\varphi,$ where $R$ is a smooth radial function on $\bR^d$ such that $R=1-\varphi$ on $B(0,1)$ and vanishes elsewhere, then we have $$\sum_{k\in\bN}\phi_k(\xi)=1,\ \ \ \xi\in\bR^d.$$ The functions $\phi_k,k\in \bN$ are called the {\it non-homogeneous Littlewood-Paley functions} on $\bR^d.$
Define the Littlewood-Paley multipliers by $\triangle_j=\phi_j(D),j\in \bN^+$ and $\triangle_0=\phi_0(D).$ Obviously, the definition of $\phi_k$ and \eqref{Young convolution} yield the fact that 
\begin{align}\label{uniformLP}
\sup_{k\in\bN}\|\triangle_k\|_{L_p(\rd)\to L_p(\rd)}<\infty, \ \ \ 1\le p\le\infty.
\end{align}

\begin{definition}
Let $s\in\bR$ and $1\le p,q\le\fz$. The non-homogeneous Besov space $B^s_{p,q}(\rd)$ is defined
as the space of all $x\in\cs'(\rd)$ such that $\|x\|_{B_{p,q}^s}<\fz,$ where 
$$\|x\|_{B^s_{
p,q}(\rd)}:=\lf(\sum_{j=0}^{\fz}
2^{jsq}\lf\|\triangle_jx\r\|^q_{L_p(\rd)}\r)^{1/q}\ \ \ \mathrm{for}\ \ q<\fz,$$
and
$$
\|x\|_{B^s_{p,\fz}(\rd)}=\sup_{j\in\bN}2^{sj}\lf\|\triangle_jx\r\|_{L_p(\rd)}.$$
\end{definition}

We abbreviate the non-homogeneous Besov class on quantum Euclidean spaces as the quantum Besov spaces. The quantum Besov spaces have the following basic properties. For the proof we refer to \cite{CHWW,McNLE}.
\begin{prop}
  Let $s\in\bR$ and $1\le p,q\le\fz:$
        \begin{enumerate}
            \item[\rm (i)]\label{sobolev_inclusion} If $s,t\in \bR,r\in \bN$ satisfying $s<r<t,$ then $B^t_{p,q}(\rd)\hookrightarrow W^r_p(\rd)\hookrightarrow B_{p,q}^s(\rd).$ 
            Moreover, we have $B_{p,1}^r(\rd)\hookrightarrow W_p^r(\rd)\hookrightarrow B_{p,\infty}^r(\rd).$ 
            
            \item[\rm (ii)]\label{schwartz_dense} If $s\in \bR,$ and $1\le p,q<\infty,$ then $\schwartzt$ is dense in $B^s_{p,q}(\rd)$ in the norm topology. 

            \item[\rm (iii)]\label{besov_embedding} If $p_0 \leq p_1, q_0\leq q_1$ and $s_0 \geq s_1+d\left(\frac{1}{p_0}-\frac{1}{p_1}\right),$ then we have
        \[
            B^{s_{0}}_{p_0,q_0}(\rd)\hookrightarrow B^{s_1}_{p_1,q_1}(\rd).
        \]
        Moreover, if $p\le p_1\le\infty,1\le q\le\infty$ and $s>\frac dp-\frac d{p_1},$ then we have 
        \[
            B^{s}_{p,q}(\rd)\hookrightarrow L_{p_1}(\rd).
        \]

        \item[\rm (iv)]\label{interpolation}The quantum Besov spaces satisfy the following real interpolation relation
        \begin{equation*}
            (B^{\alpha_0}_{p,q_0}(\bR^d_\theta),B^{\alpha_1}_{p,q_1}(\bR^d_\theta))_{\eta,q} = B^{(1-\eta)\alpha_0+\eta\alpha_1}_{p,q}(\bR^d_\theta).
        \end{equation*}
        Here, $1\leq q_0,q_1,q\leq\infty$, $1\leq p \leq\infty$, $\eta \in (0,1)$ and $\alpha_0\neq \alpha_1 \in \bR$.
        \end{enumerate}  
\end{prop}

The following result is the so-called reduction theorem of the quantum Besov space, which can be deduced by the properties of $\{\phi_k\}_{k\in\bN}$ and \eqref{Young convolution}.
\begin{prop}\label{Besov regularity}
    Let $s\in\bR$ and $1\le p,q\le\fz$. If $u\in B_{p,q}^s(\rd),$ then for every $\alpha\in \bN^d,$ we have $\partial^{\alpha}u\in B_{p,q}^{s-|\alpha|}(\rd).$
\end{prop}
\subsubsection{Difference characterization}
Below we will show that the quantum Besov spaces have equivalent difference characterizations. To prove our results, we first introduce the definitions of difference operators and modulus of smoothness on quantum Euclidean spaces. 
\begin{definition}\label{diff op}
    For $u\in \tempert$ and $h\in \bR^d,$ the $m$th difference operator is defined by
    \begin{align*}
       \Delta_h^m:u\mapsto \sum_{k=0}^m\binom{m}{k}\,(-1)^{m-k}T_{kh}u. 
    \end{align*}
    For $u\in L_p(\rd),$ the $m$th $L_p$-modulus of smoothness is defined by
    $$\omega_p^m(t,u)=\sup_{\left | h\right |\le t}\left \| \Delta_h^mu \right \|_p.$$
\end{definition}
Since the translation operator $T_{kh}$ is unitary, for $1\le p\le\infty$ we have \begin{align}\label{diffbound}
    \lf\|\Delta_h^m\r\|_{L_p(\rd)\to L_p(\rd)}\le 2^m.
    \end{align}

Now we state one of the main results in this subsection, which can also be seen in \cite{Hong}.
\begin{theorem}\label{equiv ch}
    If $1\le p,q\le\fz,s>0$ and $u\in B_{p,q}^s(\rd),$ then for $m,N\in \mathbb{N}$ such that $m+N>s$ and $0\le N<s,$ we have \begin{align}\label{new eq norm}
    \left \| u \right \|_{B_{p,q}^s}\sim_{s,m,N,p,q} \left \| u \right \|_p+\sum_{i=1}^d\lf(\int_0^{\infty}\lf(t^{-s+N}\omega_p^m(t,\partial_i^Nu)\r)^q \frac{dt}{t}\r)^{1/q},
\end{align}
with the usual modification when $q=\infty,$ where $\omega_p^m$ and $\Delta_h^m$ are given in Definition \ref{diff op}. In particular, for $s\in \bR_+\setminus \bN$ we have \begin{align}\label{s notin N}
    \left \| u \right \|_{B_{p,q}^s}\sim_{s,p,q}  \left \| u \right \|_p+\sum_{j=1}^d\lf(\int_0^{\infty}\lf(t^{-\left \{ s \right \} }\omega_p^1(t,\partial_i^{[s]}u)\r)^q \frac{dt}{t}\r)^{1/q}.
\end{align}
\end{theorem}
For the classical Besov space, 
this result is well-known (cf. e.g.  \cite{Triebel,Wang2011}). For the Besov space on quantum tori, the corresponding conclusion is proved in \cite{XXY2018}. Our approach, while motivated by these earlier works, differs slightly from them due to the presence of noncommutativity and non-compactness. 

Before proving Theorem \ref{equiv ch}, we first show the following estimate of difference operators:
\begin{lemma}\label{WBX lemma}
    If $1\le p,q\le\fz,s>0$ and $u\in L_p(\rd),$ then for $k\in\bN$ we have 
    \begin{align*}
        \left \| \Delta_h^m \triangle_k u\right \|_p \lesssim_{m,p} \min(1,\left | h\right |^m2^{km})\left \|  \triangle_k u\right \|_p.
    \end{align*}
\end{lemma}
\begin{proof}
Consider the function $\rho_{h}^m:\xi\mapsto(e^{\ri\lag h,\xi\rag}-1)^m.$
    Assume that $u=\lambda_\theta(f)$ for some $f\in \cS'(\bR^d).$ For $k\in \bN,$ let $\{\phi_k\}_{k\in \bN}$ be the non-homogeneous Littlewood-Paley decomposition and set $\phi_{-1}:=0,\triangle_{-1}:=0.$  By the equality \begin{align*}
\phi_k=\sum_{\ell=-1}^1\phi_{k+l}\phi_k,
    \end{align*} \eqref{uniformLP} and \eqref{diffbound}, we have
    \begin{align*}
       \left \| \Delta_h^m \triangle_k u\right \|_p 
        & \le \sum_{\ell=-1}^1\left \| \lambda_\theta(\rho_{h}^m\phi_{k+l}\phi_kf)\right \|_p\\
        & =\sum_{\ell=-1}^1\left \| \Delta_h^m \triangle_{k+l}\triangle_ku\right \|_p\\
        &\lesssim_m \left \| \triangle_k u\right \|_p,
    \end{align*}

On the other hand, we have \begin{align*}
        \left | \partial^{\alpha}\rho_{h}^m(\xi) \right |&\lesssim_{m,\alpha}\left | h \right |^{\left | \alpha  \right |}\lf(e^{\ri\lag h,\xi\rag}-1\r)^{m-\left | \alpha  \right |}\\
        &\le \left | h \right |^{\left | \alpha  \right |}\left | \lag h,\xi\rag \right |^{m-\left | \alpha  \right |}\\
        &\le \left | h \right |^{m}\left | \xi \right |^{m-\left | \alpha  \right |}.
    \end{align*} 
    Then by \cite[Lemma 8.2.4]{ABHN}, we have $$\left \| \mathcal{F}^{-1} \rho_{h}^m\phi_k \right \|_{1}\lesssim_m 2^{km}\left | h \right |^{m}.$$ 
Hence \eqref{Young convolution} yields \begin{align*}
        \left \| \Delta_h^m \triangle_k u\right \|_p \lesssim_m \left | h\right |^m2^{km}\left \|  \triangle_k u\right \|_p.
    \end{align*} 
\end{proof}

\begin{proof}[Proof of Theorem \ref{equiv ch}]
We may assume that $q<\infty$ since $q=\infty$ is similar. Note that $\omega_p^m(\cdot,\partial_i^Nu)$ is an increasing function on $\bR_+,$ hence it suffices to show 
    $$\left \| u \right \|_{B_{p,q}^s}\sim_{s,m,N,p,q} \left \| u \right \|_p+\sum_{i=1}^d\lf(\sum_{j=-\infty}^{\infty}\lf(2^{j(s-N)}\omega_p^m(2^{-j},\partial_i^Nu)\r)^q \r)^{1/q}.$$

    By Lemma \ref{WBX lemma}, we obtain
     \begin{align*}
        \left \| \Delta_h^m \triangle_k \partial_i^Nu\right \|_p &\lesssim_{m,p} \min(1,\left | h\right |^m2^{km})\left \|  \triangle_k \partial_i^Nu\right \|_p\\
        &\lesssim_N \min(1,\left | h\right |^m2^{km})2^{kN}\left \|  \triangle_k u\right \|_p,
    \end{align*}
    which implies that 
    $$2^{j(s-N)}\omega_p^m(2^{-j},\partial_i^Nu)\lesssim_{m,N,p} 2^{(j-k)(s-N)}\min(1,2^{(k-j)m})2^{sk}\left \|  \triangle_k u\right \|_p.$$ Take the $\ell_q$ norm on both sides and by \eqref{Young convolution} we have
    \begin{align*}
        \sum_{i=1}^d\lf(\sum_{j=-\infty}^{\infty}\lf(2^{j(s-N)}\omega_p^m(2^{-j},\partial_i^Nu)\r)^q \r)^{1/q}\lesssim_{s,m,N,p,q} \left \| u \right \|_{B_{p,q}^s}.
    \end{align*}
    Obviously $\left \| u \right \|_p\lesssim_{s,p,q} \left \| u \right \|_{B_{p,q}^s},$ thus we obtain
    $$\left \| u \right \|_p+\sum_{i=1}^d\lf(\sum_{j=-\infty}^{\infty}\lf(2^{j(s-N)}\omega_p^m(2^{-j},\partial_i^Nu)\r)^q \r)^{1/q}\lesssim_{s,m,N,p,q} \left \| u \right \|_{B_{p,q}^s}.$$

    For the converse direction, consider the function $$\rho_{jk}^m:\xi\mapsto\sum_{\ell=0}^m\binom{m}{\ell}(-1)^{m-\ell}e^{\ri 2^{-k}\xi_j\ell}=\lf(e^{\ri2^{-k}\xi_j}-1\r)^m.$$ Now it suffices to show \begin{align}\label{lemma 1.1}
        \left \|  \triangle_k u\right \|_p \lesssim_{N,m,p} 2^{-kN}\sum_{j=1}^d\left \| \mathcal{F}^{-1}(\rho_{jk}^m)\ast \partial_i^Nu \right \|_p, \ \ \ k \in\bN^+.
    \end{align}
    Indeed, if \eqref{lemma 1.1} holds, by the inequalities $$\left \| \mathcal{F}^{-1}(\rho_{jk}^m)\ast \partial_i^Nu \right \|_p\le \omega_p^m(2^{-k},\partial_i^Nu)$$ and $\|\triangle_0u\|_p\lesssim_p \|u\|_p,$ we conclude the proof.

 To prove \eqref{lemma 1.1}, consider a sequence of smooth functions $\left \{ \chi_j  \right \}_{j=1}^d$ that satisfies 
 \begin{align}
    \label{chi condition}    \sum_{j=1}^d \chi_j(\xi)=1,\forall  \xi \in \left \{\xi:\frac 12\le \left | \xi\right |\le 2 \right \};
\end{align}
and
\begin{align}
    \label{chi condition 2}   \supp \chi_j \subset \left \{\xi:\left | \xi_j\right |\ge \frac{1}{3\sqrt{d}} \right \}.
\end{align}
    The existence of such family is shown in \cite[Lemma 1.1]{Wang2011}. Then by the fact $\supp \phi_k \subset \left \{\xi:2^{k-1}\le \left | \xi\right |\le 2^{k+1} \right \}$ and \eqref{chi condition}, for any $k\in\bN^+$ we have
    \begin{align*}
       \left \|  \triangle_k u\right \|_p \lesssim_{m,p} \sum_{j=1}^d\left \| \mathcal{F}^{-1}\lf(\rho_{jk}^{-m}\phi_k\chi_j(2^{-k}\cdot)\xi_j^{-N}\r)\ast \mathcal{F}^{-1}(\rho_{jk}^m)\ast \partial_i^Nu \right \|_p. 
    \end{align*} 
    We just need to show $$\lf\|\,\rho_{jk}^{-m}\phi_k\chi_j(2^{-k}\cdot)\xi_j^{-N}\r\|_{M_p(\rd)}\lesssim_{m,N,p} 2^{-kN}$$ for $k\in\bN^+.$ Moreover, by Proposition \ref{rescaling property}, it suffices to show $$\rho_{j0}^{-m}\varphi_0\chi_j\xi_j^{-N}\in M_p(\mathbb{R}_{2^{-2k}\theta}^d)$$ and $\lf\|\,\rho_{j0}^{-m}\varphi_0\chi_j\xi_j^{-N}\r\|_{M_p(\mathbb{R}_{2^{-2k}\theta}^d)}$ is independent of $k.$

    Indeed, by Young's convolution inequality and Bernstein multiplier theorem (cf. e.g. \cite[Proposition 1.11]{Wang2011}), for a fixed integer $L>\frac d2,$
    \begin{align*}
       \left \| \,\rho_{j0}^{-m}\xi_j^{-N}\varphi_0\chi_j \right \|_{M_p(\mathbb{R}_{2^{-2k}\theta}^d)}& \le \left \|\mathcal{F}^{-1} \lf(\rho_{j0}^{-m}\varphi_0\chi_j\xi_j^{-N}\r) \right \|_1\\
       &\lesssim \left \| \,\rho_{j0}^{-m}\varphi_0\chi_j\xi_j^{-N} \right \|_2^{1-\frac d{2L}} \sum_{i=1}^d\left \| \partial_i^L\lf(\rho_{j0}^{-m}\varphi_0\chi_j\xi_j^{-N}\r) \right \|_2^{\frac d{2L}}.
    \end{align*}
    Therefore, set $$C:=\left \|\, \rho_{j0}^{-m}\phi_0\chi_j\xi_j^{-N} \right \|_2^{1-\frac d{2L}} \sum_{i=1}^d\left \| \partial_i^L\lf(\rho_{j0}^{-m}\phi_0\chi_j\xi_j^{-N}\r) \right \|_2^{\frac d{2L}},$$ then the regularity of $\phi_0$ and \eqref{chi condition 2} immediately yields that $C$ is finite and depends only on $m,N.$

    When $s\notin \bN,$ take $N=[s]$ and $m=1,$ then we obtain \eqref{s notin N}.
\end{proof}

It is worth noting that in the recent paper of Lafleche \cite{l2024}, he also gave a definition for the difference norm of quantum Besov spaces when $\theta=\hbar\begin{pmatrix}
 0 & I_n\\
 -I_n & 0
\end{pmatrix},$ which can be generalized to any real antisymmetric $d\times d$ matrix:
\begin{align*}
    \|u\|_{\mathrm {BL}}:=\left \| u \right \|_p+\sum_{i=1}^d\lf(\int_{\bR^d}\lf(|\varrho|^{-s+N}\lf\|\Delta_\varrho^m\partial_i^Nu\r\|_p\r)^q \frac{d\varrho}{|\varrho|^d}\r)^{1/q}.
\end{align*}
In the classical case $\theta=0,$ Lafleche's definition is equivalent to our definition of difference norm (cf. e.g. \cite{Triebel}). It was conjectured in McDonald's paper \cite[Remark 3.16]{McNLE} that this equivalence holds for any real antisymmetric $d\times d$ matrix $\theta.$ We confirm this conjecture in the following theorem, which is the other main result in this subsection.
\begin{theorem}\label{Laf}
    If $1\le p,q\le\fz,s>0$ and $u\in B_{p,q}^s(\rd),$ then for $m,N\in \mathbb{N}$ such that $m+N>s$ and $0\le N<s,$ we have
    \begin{align*}
        \left \| u \right \|_{\mathrm{BL}}\sim_{s,m,N,p,q} \left \| u \right \|_p+\sum_{i=1}^d\lf(\int_0^{\infty}\lf(t^{-s+N}\omega_p^m(t,\partial_i^Nu)\r)^q \frac{dt}{t}\r)^{1/q}.
    \end{align*}
\end{theorem}

Before proving Theorem \ref{Laf}, let us first prove the following equivalence.
\begin{lemma}\label{Laf1}
    If $1\le p,q\le\fz,s>0$ and $u\in B_{p,q}^s(\rd),$ then for $m,N\in \mathbb{N}$ with $m+N>s$ and $0\le N<s,$ we have 
    $$\int_0^{\infty}\lf(t^{-s+N}\omega_p^m(t,\partial_i^Nu)\r)^q \frac{dt}{t}=\frac{\Gamma(\frac d2)}{2\pi^{\frac d2}} \int_{\bR^d}\lf(|\varrho|^{-s+N}\sup_{|h|\le |\varrho|}\lf\|\Delta_h^m\partial_i^Nu\r\|_p\r)^q \frac{d\varrho}{|\varrho|^d}, \ \ \ i=1,\dots,d.$$
\end{lemma}
\begin{proof}
    Fix $i=1,\dots,d$ and set
    \begin{align*}
        I_i:=\int_0^{\infty}\lf(t^{-s+N}\omega_p^m(t,\partial_i^Nu)\r)^q \frac{dt}{t}.
    \end{align*}
  Let $\mathbb S^{d-1}$ denote the unit sphere of $\bR^d.$ For every $\theta\in \mathbb S^{d-1},$ we have
  \begin{align*}
        I_i=\int_0^{\infty}\lf(t^{-s+N}\sup_{|h|\le |t\theta|}\lf\|\Delta_h^m\partial_i^Nu\r\|_p\r)^q \frac{dt}{t}.
    \end{align*}
Hence we can integrate $I_i$ over $\mathbb S^{d-1}$ and apply the spherical coordinate transformation to deduce that
\begin{align*}   
I_i&=\frac{\Gamma(\frac d2)}{2\pi^{\frac d2}} \int_0^{\infty}\int_{\mathbb S^{d-1}}\lf(t^{-s+N}\sup_{|h|\le |t\theta|}\lf\|\Delta_h^m\partial_i^Nu\r\|_p\r)^q \frac{d\theta\, dt}{t}\\
    &=\frac{\Gamma(\frac d2)}{2\pi^{\frac d2}} \int_{\bR^d}\lf(|\varrho|^{-s+N}\sup_{|h|\le |\varrho|}\lf\|\Delta_h^m\partial_i^Nu\r\|_p\r)^q \frac{d\varrho}{|\varrho|^d},
\end{align*}
which implies the conclusion.
\end{proof}

\begin{proof}[Proof of Theorem \ref{Laf}]
    Utilizing the spherical coordinate transformation, we immediately have
    \begin{align*}
        \left \| u \right \|_{\mathrm{BL}}\lesssim_{s,m,N,p,q} \left \| u \right \|_p+\sum_{i=1}^d\lf(\int_0^{\infty}\lf(t^{-s+N}\omega_p^m(t,\partial_i^Nu)\r)^q \frac{dt}{t}\r)^{1/q}.
    \end{align*}
    For the converse direction, applying Lemma \ref{Laf1}, we have
    \begin{align*}
        I_i^{1/q}&\le \lf(\frac{\Gamma(\frac d2)}{2\pi^{\frac d2}} \int_{\bR^d}\lf(|\varrho|^{-s+N}\sup_{|h|\le \frac{|\varrho|}2}\lf\|\Delta_h^m\partial_i^Nu\r\|_p\r)^q \frac{d\varrho}{|\varrho|^d}\r)^{1/q}+\lf(\frac{\Gamma(\frac d2)}{2\pi^{\frac d2}} \int_{\bR^d}\lf(|\varrho|^{-s+N}\sup_{\frac{|\varrho|}2<|h|\le \varrho}\lf\|\Delta_h^m\partial_i^Nu\r\|_p\r)^q \frac{d\varrho}{|\varrho|^d}\r)^{1/q}\\
        &= \lf(\int_{0}^\infty\lf(t^{-s+N}\sup_{|h|\le \frac{t}2}\lf\|\Delta_h^m\partial_i^Nu\r\|_p\r)^q \frac{dt}{t}\r)^{1/q}+\lf(\frac{\Gamma(\frac d2)}{2\pi^{\frac d2}} \int_{\bR^d}\lf(|\varrho|^{-s+N}\sup_{\frac{|\varrho|}2<|h|\le \varrho}\lf\|\Delta_h^m\partial_i^Nu\r\|_p\r)^q \frac{d\varrho}{|\varrho|^d}\r)^{1/q}\\
        &\le 2^{-s+N}I_i^{1/q}+\lf(\frac{\Gamma(\frac d2)}{2\pi^{\frac d2}} \int_{\bR^d}\lf(|\varrho|^{-s+N}\sup_{\frac{|\varrho|}2<|h|\le |\varrho|}\lf\|\Delta_h^m\partial_i^Nu\r\|_p\r)^q \frac{d\varrho}{|\varrho|^d}\r)^{1/q}, 
    \end{align*}
which implies that
\begin{align*}
    I_i^{1/q}\lesssim_{s,m,N,p,q}  \lf(\int_{\bR^d}\lf(|\varrho|^{-s+N}\sup_{\frac{|\varrho|}2<|h|\le |\varrho|}\lf\|\Delta_h^m\partial_i^Nu\r\|_p\r)^q \frac{d\varrho}{|\varrho|^d}\r)^{1/q}
\end{align*}
since $2^{-s+N}<1.$
Set
\begin{align*}
    J_i=\int_{\bR^d}\lf(|\varrho|^{-s+N}\sup_{\frac{|\varrho|}2<|h|\le |\varrho|}\lf\|\Delta_h^m\partial_i^Nu\r\|_p\r)^q \frac{d\varrho}{|\varrho|^d}.
\end{align*}
 Now it remains to control the term $J_i.$ For any $h\in \bR^d$ such that $\frac{|\varrho|}2<|h|\le |\varrho|$ and any $\frac{|\varrho|}8<|h_0|\le \frac{|\varrho|}4,$ set $h_1:=h-h_0$ and we then have $$(e^{\ri\lag h,\xi\rag}-1)^{2m}=\sum_{m_1=0}^{2m}\binom{2m}{m_1}e^{\ri \lag m_1h_0,\xi\rag}(e^{\ri\lag h_1,\xi\rag}-1)^{m_1}(e^{\ri\lag h_0,\xi\rag}-1)^{2m-m_1}.$$ Since $\max(m_1,2m-m_1)\ge m,$ \eqref{diffbound} and the identity $$\Delta_h^{\max(m_1,2m-m_1)}=\Delta_h^{\max(m_1,2m-m_1)-m}\Delta_h^{m}$$ imply that
 \begin{align*}
\lf\|\Delta_h^{2m}\partial_i^Nu\r\|_p\lesssim_m \lf\|\Delta_{h_0}^{m}\partial_i^Nu\r\|_p+\lf\|\Delta_{h_1}^{m}\partial_i^Nu\r\|_p.
 \end{align*}
 Now we integrate $\lf\|\Delta_h^{2m}\partial_i^Nu\r\|_p^q$ over $K_{\varrho}:=\{h_0\in\bR^d:\frac{|\varrho|}8<|h_0|\le \frac 54|\varrho|\}$ and deduce that
 \begin{align*}  \sup_{\frac{|\varrho|}2<|h|\le |\varrho|}\lf\|\Delta_h^{2m}\partial_i^Nu\r\|_p^q\lesssim_{m,p} \int_{K_\varrho} \lf\|\Delta_{h_0}^{m}\partial_i^Nu\r\|_p^q \frac{dh_0}{|h_0|^d}.
 \end{align*}
Then the Fubini theorem yields the following estimate of $J_i:$
 \begin{align*}
     J_i&\lesssim_{s,m,N,p,q} \int_{\bR^d}\lf(|\varrho|^{-s+N} \int_{K_\varrho} \lf\|\Delta_{h_0}^{m}\partial_i^Nu\r\|_p^q \frac{dh_0}{|h_0|^d}\r)\frac{d\varrho}{|\varrho|^d}\\
     &\sim \int_{\bR^d}\lf(|h_0|^{-s+N} \int_{\tilde K_{h_0}} \lf\|\Delta_{h_0}^{m}\partial_i^Nu\r\|_p^q \frac{d\varrho}{|\varrho|^d}\r)\frac{dh_0}{|h_0|^d}\\
     &\sim \int_{\bR^d}\lf(|h_0|^{-s+N}\lf\|\Delta_{h_0}^m\partial_i^Nu\r\|_p\r)^q \frac{dh_0}{|h_0|^d},
 \end{align*}
 where $\tilde K_{h_0}:=\{\varrho\in\bR^d:\frac 45|h_0|\le |\varrho|<8|h_0|\}.$ Hence we conclude the proof.
\end{proof}

\begin{remark}
    The above proof for noncommutative case is similar to the classical proof and should be known to the experts.
\end{remark}

\subsection{Heat semigroup on quantum Besov spaces}
We now introduce the heat semigroup on $\rd$ and its behavior on the quantum Besov spaces. Let $e^{t\Delta}$ denote the Fourier multiplier with symbol $e^{-t|\cdot|^2},$
that is 
\begin{align*}
    e^{t\Delta}\lt(f)=\lt(e^{-t|\cdot|^2}f), \ \ \ \forall f\in \cS(\bR^d).
\end{align*}
We call $\{e^{t\Delta}\}_{t\ge 0}$ the heat semigroup on the quantum Euclidean spaces, whose generator is the quantum Laplacian $\Delta$. Note that $\{e^{t\Delta}\}_{t\ge 0}$ is a contractive $C_0$ semigroup on $L_2(\rd)$ due to Lemma \ref{Hauss}. By a standard argument, it extends to a contractive $C_0$ semigroup on $L_p(\rd),1\le p\le\infty.$ The details can be seen in \cite{CHWW} for example. 

The following proposition states the properties of heat semigroup on the quantum Besov spaces.
\begin{prop}\label{heat group}
    Let $1\le p,q\le\fz$ and $s\in \bR,$
then $\{e^{t\Delta}\}_{t\ge 0}$ satisfies the following properties:
\begin{enumerate}
    \item[\rm{(i)}]$\{e^{t\Delta}\}_{t\ge 0}$ is a $C_0$ contractive semigroup on $B_{p,q}^s(\rd),1\le p<\infty,1\le q\le\infty.$

\item[\rm{(ii)}]$\{e^{t\Delta}\}_{t\ge 0}$ is an analytic semigroup on $B_{p,q}^s(\rd),1< p<\infty,1\le q\le\infty.$ 

\item[\rm{(iii)}] For $1\le q\le \infty,s,r\in \bR$ and $u\in B_{p,q}^s(\rd),$ \begin{align}\label{heatdiff}
    \|e^{t\Delta}u\|_{B_{p,q}^r}\lesssim (1+t^{\frac{s-r}{2}})\|u\|_{B_{p,q}^s}.
\end{align}
\end{enumerate}
\end{prop}
\begin{proof}
     \rm{(i)} and \rm{(iii)} follow from \eqref{Young convolution}, hence we only need to show \rm{(ii)}. It suffices to show that $\Delta$ is sectorial on $B_{p,q}^s(\rd):$ for every $u\in B_{p,q}^s(\rd)$ and $\lambda\in\bC$ with $\Re(\lambda)>0,$ $$\|(\lambda-\Delta)^{-1}u\|_{B_{p,q}^s}\lesssim |\lambda|^{-1}\|u\|_{B_{p,q}^s}.$$ Indeed, \cite[Lemma 3.4]{CHWW} has shown that $e^{-t\Delta}$ is an analytic semigroup on $L_p(\rd),1<p<\infty.$ Combined with \cite[Theorem G.5.2]{HNVW} we deduce that $\Delta$ is sectorial on $L_p(\rd).$ Hence it follows that \begin{align*}
         \|(\lambda-\Delta)^{-1}\triangle_j u\|_p\lesssim |\lambda|^{-1}\|\triangle_ju\|_p,
     \end{align*} which implies the conclusion by summing $j\in\bN$.
\end{proof}

\section{Multiple operator integral} \label{s3}
In this section, we will introduce the related theories of multiple operator integrals. We first introduce two equivalent definitions of multiple operator integrals. Then we will present the boundedness, continuity and perturbation formula of multiple operator integrals, which will play crucial roles in establishing our main results.

Throughout this section, let $\cM$ be a semifinite von Neumann subalgebra of $\cB(H)$ with a n.s.f trace $\tau$ and $\cM_{sa}$ denote the self-adjoint subspace of $\cM,$ where $H$ is a separable Hilbert space. A closed densely defined operator $A$ on $H$ is said to be affiliated with $\cM,$ if it commutes with all unitary operators in $\cM.$ We write $A\eta \cM$ to indicate that $A$ is affiliated with $\cM.$ Moreover, if $A$ is also self-adjoint, we use the notation $A\eta\cM_{sa}$. We also use $\cL_p(\cM)$ to denote the Schatten-von Neumann ideal $L_p(\cM)\cap L_\infty(\cM)$ equipped with the norm $\|\cdot\|_{\cL_p}=\|\cdot\|_p+\|\cdot\|_\fz.$ For $0< p_1,\dots,p_n,p\le\infty,$ we will say $(p_1,\dots,p_n;p)$ is a {\it H{\"o}lder tuple} if $\frac 1p=\sum_{j=1}^n\frac{1}{p_j}.$

We now recall the definitions of multiple operator integrals. There are several different ways to define the multiple operator integrals (cf. e.g. \cite{ACDS2009, Peller2006, PSS2013}), and we first introduce the formulation given in \cite[Definition 3.1]{PSS2013}. 

\begin{definition}\label{def of Moi}
    Let $A_i\eta \cM_{sa},i=1,\dots,n,E_i$ be the spectral set of $A_i$ and $(p_1,\dots,p_n;p)$ be a H{\"o}lder tuple. Let $E_{i,m}^l:=E_i[\frac{l}{m},\frac{l+1}{m})$ denote the spectral projection with respect to the interval $[\frac{l}{m},\frac{l+1}{m}),$ where $m \in \mathbb{N}^+,l \in \mathbb{Z}.$ Assume that $X_i\in L_{p_i}(\cM),$ where $1\le p_i,p\le \infty$ and $\phi:\bR^{n+1}\to \bC$ be a bounded Borel function. Suppose for $m\in \bN^+,$ the series
    \begin{align*}
        S_{\phi,m}^{A_0,\dots,A_n}(X_1,\dots,X_n):=\sum_{\ell_0,\dots,l_n\in \bZ} \phi(\frac{l_0}{m},\dots,\frac{l_n}{m})E_{0,m}^{l_0}X_1E_{1,m}^{l_1}X_2\cdots X_nE_{n,m}^{l_n}
    \end{align*}
    converges in $L_p(\cM)$ norm and $$S_{\phi,m}^{A_0,\dots,A_n}:\prod_{j=1}^n L_{p_j}(\cM)\mapsto L_p(\cM),(X_1,\dots,X_n)\mapsto S_{\phi,m}(X_1,\dots,X_n)$$ is bounded. If $\{S_{\phi,m}\}_{m}$ converges in $\sot$ to a multilinear operator $T_{\phi}^{A_0,\dots,A_n},$ then we say that $T_{\phi}^{A_0,\dots,A_n}$ is a {\it multiple operator integral} associated with $A_0,\dots,A_n$ and $\phi$. 
\end{definition}


\begin{remark}
   Once the $\sot$-limit $T_{\phi}^{A_0,\dots,A_n}$ exists, it is necessarily a bounded multilinear operator due to the Banach-Steinhaus theorem. 
\end{remark}

We will also introduce another approach to define the multiple operator integral that was given in \cite{Peller2006} (also see \cite[Definition 4.1]{ACDS2009}), which is more intuitive.
To begin with, we first introduce a special function class. Let $\fU_n$ be the collection of bounded Borel functions $\phi:\bR^{n+1}\to \bC$ admitting the following representation:
\begin{align}\label{BG de}
\phi(t_0,\dots,t_n)=\int_{\Omega}\alpha_0(t_0,\omega)\cdots \alpha_n(t_n,\omega)\,d\mu(\omega), \ \ \ t_0,\dots,t_n\in \mathbb{R},
\end{align}
where $(\Omega,\mu)$ is a $\sigma$-finite measure space and for each $\omega\in \Omega,$ $\alpha_i(\cdot,\omega):\bR\to \bC,i=0,\dots,n$ are bounded Borel functions such that \begin{align*}
    \int_{\Omega}\prod_{i=0}^n\sup_{t_i \in \bR}\left | \alpha_i(t_i,\omega) \right |d\left | \mu \right | (\omega)<\infty.
\end{align*}
The space $\fU_n$ is often referred to as the {\it projective tensor product space}, which is a Banach algebra equipped with the norm
\begin{align*}
    \left \| \phi \right \|_{\mathfrak{U}_n}:= \inf \int_{\Omega}\prod_{i=0}^n\sup_{t_i \in \bR}\left | \alpha_i(t_i,\omega) \right |d\left | \mu \right | (\omega),
\end{align*}
where the infimum is taken over all the representations of form \eqref{BG de}. 
We also define $\fC_n$ to be the subspace of $\fU_n$ consisting of those functions $\phi$ admitting representations of form \eqref{BG de}, in which $\alpha_i(\cdot,\omega):\bR\to \bC,i=0,\dots,n$ satisfy the following conditions: $\alpha_i(\cdot,\omega)$ are continuous and there exists a sequence of measurable subsets $\{\Omega_k\}_{k\in\bN^+}$ increasing to $\Omega$ such that the families $\{\alpha_i(\cdot,s)\}_{s\in \Omega_k}$ are uniformly bounded and uniformly equicontinuous. The norm of $\fC_n$ is defined similarly to that on $\fU_n,$ while the infimum being taken over representations meeting the additional conditions above. Note that  $\fC_n$ is a closed subalgebra of $\fU_n$ (see \cite[Lemma 4.6]{PS2004}).

We now introduce the second definition of multiple operator integrals. Suppose $\phi \in \fU_n$ admitting the representation \eqref{BG de} and $X_j\in L_{p_j}(\cM),$ where $(p_1,\dots,p_n;p)$ is a H{\"o}lder tuple with $1\le p_j,p\le \infty,$ then the multiple operator integral associated with $A_0,\dots,A_n$ and $\phi$ is defined as the following multilinear operator from $\prod_{j=1}^nL_{p_j}(\cM)$ to $L_{p}(\cM):$ \begin{align}\label{eq Moi def}
  \tilde T^{A_0,\dots,A_n}_{\phi}(X_1,\dots,X_n):=\int_{\Omega}\alpha_0(A_0,\omega)X_1\alpha_1(A_1,\omega)\cdots X_n\alpha_n(A_n,\omega)\,d\mu(\omega).
\end{align} 
By the triangle inequality and H{\"o}lder inequality, it is straightforward to see that $\tilde T_\phi^{A_0,\dots,A_n}$ is bounded with norm at most $\|\phi\|_{\fU_n}.$ 

Note that the latter definition requires more regularity for $\phi$ than the former one. This naturally raises the question that whether the two definitions are equivalent. In general, the answer is negative, even in situations where both expressions are well defined. However, when $\phi \in \fC_n,$ it was shown in \cite[Lemma 3.5]{PSS2013} that they indeed coincide.    

\subsection{Boundedness of multiple operator integrals}
In this subsection, we investigate the operator norm of multiple operator integrals. When $\phi\in \fC_n,$ the identity $T_\phi^{A_0,\dots,A_n}=\tilde T_\phi^{A_0,\dots,A_n}$ already implies that $\|T_\phi^{A_0,\dots,A_n}\|\le \|\phi\|_{\fC_n}.$ For a general bounded function $\phi,$ however, a complete characterization of the norm is not yet available. In this work, we focus on the case when $\phi=F^{[n]}$ exactly for some $F\in C^n(\bR).$ The definition of $n$th-order divided difference $F^{[n]}$ is is defined inductively: $F^{[0]}=F$ and
\begin{align*}
    F^{[n]}(\lambda_0,\lambda_1,\tilde{\lambda })=\left\{\begin{matrix} 
  \frac{F^{[n-1]}(\lambda_0,\tilde{\lambda })-F^{[n-1]}(\lambda_1,\tilde{\lambda })}{\lambda_0-\lambda_1},\ \ \ \text{if} \ \lambda_0\ne \lambda_1 \\  
  \frac{\partial}{\partial\lambda_0}F^{[n-1]}(\lambda_0,\tilde{\lambda }),\ \ \ \ \ \ \ \text{if} \ \lambda_0=\lambda_1 
\end{matrix}\right., 
\end{align*}
where $\lambda_i\in \mathbb{R},i=0,\dots,n$ and $\tilde{\lambda }=(\lambda_2,\dots,\lambda_n).$

The following $L_p$ boundedness estimate for $T_{F^{[n]}},$  established by Potapov et al. in \cite{PSS2013}, will be frequently used throughout the paper.   
\begin{prop}\label{Lp boundedness}\cite[Lemma 3.5, Theorem 5.3]{PSS2013}
    Assume that $A_i\eta \cM_{sa},i=1,\dots,n,X_j\in L_{p_j}(\cM),j=1,\dots,n$ and $(p_1,\dots,p_n;p)$ is a H{\"o}lder tuple. Let $F\in C^n(\bR)$ with $F^{(n)}$ being bounded, then for $1<p_j,p<\infty,j=1,\dots,n$ we have 
    \begin{align}\label{Lp bound es}
    \left \|T^{A_0,\dots,A_n}_{F^{[n]}}(X_1,\dots,X_n)  \right \|_p\lesssim_{n,p_1,\dots,p_n} \left \|F^{(n)}\right \|_{\infty} \prod_{j=1}^{n}\left \| X_j \right \|_{p_j}.
\end{align}
If $F^{[n]}\in \fC_n$ additionally, then for $1\le p_j,p\le \infty$ we have
\begin{align}\label{Lp bound es 2}
    \left \|T^{A_0,\dots,A_n}_{F^{[n]}}(X_1,\dots,X_n)  \right \|_p\le \left \|F^{[n]}\right \|_{\fC_n} \prod_{j=1}^{n}\left \| X_j \right \|_{p_j}.
\end{align}
\end{prop}

\begin{remark}\label{depend}
    The estimates \eqref{Lp bound es} and \eqref{Lp bound es 2} are independent of the choice of $A_0,\dots,A_n.$ 
\end{remark}

In general, no both necessary and sufficient condition is known for characterizing those functions $F$ for which $F^{[n]}\in \fC_n.$ However, the following lemma shows that $F$ belonging in the modified Besov space will provide a sufficient condition.

\begin{lemma}\cite[Theorem 4]{PS2014}\label{Besov Cn}
   For $n\in \bN,$ define the the modified Besov space 
\begin{align*}
    \tilde{B}_{\infty,1}^n(\bR):=\{F\in C^{n}(\bR) :  \|F^{(n)}\|_{\infty}+\sum_{k\in \bZ}\|\cF^{-1}(\hat F\varphi_k)\|_{\infty}<\infty\}
\end{align*}
with norm $$\|F\|_{\tilde{B}_{\infty,1}^n}:=\|F^{(n)}\|_{\infty}+\sum_{k\in \bZ}\|\cF^{-1}(\hat F\varphi_k)\|_{\infty},$$ where $\{\varphi_k\}_{k\in \bZ}$ are the homogeneous Littlewood-Paley functions on $\bR.$ 
Then for any $F\in \tilde{B}_{\infty,1}^n(\bR),$ we have $F^{[n]}\in \fC_n.$ 
\end{lemma}

\begin{remark}
    Combining Lemma \ref{Besov Cn} with \eqref{Lp bound es 2}, we have 
\begin{align}\label{Lp bound es 3}
     \left \|T^{A_0,\dots,A_n}_{F^{[n]}}(X_1,\dots,X_n)  \right \|_p\lesssim_{n,p_1,\dots,p_n} \prod_{j=1}^{n}\left \| X_j \right \|_{p_j}.
\end{align}
\end{remark}

Define the Wiener space $$W_n(\bR):=\{F\in C^{n}(\bR) :  F\in L_\infty(\bR) \ , \ \cF(F^{(n)})\in L_1(\bR)\}$$ with norm $$\|F\|_{W_n(\bR)}:=\|F\|_\infty+\|\cF(F^{(n)})\|_1.$$ It is straightforward to verify that the following strict embeddings hold: $$C_c^{n+1}(\bR)\hookrightarrow W_n(\bR)\hookrightarrow \tilde{B}_{\infty,1}^n(\bR).$$ Hence for $F\in W_n(\bR),$ we also have $F^{[n]}\in \fC_n.$

\subsection{Continuity of multiple operator integrals}
We have introduced the continuity of multiple operator integrals with respect to the $L_p(\cM)$ norm (equivalent to the boundedness) in Proposition \ref{Lp boundedness}. In this subsection, we will concern with their continuity with respect to the $\sot$ on $\cM.$ Moreover, viewing $$(A_0,\dots,A_n)\mapsto T_{\phi}^{A_0,\dots,A_n}(X_1,\dots,X_n)$$ as a nonlinear function, it is natural to investigate its continuity properties as well.

Now we state several relevant results. A sequence of self-adjoint operators $\{x_n\}$ is said to converge in the strong resolvent sense to $x$ if $$(z-x_n)^{-1}\sotto (z-x)^{-1}, \ \ \ \mathrm{for \, any}\ z\in \bC\setminus \bR.$$ We denote this by $x_n\srto x.$ It is straightforward to see that $x_n\sotto x$ implies $x_n\srto x.$

The following result was established in \cite[Proposition 4.9]{ACDS2009}.
\begin{prop}\label{SOT continuity}
   For $\phi\in \fC_n,A_i^{(k)}\eta \cM_{sa},X_j^{(k)}\in \cM,k\in \bN,$ suppose that $A_i^{(k)}\srto A_i$ and $X_j^{(k)}\sotto X_j$ when $k\to\infty$ for every $i=0,\dots,n,j=1,\dots,n,$ then we have $$T^{A_{0}^{(k)},\dots,A_{n}^{(k)}}_{\phi}(X_1^{(k)},\dots,X_n^{(k)})\sotto T^{A_0,\dots,A_n}_{\phi}(X_1,\dots,X_n).$$ 
\end{prop}

Moreover, we establish the following continuity result with respect to the $L_p(\cM)$ norm.
\begin{prop}\label{Lp continuity}
    For a H{\"o}lder tuple $(p_1,\dots,p_n;p),$ suppose that one of the following assumptions holds:
    \begin{enumerate}
        \item[\rm{(A)}]$F\in C^n(\bR)$ with compact support, $1<p_j,p<\infty,j=1,\dots,n.$
        \item[\rm{(B)}]$F\in \tilde B_{\fz,1}^n(\bR),$ $1\le p_j\le \infty,j=1,\dots,n,1\le p<\infty.$
    \end{enumerate}
    Then for $A_i^{(k)}\eta \cM_{sa},X_j^{(k)}\in L_{p_j}(\cM),k\in \bN$ such that $A_i^{(k)}\srto A_i$ and $X_j^{(k)}\overset{L_p}\to X_j$ when $k\to\infty$ for every $i=0,\dots,n,j=1,\dots,n,$  $$\lf\|T^{A_{0}^{(k)},\dots,A_{n}^{(k)}}_{F^{[n]}}(X_1^{(k)},\dots,X_n^{(k)})-T^{A_0,\dots,A_n}_{F^{[n]}}(X_1,\dots,X_n)\r\|_p\to 0.$$   
\end{prop}

\begin{remark}
   With more restrictions about the symbol $F$ and the operators $A_i^{(k)},A_i,$ this assertion was also proved in \cite[Proposition 4.19]{ST2019} and \cite[Lemma 2.5]{PS2024} on the Schatten classes. 
\end{remark}

To prove Proposition \ref{Lp continuity}, we first show the following lemma.
\begin{lemma}\label{SOT to Lp}
   If $U,V\in \cM,\{U_n\}_n,\{V_n\}_n\subset \cM$ such that $U_n\sotto U,V_n\sotto V$ when $n\to\infty,$ then for any $(x_n)_n\subset L_p(\cM),1\le p<\infty,$ such that $x_n\overset{L_p}\to x,$ we have $$\lim_{n\to \fz}\|U_nx_nV_n-UxV\|_p= 0.$$
\end{lemma}
\begin{proof}
  It suffices to consider the case $x_n\equiv x$ due to the inequality $$\|U_nx_nV_n-UxV\|_p\lesssim \|x_n-x\|_p+\|U_nxV_n-UxV\|_p.$$ We may assume that $U,V=0.$ When $p=2,$ $\lim_{n\to \infty}\|xV_n\|_2=0$ by the normality of $\tau.$ Thus by the uniform boundedness of $\{U_n\}_n,$ we also have $\lim_{n\to \infty}\|U_nxV_n\|_p=0.$ Set $$M:=\sup_{n}\,(\|U_n\|_\fz\,\|V_n\|_\fz)+\sup_n\|U_n\|_\fz.$$
    Now we first consider the case when $x$ has $\tau$-finite support. For $2<p<\infty,$ \begin{align*}
        \lim_{n\to\infty}\|U_nxV_n\|_p&\le \lim_{n\to\infty}\|U_nxV_n\|_2^{1-\frac 2p}\|U_nxV_n\|_{\infty}^{\frac 2p}\\
        &\le (M\|x\|_\fz)^{1-\frac 2p}\lim_{n\to\infty}\|U_nxV_n\|_2^{1-\frac 2p}\\
        &=0.
    \end{align*} 
    For $1\le p<2,$ suppose $\frac 1p=\frac 1q+\frac 12$ for some $q,$ then
    \begin{align*}
        \lim_{n\to\infty}\|U_nxV_n\|_p\le \lim_{n\to\infty}M\|l(x)\|_q\|xV_n\|_2=0,
    \end{align*}
 where $l(x)$ denotes the left support of $x.$ Hence we obtain the conclusion for operators $x$ with $\tau$-finite support. 
 
 Now we extend this result to every $x\in L_p(\cM).$     For every $\varepsilon>0,$ choose $y\in L_p(\cM)$ with $\tau$-finite support such that $\|y-x\|_p<\varepsilon,$ then 
    \begin{align*}
         \lim_{n\to\infty}\|U_nxV_n\|_p&\le  \lim_{n\to\infty}\|U_nyV_n\|_p+ \lim_{n\to\infty}\|U_n(x-y)V_n\|_p\\
         &\le M\varepsilon.
    \end{align*}
    Let $\varepsilon \searrow 0,$ then we have $\lim_{n\to\infty}\|U_nxV_n\|_p=0.$
\end{proof}

\begin{proof}[Proof of Proposition \ref{Lp continuity}]
    We only consider the case $n=1$ since $n\ge 2$ is similar. We first assume that \rm{(B)} holds.
    Due to Lemma \ref{Besov Cn}, $F^{[1]}$ admits the representation
$$F^{[1]}(t_0,t_1)=\int_{\Omega}\alpha_0(t_0,\omega)\alpha_1(t_1,\omega)\,d\mu(\omega), \ \ \ t_0,t_1\in \mathbb{R},$$ where $\alpha_0,\alpha_1$ are bounded, continuous and satisfy 
\begin{align}\label{Cn condition}
\int_{\Omega}\sup_{t_0\in \bR}|\alpha_0(t_0,\omega)|\sup_{t_1\in \bR}|\alpha_1(t_1,\omega)|\,d\mu(\omega)<\infty.
\end{align}
By \cite[Theorem \rm{VIII.20 (b)}]{RS1972}, $\alpha_i(A_i^{(k)},\omega)\sotto \alpha_i(A_i,\omega)$ for $i=0,1$. Then by Lemma \ref{SOT to Lp}, we have $$\lim_{k\to \infty}\|\alpha_0(A_0^{(k)},\omega)X_1^{(k)}\alpha_1(A_1^{(k)},\omega)- \alpha_0(A_0,\omega)X_1\alpha_1(A_1,\omega)\|_p=0.$$
Now using the Lebesgue dominated theorem for Bochner integrals \cite[Corollary \rm{III}.6.16]{DS}, we deduce from \eqref{eq Moi def} and \eqref{Cn condition} that
    $$\lim_{k\to \infty}\|T^{A_{0}^{(k)},A_{1}^{(k)}}_{F^{[1]}}(X_1^{(k)})-T^{A_0,A_1}_{F^{[1]}}(X_1)\|_p=0.$$ 

Now we assume that \rm{(A)} holds, in which case there exists $F_N\in C_c^\fz(\bR)\subset \tilde B_{\infty,1}^n(\bR)$ such that $\lim_{N\to\fz}\|F_N'-F'\|_\fz=0.$ Then from Proposition \ref{Lp boundedness}, we have
$$\lim_{N\to\infty}\sup_{k\in\bN}\|T_{F_N^{[1]}}^{A_0^{(k)},A_1^{(k)}}(X_1^{(k)})- T_{F^{[1]}}^{A_0^{(k)},A_1^{(k)}}(X_1^{(k)})\|_p\lesssim_p\sup_{k\in\bN}\|X_1^{(k)}\|_p\lim_{N\to\infty}\|F_N'-F'\|_\infty=0$$ and $$\lim_{N\to\infty}\|T_{F_N^{[1]}}^{A_0,A_1}(X_1)- T_{F^{[1]}}^{A_0,A_1}(X_1)\|_p\lesssim_p\|X_1\|_p\lim_{N\to\infty}\|F_N'-F'\|_\infty=0,$$ which yields the conclusion. 

\end{proof}


\subsection{Perturbation formula of multiple operator integrals}
In this subsection, we will introduce the perturbation formula, which plays a central role in the study of Fr{\'e}chet differentiability and nonlinear estimates of operator functions (cf. e.g \cite{ACDS2009,PSTZ2019}). For double operator integrals, the perturbation formula refers to the well-known L{\"o}wner identity:
$$F(X)-F(Y)=T_{F^{[1]}}^{X,Y}(X-Y),$$ which was established in \cite{Birman-Solomyak-III} for the Schatten class and \cite{ACDS2009,Peller1985} for the noncommutative $L_p$ space. It also has a $n$-variable generalization:
\begin{equation}\label{Perturbation formula}
     \begin{aligned}
    &\quad \   T_{F^{[n]}}^{A_1,\dots,A_{i-1},A,A_{i},\dots,A_{n}}(X_1,\dots,X_{n})
    -T_{F^{[n]}}^{A_1,\dots,A_{i-1},B,A_{i},\dots,A_{n}}(X_1,\dots,X_{n})\\
    &=T_{F^{[n+1]}}^{A_1,\dots,A_{i-1},A,B,A_{i},\dots,A_{n}}(X_1,\dots,X_{i-1},A-B,X_i,\dots,X_{n}).
\end{aligned}
\end{equation}
 This identity was proved in \cite{PSTZ2019} for the noncommutative $L_p$ space with $1<p<\infty,$ in \cite{LMS2020} for the Schatten class $\cS_p$ with $1\le p<\infty$ and in \cite{Peller2006} for $\cB(H)$. In the following proposition, we extend the result for the noncommutative $L_p$ space in \cite[Theorem 28]{PSTZ2019} in full generality.

\begin{prop}\label{prop: perturbation formula}
Suppose that one of the following assumptions holds:
\begin{enumerate}
        \item[\rm{(A)}]$F\in C^{n+1}(\bR)$ with $F^{(n)},F^{(n+1)}$ being bounded, $1<p_j<\infty,j=0,\dots,n$ with $0<\sum_{j=0}^n\frac{1}{p_j}< 1,$ 
        \item[\rm{(B)}] $F\in \tilde{B}_{\infty,1}^{n}(\bR)\cap \tilde{B}_{\infty,1}^{n+1}(\bR),$ $1\le p_j\le \infty,j=0,\dots,n$ with $0\le \sum_{j=0}^n\frac{1}{p_j}\le 1.$ 
    \end{enumerate}
    Then for $A_i\eta \cM_{sa},A,B\in L_{p_0}(\cM)$ self-adjoint and $X_i\in L_{p_j}(\cM),i=1,\dots,n,$ \eqref{Perturbation formula} holds.
\end{prop}

\begin{proof}
The first part follows exactly as the proof of \cite[Theorem 28]{PSTZ2019}. We omit the details. We therefore focus on the second case. We may assume that $i=1$ and other cases follow from the enumeration of variables.  

\textbf{Case 1.} If $p_0=\infty,$ without loss of generality, we may assume that there is an integer $0\le m\le n$ such that $p_j=\infty,j=0,\dots,m$ and $p_j<\infty, j=m+1,\dots,n.$ Suppose first that $X_1,\dots,X_n,A,B\in \cS_{\cm},$ then we obtain \eqref{Perturbation formula} through \cite[Theorem 28]{PSTZ2019} since $\cS_{\cm}\subset L_q(\cM)$ for every $n+1<q<\infty.$ Recall that $\cS_{\cm}$ is dense in $\cM$ with respect to the $\sot$. Hence Proposition \ref{SOT continuity} allows us to pass to the $\sot$-limit in \eqref{Perturbation formula}, leading to \eqref{Perturbation formula} for $X_1,\dots,X_m,A,B\in \cM,X_{m+1},\dots X_{n}\in \cS_{\cm}.$ 
 
 On the other hand, again recall that $\cS_{\cm}$ is dense in $L_{p_j}(\cM)$ for $1\le p_j<\infty,j=m+1,\dots,n.$ Take $X_j^{(k)}\in \cS_{\cm},k\in\bN$ such that $X_j^{(k)}\overset{L_{p_j}}\to X_j$ when $k\to\infty,$ then by Proposition \ref{Lp boundedness}, the left side of \eqref{Perturbation formula} $$T_{F^{[n]}}^{A,A_1,\dots,A_{n}}(X_1,\dots,X_{m},X_{m+1}^{(k)},\dots,X_n^{(k)})
    -T_{F^{[n]}}^{B,A_1,\dots,A_{n}}(X_1,\dots,X_{m},X_{m+1}^{(k)},\dots,X_n^{(k)})$$
    converges to $$T_{F^{[n]}}^{A,A_1,\dots,A_{n}}(X_1,\dots,X_{n})
    -T_{F^{[n]}}^{B,A_1,\dots,A_{n}}(X_1,\dots,X_{n})$$ in  $L_r(\cM)$ norm, where $0<\sum_{j=m+1}^n\frac{1}{p_j}=\frac 1r\le 1,$ and then converges in measure topology. Similarly, the right side of \eqref{Perturbation formula} $$T_{F^{[n+1]}}^{A,B,A_1,\dots,A_{n}}(A-B,X_1,\dots,X_{m},X_{m+1}^{(k)},\dots,X_n^{(k)})$$ converges to $$T_{F^{[n+1]}}^{A,B,A_1,\dots,A_{n}}(A-B,X_1,\dots,X_n)$$ in measure topology. Hence we conclude the proof for this case.

   \textbf{Case 2.} If $p_0<\infty,$ we first assume that $A,B\in \cL_{p_0}(\cM).$ This will be attributed to the previous case $p_0=\infty,$ leading to \eqref{Perturbation formula} for $A,B\in \cL_{p_0}(\cM)$ and $X_{j}\in L_{p_j}(\cM),j=1,\dots,n.$ For any $A,B\in L_{p_0}(\cM),$ let $A_n=AE_A[-n,n],B_n=BE_B[-n,n],$ then by Lemma \ref{SOT to Lp}, $A_n,B_n\in \cL_{p_0}(\cM)$ and $A_n\srto A,B_n\overset{L_{p_0}}\to B$ when $n\to\infty.$ Moreover, by \cite[Lemma 3.3]{CMSS2019}, $A_n\srto A,B_n\srto B.$  Hence by Proposition \ref{Lp continuity}, we conclude the proof for this case.
\end{proof}

When $n=0,$ \eqref{Perturbation formula} reduces to the L{\"o}wner identity $$F(X)-F(Y)=T_{F^{[1]}}^{X,Y}(X-Y).$$ Then by Proposition \ref{Lp boundedness}, we obtain the following straightforward corollary, which is well-known to the experts.
\begin{cor}\label{Lipes}
    Assume that $X,Y\in L_p(\cM).$ If $F\in \mathrm{Lip}(\bR),$ then for  $1<p<\infty$ we have \begin{align}\label{Lipes1}
        \|F(X)-F(Y)\|_p\lesssim_p \|F\|_{\mathrm{Lip}}\|X-Y\|_p.
    \end{align}
    If $F\in \tilde{B}_{\fz,1}^1(\bR),$ then  for $1\le p\le \infty$ we have \begin{align}\label{Lipes2}
        \|F(X)-F(Y)\|_p\lesssim_p \|F\|_{\tilde{B}_{\fz,1}^1}\|X-Y\|_p.
    \end{align}
\end{cor}

\section{Quantum chain rule}\label{s3.5}
In this section, we establish the following chain rule for $T_F$ on quantum Besov spaces, which serves as the key ingredient in the proof of our main results. 
\begin{theorem}\label{nc chain rule}
     Let $1\le p, q\le \infty,s>\max\lf(\frac dp,1\r),s\notin \bN$ and $\beta\in \bN^d$ such that $1\le |\beta|\le [s].$ Set $K:=|\beta|.$ Suppose that $F(0)=0$ and one of the following assumptions holds:
    \begin{enumerate}
        \item[\rm{(A)}]$F\in_{\mathrm{loc}} C^K(\bR),1<p<\infty.$
        \item[\rm{(B)}]$F\in_{\mathrm{loc}} \tilde B_{\fz,1}^K(\bR),1\le p\le \infty.$
    \end{enumerate}    Then for $u\in B_{p,q}^s(\rd)$ self-adjoint, we have $\partial^{\beta}F(u)\in L_p(\rd)$ and 
    \begin{equation}\label{chain rule formula} 
    \begin{aligned}
\partial^{\beta}F(u)=\sum_{\ell=1}^K\sum_{{\begin{subarray}{c}\alpha_1+\dots+\alpha_{\ell }=\beta \\ \alpha_1,\dots,\alpha_\ell\in \bN^d\setminus \{0\} \end{subarray}}}\frac{\beta!}{\alpha_1!\cdots \alpha_\ell!}T_{F^{[\ell]}}^{u,\dots,u}(\partial^{\alpha_1}u,\dots,\partial^{\alpha_{\ell}}u),
    \end{aligned}
    \end{equation} 
where for $\beta=(\beta^1,\dots,\beta^d)\in\bN^d,$ set $\beta!:=\beta^1!\cdots \beta^d!.$ 
\end{theorem}

In fact, various forms of chain rule in noncommutative settings have been extensively studied, which can be found in \cite{BS1999,Rota1980} for instance. For example, from \cite[Theorem 3.1]{BS1999} and \cite[Theorem 5.7]{ACDS2009}, for any inner derivative $\cd$ on $\cB(H),u\in \cB(H)$ self-adjoint and $F\in W_{n+1}(\bR),$ 
\begin{align*}
    \cd^nF(u)&=\sum_{r=1}^n\sum_{\substack{m_1k_1+\dots+m_rk_r=n\\ m_1+\dots+m_r=r}}c(r,m_1,\dots,m_r,k_1,\dots,k_{r})T_{F^{[r]}}^{u,\dots,u}((\cD^{m_1}u)^{k_1},\dots,(\cD^{m_r}u)^{k_r}).
\end{align*}
Our results have some novel aspects compared to their results. Explicitly, we require less regularity of $F$ and establish the chain rule for the quantum derivation $\partial_j:x\in \cB(L_2(\rd))\mapsto [\mathbf D_j,x],$ which can not be covered by their results since $\mathbf D_j\notin \cB(L_2(\rd)).$

In order to prove Theorem \ref{nc chain rule}, we first show the following lemmas. 
\begin{lemma}\label{homomorphism acting}
Let $\pi$ be a linear map given by $\pi:x\mapsto u^*xu$ for some unitary operator $u\in \cM.$ For a H{\"o}lder tuple $(p_1,\dots,p_n;p),$ suppose that one of the following assumptions holds:
    \begin{enumerate}
        \item[\rm{(A)}]$F\in C^n(\bR)$ with compact support, $1<p_j,p<\infty,j=1,\dots,n$.
        \item[\rm{(B)}]$F\in \tilde B_{\fz,1}^n(\bR),1\le p_j,p\le \infty,j=1,\dots,n.$
    \end{enumerate}   
    Then for $A_i \eta \cM_{sa},i=0,\dots,n$ and $X_j \in L_{p_j}(\cM),j=1,\dots,n,$ we have
    \begin{align*}
        \pi(T_{\phi}^{A_0,\dots,A_n}(X_1,\dots,X_n))=T_{\phi}^{\pi(A_0),\dots,\pi(A_n)}(\pi(X_1),\dots,\pi(X_n)).
    \end{align*}
    In particular, for every $X\in L_p(\cM),$ we have $\pi(F(X))=F(\pi(X)).$
\end{lemma}
\begin{proof}
    We first assume that $A_i \in \cM,i=0,\dots,n$ are self-adjoint and $X_j \in \cM,j=1,\dots,n.$ By the continuity of $\pi$ and Definition \ref{def of Moi}, we just need to show that for every $m\in\bN,$
    $$\pi\lf(S_{\phi,m}^{A_0,\dots,A_n}(X_1,\dots,X_n)\r)=S_{\phi,m}^{\pi(A_0),\dots,\pi(A_n)}\lf(\pi(X_1),\dots,\pi(X_n)\r).$$ 

    Indeed, note that $\pi$ is an endomorphism on $\cM,$ then by the expression of $S_{\phi,m}^{A_0,\dots,A_n},$ it suffices to show that for any $A\in \cM_{sa},$ we have $\pi(E_A)=E_{\pi(A)}$ for every $E_A=\chi_B(A)$ with $B=[\frac{l}{m},\frac{l+1}{m})$ for some $l\in \bZ.$ Note that we can find $f_n\in C(\bR)$ such that $0\le f_n\le 1$ and $f_n\to \chi_B$ pointwise. Hence $f_n(A)\sotto E_A$. Since $$\|f_n(A)\|\le\sup_{t\in \bR} f_n(t)\le1, \ \ \ n\in\bN^+,$$ $f_n(A)$ and $E_A$ are contained in the unit ball of $\cM$. It is well-known that the ultra-strong operator topology is equivalent to the $\sot$ on the unit ball of $\cM,$ then $f_n(A)\to E_A$ also in the ultra-strong operator topology. Due to \cite[Proposition 1.5.3]{Dixmier1977} and \cite[Theorem 2.4.23]{Bratteli-Robinson}, $$f_n(\pi(A))=\pi(f_n(A))\to \pi(E_A)$$ in the ultra-strong operator topology, and thus in the $\sot$. Moreover, $$f_n(\pi(A))\sotto \chi_B(\pi(A))=E_{\pi(A)},$$  which implies that $\pi(E_A)=E_{\pi(A)}.$

    Now for the general case, like in the proof of Proposition \ref{prop: perturbation formula}, we may choose  $A_i^{(k)}=A_iE_{A_i}[-k,k]\in \cM$ self-adjoint and $X_j^{(k)}=E_{X_i}[-k,k]\in \cL_{p_j}(\cM),k\in \bN$ such that $A_i^{(k)}\srto A_i$ and $X_j^{(k)}\overset{L_p}\to X_j$ when $k\to\infty$ for every $i=0,\dots,n,j=1,\dots,n.$ Clearly, by the properties of $\pi,$ we have $\pi(X_j^{(k)})\overset{L_p}\to \pi(X_j)$ and for any $z\in\bC\setminus\bR,$ $$(z-\pi(A_i^{(k)}))^{-1}=u(z-A_i^{(k)})^{-1}u^*\sotto u(z-A_i)^{-1}u^*=(z-\pi(A_i))^{-1},$$ which implies that $\pi(A_i^{(k)})\srto A_i$ when $k\to\infty.$ Meanwhile, for every $k\in\bN$ we have
    \begin{align*}
    \pi(T_{\phi}^{A_0^{(k)},\dots,A_n^{(k)}}(X_1^{(k)},\dots,X_n^{(k)}))\to T_{\phi}^{\pi(A_0^{(k)}),\dots,\pi(A_n^{(k)})}(\pi(X_1^{(k)}),\dots,\pi(X_n^{(k)})).
    \end{align*}
    Let $k\to\infty$ and by Proposition \ref{Lp continuity}, we deduce that.
\begin{align*}
        \pi(T_{\phi}^{A_0,\dots,A_n}(X_1,\dots,X_n))=T_{\phi}^{\pi(A_0),\dots,\pi(A_n)}(\pi(X_1),\dots,\pi(X_n)).
    \end{align*}
    
    Finally, note that  $\pi(F(X))=F(\pi(X))$ for every $X\in \cL_p(\cM).$ Hence by Corollary \ref{Lipes} and the continuity of $\pi,$ we similarly deduce that $\pi(F(X))=F(\pi(X))$ for every $X\in L_p(\cM).$
\end{proof}

\begin{lemma}\label{index pick}
    Let $u,s,p,q$ be as in Theorem \ref{nc chain rule}. For every $\beta\in \bN^d$  such that $1\le K:=|\beta|\le [s],\ell=1,\dots,K$ and $\alpha_1',\dots,\alpha_\ell'\in \bN^d\setminus \{0\}$ such that $\sum_{k=1}^\ell\alpha_k'=\beta,$ we can select a suitable H{\"o}lder tuple $(p_0,\dots,p_\ell;p)$ associated with $(\alpha_1',\dots,\alpha_\ell')$ with $1\le p_k\le \infty,k=0,\dots,\ell,$ satisfying $$\|u\|_{p_0} \ \ \text{and} \ \ \|u\|_{p_k}+\|\partial^{\alpha_k'}u\|_{p_k}\lesssim_{s,p,q} \|u\|_{B_{p,q}^s},\quad k=1,\dots,\ell.$$ Moreover, when $p\ne \infty,$ the exponents $p_k,k=0,\dots,\ell$ can be chosen such that $1<p_k<\infty.$ 
\end{lemma}
\begin{proof}
  If $p=\infty,$ we simply set $p_k:=\infty,k=0,\dots,\ell.$ Otherwise set \begin{align}\label{p0,...,pk}
      p_0:=\frac{Kp}{\delta},\ \ \ p_k:=\frac{Kp}{|\alpha_k'|-\delta_k},\ \ \ k  =1,\dots,\ell,
  \end{align} where $0\le\delta_k<|\alpha_k'|$ is to be chosen and $\delta:=\sum_{k=1}^\ell\delta_k.$ Next set
  $$s_k:=|\alpha_k'|+\varepsilon_k+(\frac dp-\frac{d}{p_k}),\ \ \ s_0:=\frac dp-\frac{d}{p_0},$$ where $\varepsilon_k>0$ depends on the choice of $\delta_k.$ We will prove that by choosing $\delta_k$ and $\varepsilon_k,$ the exponents $p_k$ can be arranged to satisfy all the requirements of this lemma. 

  First we observe that $$s_k=|\alpha_k'|+\varepsilon_k+\frac{(K-(|\alpha_k'|-\delta_k))d}{Kp}.$$
 Moreover, we claim that $$|\alpha_k'|+\frac{(s-|\alpha_k'|)Kp}{d}-K>0.$$ To see this, note if $K<\frac dp,$ then by $s>\frac dp$ we have $$|\alpha_k'|+\frac{(s-|\alpha_k'|)Kp}{d}-K>K(\frac{sp}{d}-1)>0;$$ if $K\ge \frac dp,$ then $$|\alpha_k'|+\frac{(s-|\alpha_k'|)Kp}{d}-K\ge \frac{(s-K)KP}{d}\ge \frac{\left \{ s \right \}Kp}{d}>0.$$ 
 Hence by choosing $\delta_k$ such that \begin{align}\label{delta k condition}
     \frac{Kp}{d}\varepsilon_k+\delta_k<\min\lf(|\alpha_k'|+\frac{(s-|\alpha_k'|)Kp}{d}-K,|\alpha_k'|\r),
 \end{align}  we obtain $s_k< s.$ 

  Now we estimate $\left \| \partial^{\alpha_k'}u \right \|_{p_k}$ and $\left \| u \right \|_{p_0}.$ By Proposition \ref{besov_embedding} \rm{(i)}, we have $$\|u\|_{p_k}+\| \partial^{\alpha_k'}u \|_{p_k} \le \left \| u \right \|_{W_{p_k}^{|\alpha_k'|}}\lesssim_{\varepsilon,p_k,q}\left \| u \right \|_{B_{p_k,q}^{|\alpha_k'|+\varepsilon}}$$ for any $\varepsilon>0.$ Choosing $\varepsilon=\varepsilon_k$  
 and by Proposition \ref{besov_embedding} \rm{(iii)}, $$\|u\|_{p_k}+\| \partial^{\alpha_k'}u \|_{p_k}\lesssim_{s,p,q} \left \| u \right \|_{B_{p_k,q}^{|\alpha_k'|+\varepsilon_k}}\lesssim_{s,p,q} \left \| u \right \|_{B_{p,q}^{s_k}}\le \left \| u \right \|_{B_{p,q}^{s}}.$$ 
{Here and below, the associated bounds depend only on $s,p,q$, see Remark \ref{independence}.}
 Similarly, by Proposition \ref{besov_embedding} \rm{(iii)}, $$\left \| u \right \|_{p_0}\lesssim_{s,p,q} \left \| u \right \|_{B_{p_0,q}^{s-s_0}}\lesssim_{s,p,q} \|u\|_{B_{p,q}^s}.$$ 
  
  Finally, note that the exponents $p_0,\dots,p_\ell$ defined in \eqref{p0,...,pk} satisfy that $1<p_k<\infty,k=0,\dots,\ell$ since $0\le \delta_k<|\alpha_k'|.$ Therefore,  $p_k,k=0,\dots,\ell$ satisfy all the requirements.
\end{proof}

\begin{remark}\label{independence}
    Here we explain why we can find $\kappa(s,p,q)>0$ such that
$$\|u\|_{p_k}+\| \partial^{\alpha_k'}u \|_{p_k}\le \kappa(s,p,q) \left \| u \right \|_{B_{p_k,q}^{|\alpha_k'|+\varepsilon_k}}.$$
Indeed, from Proposition \ref{besov_embedding} we would have
$$\|u\|_{p_k}+\| \partial^{\alpha_k'}u \|_{p_k}\le C(s,\alpha_k',\varepsilon_k,p,p_k,q) \left \| u \right \|_{B_{p_k,q}^{|\alpha_k'|+\varepsilon_k}}.$$
However, note that for a fixed tuple $(s,K,\alpha_k',p)$ appearing in the assumptions of Lemma \ref{index pick}, we can find and fix a pair $(\varepsilon_k,\delta_k)$ satisfying \eqref{delta k condition}, which means that one may write $$\varepsilon_k:=\varepsilon_k(s,p,K,\alpha_k'), \ \ \ \delta_k:=\delta_k(s,p,K,\alpha_k')$$ as functions of $s,p,K,\alpha_k'$. From \eqref{p0,...,pk}, $p_k$ is also a function of $s,p,K,\alpha_k',$ hence one may rewrite $ C(s,\alpha_k',\varepsilon_k,p,p_k,q)$ as $C'(s,K,\alpha_k',p,q).$ Finally, taking
$$\kappa(s,p,q)=\max_{K=1,\dots,[s]}\,\max_{|\alpha_k'|=1,\dots,K}C'(s,K,\alpha_k',p,q),$$ we conclude the proof.
\end{remark}

\begin{remark}\label{p0...pk3}
   From of Lemma \ref{index pick}, if we take 
   $$p_k=\frac{Kp}{|\alpha_k'|}, \ \ \ k=1,\dots,\ell,$$ then arguing as in Remark \ref{independence} we have
   \begin{align*}
\|u\|_{p_k}+\|\partial^{\alpha_k'}u\|_{p_k}\lesssim_{s,p,q} \|u\|_{B_{p,q}^s}, \ \ \ k=1,\dots,\ell.
   \end{align*}
Indeed, by \eqref{delta k condition}, it suffices to take $\delta_k=0$ and let
 \begin{align}\label{p0...pk4}
      \frac{Kp}{d}\varepsilon_k<\min\lf(|\alpha_k'|+\frac{(s-|\alpha_k'|)Kp}{d}-K,|\alpha_k'|\r).
 \end{align} 
\end{remark}

\begin{proof}[Proof of Theorem \ref{nc chain rule}]
Fix $u\in B_{p,q}^s(\rd)$ with $s>\frac dp$. By Proposition \ref{besov_embedding}, we have $u\in L_{\infty}(\rd).$  Therefore, for a given $F$ satisfying {\rm (A)} or {\rm (B)} in the assumptions, we may assume that $F\in X(\bR)$ with $\supp(F)\subset [-\|u\|_{\infty},\|u\|_{\infty}]$ and with norm $\|F\|_{X}^{\|u\|_\fz,\,\loc},$  where $X=C^K(\bR)$ when $1<p<\infty$ and $X=\tilde B_{\infty,1}^K(\bR)$ when $p=1,\infty.$ Then we can apply Proposition \ref{Lp continuity} and Lemma \ref{homomorphism acting} below. 
    For $p\ne \infty,$ we proceed by induction on $K.$ For $K=0,$ taking $X=u,Y=0$ in \eqref{Lipes1} implies $$\|F(u)\|_{p}\lesssim_p \|u\|_{p}.$$ For $K=1,$ we may assume that $\partial^{\beta}=\partial_1.$ Recall that the translation operator $T_{ze_1},z\in\bR$ is defined via the translation $\tau_{ze_1}$ on $L_2(\bR^d),$ that is $T_{ze_1}(x)=\tau_{ze_1}x\tau_{-ze_1}$ for any $x\in L_\infty(\rd)$. Moreover, note that $T_{ze_1}u \srto u$ by Proposition \ref{tran prop}. Hence we have
    \begin{eqnarray*}
\partial_1F(u)&\overset{\mathrm{Prop}\, \ref{partial prop}}=&\lim_{z \to 0}\frac{T_{ze_1}F(u)-F(u)}{z}\\
    &\overset{\mathrm{Lem}\, \ref{homomorphism acting}}=&\lim_{z \to 0}\frac{F(T_{ze_1}u)-F(u)}{z}\\
    &\overset{\mathrm{Prop}\, \ref{prop: perturbation formula}}=&\lim_{z \to 0}T^{T_{ze_1}u,u}_{F^{[1]}}\lf(\frac{T_{ze_1}u-u}{z}\r)\\
    &\overset{\mathrm{Prop}\, \ref{Lp continuity}}=&T^{u,u}_{F^{[1]}}(\partial_1u),
\end{eqnarray*}
where the limit exists in $L_p(\rd)$ norm. Moreover, this identity together with Proposition \ref{besov_embedding} \rm{(i)} and Proposition \ref{Lp boundedness} lead to $\|\partial_1F(u)\|_p\lesssim_p \|\partial_1u\|_p\lesssim_{s,p,q} \|u\|_{B_{p,q}^s}.$

Assume now that the assertion holds for $K-1,$ we then show it for $K.$ We may assume that $\partial^{\beta}=\partial_1\partial^{\alpha},$ where $\alpha\in \bN^d$ with $|\alpha|=K-1.$ Then by the same reason as $K=1,$ the induction hypothesis implies that
    \begin{eqnarray*}
\partial^{\beta}F(u)&\overset{\mathrm{Prop}\, \ref{partial prop}}=&\lim_{z \to 0}\frac{T_{ze_1}\partial^{\alpha}F(u)-\partial^{\alpha}F(u)}{z}\\
    &\overset{\mathrm{Lem}\, \ref{homomorphism acting}}=&\lim_{z\to 0}\sum_{j=1}^{K-1}\sum_{{\begin{subarray}{c}\gamma_1+\dots+\gamma_{j}=\alpha \\ \gamma_1,\dots,\gamma_j\in \bN^d\setminus \{0\} \end{subarray}}}\frac{1}{z}\frac{\alpha!}{\gamma_1!\cdots \gamma_j!}T_{F^{[j]}}^{T_{ze_1}u,\dots,T_{ze_1}u}(T_{ze_1}\partial^{\gamma_1}u,\dots,T_{ze_1}\partial^{\gamma_{j}}u)\\
    &&-\lim_{z\to 0}\sum_{j=1}^{K-1}\sum_{{\begin{subarray}{c}\gamma_1+\dots+\gamma_{j}=\alpha \\ \gamma_1,\dots,\gamma_j\in \bN^d\setminus \{0\} \end{subarray}}}\frac{1}{z}\frac{\alpha!}{\gamma_1!\cdots \gamma_j!}T_{F^{[j]}}^{u,\dots,u}(\partial^{\gamma_1}u,\dots,\partial^{\gamma_{j}}u).
\end{eqnarray*}
with the limit again taken in $L_p(\rd)$ norm. It therefore suffices to analyze terms of the form $$T_{F^{[j]}}^{T_{ze_1}u,\dots,T_{ze_1}u}(\partial^{\gamma_1}T_{ze_1}u,\dots,\partial^{\gamma_j}T_{ze_1}u)
    -T_{F^{[j]}}^{u,\dots,u}(\partial^{\gamma_1}u,\dots,\partial^{\gamma_j}u)$$ for $j=1,\dots,K-1$ and $\sum_{k=1}^j|\gamma_k|=K-1.$

    For this purpose, we define that: for any $t\in\bN^+,m=1,\dots,t$ and two vectors $a=(a_1,\dots,a_t), b=(b_1,\dots,b_t),$ define the transform
\begin{align}\label{Tq}
    \mathsf T_m(a,b)=(a_1,\dots,a_m,b_{m+1},\dots,b_t)
\end{align}
and adopt the convention that
$\mathsf T_0(a,b)=a.$
 Moreover, we simply write that
    \begin{align*}
&\vec{u}=(u,\dots,u)\\
&T_{ze_1}\vec{u}=(T_{ze_1}u,\dots,T_{ze_1}u)\\
&\vec{v}=(\partial^{\gamma_1}u,\dots,\partial^{\gamma_j}u)\\
&T_{ze_1}\vec{v}=(T_{ze_1}\partial^{\gamma_1}u,\dots,T_{ze_1}\partial^{\gamma_j}u).
    \end{align*}
 Now we first observe that
 \begin{equation}\label{sum and minus}
     \begin{aligned}
        &\ \ \ \ \ \ \frac{1}{z}\lf(T_{F^{[j]}}^{T_{ze_1}u,\dots,T_{ze_1}u}(\partial^{\gamma_1}T_{ze_1}u,\dots,\partial^{\gamma_j}T_{ze_1}u)
    -T_{F^{[j]}}^{u,\dots,u}(\partial^{\gamma_1}u,\dots,\partial^{\gamma_j}u)\r)\\
    &=\frac 1z\lf(\sum_{r=1}^{j+1}\lf(T_{F^{[j]}}^{\mathsf T_r(T_{ze_1}\vec{u},\vec{u})}(T_{ze_1}\vec{u})-T_{F^{[j]}}^{\mathsf T_{r-1}(T_{ze_1}\vec{u},\vec{u})}(T_{ze_1}\vec{v})\r)+\sum_{r=1}^j\lf(T_{F^{[j]}}^{\vec{u}}(\mathsf T_r(T_{ze_1}\vec{v},\vec{v}))-T_{F^{[j]}}^{\vec{u}}(\mathsf T_{r-1}(T_{ze_1}\vec{v},\vec{v}))\r)\r).
    \end{aligned}
 \end{equation}
    Applying Lemma \ref{index pick} for $(\alpha_1',\dots,\alpha_j')=(\gamma_1,\dots,\gamma_j),$ one may choose 
$(p_0,\dots,p_j)$ defined in \eqref{p0,...,pk}
 such that $\partial^{\gamma_k}u\in L_{p_k}(\rd),k=1,\dots,j$ and $u\in L_{p_0}(\rd).$
Hence applying Proposition \ref{prop: perturbation formula} with $A=T_{ze_1}u$ and $B=u,$ we then have 
\begin{align}\label{diff1}
    \frac 1z\lf(T_{F^{[j]}}^{\mathsf T_r(T_{ze_1}\vec{u},\vec{u})}(T_{ze_1}\vec{u})-T_{F^{[j]}}^{\mathsf T_{r-1}(T_{ze_1}\vec{u},\vec{u})}(T_{ze_1}\vec{v})\r)=T_{F^{[j+1]}}^{a^r}(b^r).
\end{align}
By direct calculation, we also have
\begin{align}\label{diff2}
    \frac 1z\lf(T_{F^{[j]}}^{\vec{u}}(\mathsf T_r(T_{ze_1}\vec{v},\vec{v}))-T_{F^{[j]}}^{\vec{u}}(\mathsf T_{r-1}(T_{ze_1}\vec{v},\vec{v}))\r)=T_{F^{[j]}}^{u,\dots,u}(c^r),
\end{align}
which implies
\begin{equation}\label{difference equality}
        \begin{aligned}
           \frac{1}{z}\lf(T_{F^{[j]}}^{T_{ze_1}u,\dots,T_{ze_1}u}(\partial^{\gamma_1}T_{ze_1}u,\dots,\partial^{\gamma_j}T_{ze_1}u)
    -T_{F^{[j]}}^{u,\dots,u}(\partial^{\gamma_1}u,\dots,\partial^{\gamma_j}u)\r)
=\sum_{r=1}^{j+1}T_{F^{[j+1]}}^{a^r}(b^r)+\sum_{r=1}^{j}T_{F^{[j]}}^{u,\dots,u}(c^r),
        \end{aligned}
    \end{equation}
where we recall that $\partial_1^z:=\frac{T_{ze_1}-1}{z}$ and
\begin{align*}
    &a^r=(\underset{{r} }{\underbrace{T_{ze_1}u,\dots,T_{ze_1}u}},\underset{{j-r+2} }{\underbrace{u,\dots,u}} )\\
    &b^r=(T_{ze_1}\partial^{\gamma_1}u,\dots,T_{ze_1}\partial^{\gamma_{r-1}}u,\partial_1^zu,T_{ze_1}\partial^{\gamma_{r}} u,\dots,T_{ze_1}\partial^{\gamma_j}u)\\
    &c^r=(T_{ze_1}\partial^{\gamma_1}u,\dots,T_{ze_1}\partial^{\gamma_{r-1}}u,\partial_1^z\partial^{\gamma_{r}} u,\partial^{\gamma_{r+1}}u,\dots,\partial^{\gamma_j}u).
\end{align*}

Now we start with \eqref{difference equality}. When $z\to 0,$ we will show that the right side converges to $$\sum_{r=1}^{j+1}T_{F^{[j+1]}}^{u,\dots,u}(\xi^r)+\sum_{r=1}^jT_{F^{[j]}}^{u,\dots,u}(\eta^r)$$ in $L_p(\rd)$ norm, where 
\begin{align*}
    &\xi^r=(\partial^{\gamma_1}u,\dots,\partial^{\gamma_{r-1}}u,\partial_1 u,\partial^{\gamma_{r}}u,\dots,\partial^{\gamma_j}u)\\
    &\eta^r=(\partial^{\gamma_1}u,\dots,\partial^{\gamma_{r-1}}u,\partial_1\partial^{\gamma_{r}} u,\partial^{\gamma_{r+1}}u,\dots,\partial^{\gamma_j}u).
\end{align*}
Indeed, let $(\alpha_1',\dots,\alpha_{j+1}')=(\gamma_1,\dots,\gamma_{r-1},e_1,\gamma_r,\dots,\gamma_j),$ then we have $\sum_{k=1}^{j+1}|\alpha_k'|=\sum_{k=1}^j|\gamma_k|+1=K.$
Hence we define $(p_1,\dots,p_{j+1})$ as in \eqref{p0...pk4} and apply Remark \ref{p0...pk3} for $(\alpha_1',\dots,\alpha_{j+1}')$ to obtain that $u\in W_{p_k}^{|\alpha_k|}(\rd),k=1,\dots,j+1.$ In particular, we have $u\in W_{p_r}^1(\rd),$ which implies that $\lim_{z\to 0}\|\partial_1^zu-\partial_1u\|_{p}=0$ by Proposition \ref{partial prop}.
Recall that by Proposition \ref{tran prop}, $T_{ze_1}u\sotto u$ when $z\to 0$ since $u\in  L_\infty(\rd)$ and $\lim_{z\to 0}\|T_{ze_1}\partial^{\alpha_k}u-\partial^{\alpha_k}u\|_{p_k}=0.$
Hence by Proposition \ref{Lp continuity}, we deduce that
\begin{align*}
    T_{F^{[j+1]}}^{a^r}(b^r)\to T_{F^{[j+1]}}^{u,\dots,u}(\xi^r)
\end{align*}
in $L_p(\rd)$ norm. Similarly we have
\begin{align*}
    T_{F^{[j]}}^{u,\dots,u}(c^r)\to T_{F^{[j]}}^{u,\dots,u}(\eta^r)
\end{align*}
in $L_p(\rd)$ norm.
Hence 
\begin{align*}
    \sum_{r=1}^{j+1}T_{F^{[j+1]}}^{a^r}(b^r)+\sum_{r=1}^{j}T_{F^{[j]}}^{u,\dots,u}(c^r)\to\sum_{r=1}^{j+1}T_{F^{[j+1]}}^{u,\dots,u}(\xi^r)+\sum_{r=1}^jT_{F^{[j]}}^{u,\dots,u}(\eta^r)
\end{align*}
in $L_p(\rd)$ norm.
Taking sum over $j=1,\dots,K$ yields the desired formula \eqref{chain rule formula}.  
 
 Finally starting with \eqref{chain rule formula}, we apply Remark \ref{p0...pk3} for $(\alpha_1',\dots,\alpha_\ell')=(\alpha_1,\dots,\alpha_\ell)$  
 and define $(p_1,\dots,p_\ell)$ as in \eqref{p0...pk4}. Then Proposition \ref{Lp boundedness} yields
\begin{align*}   
\lf\|T_{F^{[\ell]}}^{u,\dots,u}(\partial^{\alpha_1}u,\dots,\partial^{\alpha_\ell}u)\r\|_p&\lesssim_{s,p,F,\|u\|_\infty} \prod_{k=1}^\ell\|\partial^{\alpha_k}u\|_{p_k}\\
&\lesssim_{s,p,q} 1+\|u\|_{B_{p,q}^s}^{K}<\infty.
\end{align*}
We then conclude the desired estimate by taking sum over $\ell=1,\dotsm, K$. 

For $p=\infty,$ we can similarly deduce \eqref{chain rule formula} and $\partial^{\beta}F(u)\in L_{\infty}(\rd)$ by utilizing Proposition \ref{Lp boundedness} and Proposition \ref{SOT continuity}, in place of Proposition \ref{Lp continuity}. We should point out that, in the case $p=\infty,$ some detailed difference will arise, for instance, all the limits appearing in the proof should exist in $\sot$; but this 
  will not affect the outline of proof.

\end{proof}


\section{Proof of the main results}\label{s4}
In this section, we will prove our main results, i.e. Theorems \ref{s smaller than 1}, \ref{main re 1} and \ref{main re 4}.
Our proof relies mainly on the difference characterization of quantum Besov spaces in Section \ref{s2.5}, the theories of multiple operator integrals in Section \ref{s3}, the quantum chain rule in Section \ref{s3.5} and nonlinear interpolation introduced by Tartar \cite{Tartar}. Like in the proof of Theorem \ref{nc chain rule}, if $F\in_{\mathrm{loc}} X(\bR)$ for $X(\bR)$ appearing in the main theorems and $u\in L_\infty(\rd),$ then we may assume that $F\in X(\bR)$ with norm $\|F\|_{X}^{\|u\|_\fz,\,\loc}$.

\begin{proof}[Proof of Theorem \ref{s smaller than 1}]
   If the assumption \rm{(B)} or \rm{(D)} holds, then for every $0<|h|\le t,$ by \eqref{Lipes1} in Corollary \ref{Lipes} and Lemma \ref{homomorphism acting} we have $$\|\Delta_h^1F(u)\|_p=\|F(T_hu)-F(u)\|_p\lesssim_p \|F\|_{X}\|\Delta_h^1u\|_p$$ and $$\|F(u)\|_p\lesssim_p \|F\|_{X} \|u\|_p,$$ where $X=\rm{Lip}$ when assumption \rm{(B)} holds and $X=\tilde B_{\infty,1}^1$ when assumption \rm{(D)} holds. Hence applying \eqref{s notin N} in Theorem \ref{equiv ch}, we have
    $$\|F(u)\|_{B_{p,q}^s}\lesssim_p \|F\|_{X}\lf(\left \| u \right \|_p+\lf(\int_0^{\infty}\lf(t^{-s}\omega_p^1(t,u)\r)^q \frac{dt}{t}\r)^{1/q}\r)\lesssim_{s,p,q}  \|F\|_{X}\|u\|_{B_{p,q}^s}.$$

    If the assumption \rm{(A)} or \rm{(C)} holds, we similarly deduce the conclusion with the constant $\|F\|_{X}^{\|u\|_\fz,\,\loc}$.

\end{proof}

The proof of Theorem \ref{main re 1} will be divided into two cases: $s\notin \bN$ and $s\in \bN.$ For $s\notin \bN,$
we begin with the following lemma, which is stronger than Lemma \ref{index pick}. 
\begin{lemma}\label{index 2}
    Let $u,s,p,q,\beta$ be as in Theorem \ref{main re 1} with $s\notin \bN$ and  $i=1,\dots,d.$ For any $\ell=1,\dots,[s],$ $\alpha_0'\in\bN,\alpha_k'\in\bN^+,k=1,\dots,\ell$ and $\sum_{k=0}^\ell\alpha_k'=[s],$ we can select a suitable H{\"o}lder tuple $(p_0,\dots,p_\ell;p)$ associated with $(\alpha_0',\dots,\alpha_\ell')$ with $1\le p_k\le \infty,k=0,\dots,\ell,$ satisfying $$\|u\|_{p_k}+\|\partial_i^{\alpha_k'}u\|_{p_k}\lesssim_{s,p,q} \|u\|_{B_{p,q}^s},\ \ \ k=1,\dots,\ell$$ and \begin{align}\label{difference estimate}
\|u\|_{p_0}+\lf(\int_0^{\infty}\lf(t^{-\left \{ s \right \} }\omega_{p_0}^1(t,\partial_i^{\alpha_0'} u)\r)^q \frac{dt}{t}\r)^{1/q}\sim_{s,p,q} \|\partial_i^{\alpha_0'}u\|_{B_{p_0,q}^{\{s\}}} \lesssim_{s,p,q} \|u\|_{B_{p,q}^{s}}.
    \end{align}
Moreover, when $p\ne \infty,$ the exponents $p_k,k=0,\dots,\ell$ can be chosen such that $1<p_k<\infty.$ 
\end{lemma}
\begin{proof}
When $p=\infty,$ set $p_k:=\infty$. Otherwise for any $0<\delta_k<\alpha_k',k=1,\dots,\ell,$ set $\delta:=\sum_{k=1}^\ell \delta_k$ and \begin{align}\label{p0,..pk2}
    p_0:=\frac{[s]p}{\alpha_0'+\delta}, \ \ \ p_k:=\frac{[s]p}{\alpha_k'-\delta_k}, \ \ \ k=1,\dots,\ell.
\end{align}  For $p=\infty,$ the assertions are easy and we omit the proof. For $p\ne\infty,$
   it suffices to show the existence of  $\delta_k,k=1,\dots,\ell$ such that the following inequalities hold:
   \begin{align}\label{delta2}
   \delta_k<\min\lf(\alpha_k'+\frac{(s-\alpha_k')[s]p}{d}-[s],\alpha_k'\r)
\end{align}
and
   \begin{align}\label{delta1}
       \delta>\max\lf(\lf(1-\frac{([s]-\alpha_0')p}{d}\r)[s]-\alpha_0',0\r)
   \end{align}
 Indeed, 
 if \eqref{delta2} holds, replacing $K$ with $[s]$ in \eqref{delta k condition}, and taking  $\varepsilon_k$  sufficiently small, we then deduce similarly that \begin{align}\label{alphak'}
     s_k:=\alpha_k'+\varepsilon_k+\frac dp-\frac d{p_k}<s, \ \ \ k=1,\dots,\ell.
 \end{align} 
 Arguing as in Remark \ref{independence}, we deduce the estimate $$\|u\|_{p_k}+\|\partial_i^{\alpha_k'}u\|_{p_k}\lesssim_{s,p,q}\|u\|_{B_{p,q}^s}$$ by Proposition \ref{besov_embedding} \rm{(i)}, \rm{(iii)}. 
 Similarly, if \eqref{delta1} holds, then we can deduce that \begin{align}\label{alpha0}
     s_0:=\alpha_0'+\{s\}+\frac dp-\frac d{p_0}<s,
 \end{align} which implies that
 $$\|\partial_i^{\alpha_0'}u\|_{B_{p_0,q}^{\{s\}}} \lesssim_{s,p,q}\|u\|_{B_{p_0,q}^{\{s\}+\alpha_0'}}\lesssim_{s,p,q} \|u\|_{B_{p,q}^{s}}$$
 together with Proposition \ref{besov_embedding} \rm{(iii)} and Proposition \ref{Besov regularity}. 

 Now we show that the choice of such $\delta_k$'s is possible. First, note that \eqref{delta2} implies
 \begin{align}\label{sup}
     \delta:=\sum_{k=1}^\ell\delta_k<\min\lf([s]-\alpha_0'+\frac{(\ell s-[s]+\alpha_0')p[s]}{d}-\ell [s],[s]-\alpha_0'\r).
 \end{align}
 Second, since $\alpha_1',\dots,\alpha_\ell'>0,$ we have $\alpha_0'<[s]$, then it is straightforward to verify that
 \begin{equation}\label{supinf}
     \begin{aligned}
    \min\lf([s]-\alpha_0'+\frac{(\ell s-[s]+\alpha_0')p[s]}{d}-\ell [s],[s]-\alpha_0'\r)&\ge \min\lf(\frac{\{s\}[s]p}d-\lf(1-\frac{p[s]}{d}\r)\alpha_0',[s]-\alpha_0'\r)\\
 &>\max\lf(\lf(1-\frac{([s]-\alpha_0')p}{d}\r)[s]-\alpha_0',0\r).
\end{aligned}
 \end{equation}
 Therefore, it is possible to choose $\delta_k$'s properly such that  \eqref{delta1} and \eqref{sup}  hold simultaneously, and so do \eqref{delta2} and \eqref{delta1}.
\end{proof}


Note that Lemma \ref{index 2} requires that there is at most one index in $(\alpha_0',\dots,\alpha_\ell')$ that is equal to $0.$ We also obtain the following variant of Lemma \ref{index 2}, where we allow the existence of two indices that may be equal to $0$.
\begin{lemma}\label{index 3}
    With the same assumptions of Lemma \ref{index 2} and making a convention that $\alpha_{-1}'=0,$ we can select a suitable H{\"o}lder tuple $(p_{-1},\dots,p_\ell;p)$ associated with $(\alpha_0',\dots,\alpha_\ell')$ with $1\le p_k\le \infty,k=-1,\dots,\ell,$ satisfying $$\|u\|_{p_k}+\|\partial_i^{\alpha_k'}u\|_{p_k}\lesssim_{s,p,q} \|u\|_{B_{p,q}^s}, \ \ \ k=-1,1,\dots,\ell$$ and \eqref{difference estimate}.
Moreover, when $p\ne \infty,$ the exponents $p_k,k=-1,\dots,\ell$ can be chosen such that $1<p_k<\infty.$ 
\end{lemma}
\begin{proof}
When $p=\infty,$ set $p_k:=\infty$. Otherwise for any $0<\delta_k<\alpha_k',k=1,\dots,\ell$ and any $\delta_{-1},\delta_0>0$ such that $\delta_{0}+\delta_{-1}=\sum_{k=1}^\ell \delta_k,$ set $\delta:=\delta_0+\delta_{-1}$ and \begin{align}\label{p0,..pk5}
   p_{-1}:=\frac{[s]p}{\delta_{-1}}, \ \ \ p_0:=\frac{[s]p}{\alpha_0'+\delta_0},\ \ \ p_k:=\frac{[s]p}{\alpha_k'-\delta_k}, \ \ \ k=1,\dots,\ell.
\end{align}  For $p=\infty,$ the assertions are easy and we omit the proof.
 For $p\ne\infty,$ since we have $\delta=\sum_{k=1}^\ell\delta_k,$ as the proof of Lemma \ref{index 2}, there exists $\delta_k,k=1,\dots,\ell$ satisfying  \eqref{delta2} such that \eqref{delta1} holds, which implies the estimate $$\|u\|_{p_k}+\|\partial_i^{\alpha_k'}u\|_{p_k}\lesssim_{s,p,q}\|u\|_{B_{p,q}^s}, \ \ \ k=1,\dots,\ell.$$ Moreover, recall that \eqref{delta2} implies 
 \eqref{sup}. Then we can choose $\delta_{-1},\delta_0$ such that $0<\delta_{-1}<[s]$ sufficiently small and $\delta_0$ satisfying \eqref{delta1} simultaneously,  which implies \eqref{difference estimate} arguing as in Lemma \ref{index 2} and $$\|u\|_{p_{-1}}\lesssim_{s,p,q}\|u\|_{B_{p,q}^s}$$ by Proposition \ref{besov_embedding} \rm{(iii)}.  
\end{proof}

\begin{proof}[Proof of Theorem \ref{main re 1} when $s\notin\bN$]
The case $0<s<1$ is covered by Theorem \ref{s smaller than 1}.
    When $s>1,$ 
    by \eqref{s notin N}, we have
    \begin{align}\label{want}
        \|F(u)\|_{B_{p,q}^s}\sim_{s,p,q} \left \| F(u) \right \|_p+\sum_{j=1}^d\lf(\int_0^{\infty}\lf(t^{-\left \{ s \right \} }\omega_p^1(t,\partial_i^{[s]}F(u))\r)^q \frac{dt}{t}\r)^{1/q}.
    \end{align}
    By substituting \eqref{chain rule formula} into \eqref{want} and by the triangle inequality, it suffices to analyze terms like \begin{align}\label{analyze terms}
        \lf(\int_0^{\infty}\lf(t^{-\left \{ s \right \} }\sup_{|h|\le t}\left \|\Delta_h^1 T_{F^{[\ell]}}^{u,\dots,u}(\partial_i^{\alpha_1}u,\dots,\partial_i^{\alpha_\ell}u) \right \|_p\r)^q \frac{dt}{t}\r)^{1/q}, \ \ \ \ell=1,\dots,[s],\ \ \ i=1,\dots,d
    \end{align}
    with $\alpha_1,\dots,\alpha_\ell\in \bN^+$ such that $\sum_{k=1}^\ell \alpha_k=[s].$
    Similar to the proof of \eqref{difference equality}, by Proposition \ref{homomorphism acting}, Lemma \ref{index pick} and Proposition \ref{prop: perturbation formula}, we have
 \begin{align}\label{diff uv}
     \Delta_h^1 T_{F^{[\ell]}}^{u,\dots,u}(\partial_i^{\alpha_1}u,\dots,\partial_i^{\alpha_\ell}u)=\sum_{r=1}^{\ell+1}T_{F^{[\ell+1]}}^{d^r(u)}(e^r(u))+\sum_{r=1}^{\ell}T_{F^{[\ell]}}^{u,\dots,u}(f^r(u)),
 \end{align}
where 
\begin{equation}\label{def}
    \begin{aligned}
    &d^r(u)=(\underset{{r} }{\underbrace{T_{h}u,\dots,T_{h}u}},\underset{{\ell-r+2} }{\underbrace{u,\dots,u}} )\\
&e^r(u)=(T_h\partial_i^{\alpha_1}u,\dots,T_h\partial_i^{\alpha_{r-1}}u,\Delta_h^1u,\partial_i^{\alpha_{r}} u,\dots,\partial_i^{\alpha_\ell}u)\\
    &f^r(u)=(T_h\partial_i^{\alpha_1}u,\dots,T_h\partial_i^{\alpha_{r-1}}u,\partial_i^{\alpha_r}\Delta_h^1 u,\partial_i^{\alpha_{r+1}}u,\dots,\partial_i^{\alpha_\ell}u).
\end{aligned}
\end{equation}
 By the triangle inequality,
 \begin{equation}\label{analyze0}
     \begin{aligned}
         \eqref{analyze terms}&\le\sum_{r=1}^{\ell+1} \lf(\int_0^{\infty}\lf(t^{-\left \{ s \right \} }\sup_{|h|\le t}\left \|T_{F^{[\ell+1]}}^{d^r(u)}(e^r(u)) \right \|_p\r)^q \frac{dt}{t}\r)^{1/q}\\
         &\quad+\sum_{r=1}^{\ell}\lf(\int_0^{\infty}\lf(t^{-\left \{ s \right \} }\sup_{|h|\le t}\left \|T_{F^{[\ell]}}^{u,\dots,u}(f^r(u)) \right \|_p\r)^q \frac{dt}{t}\r)^{1/q}\\
&=:\sum_{r=1}^{\ell+1}{\rm I}_r^\ell+\sum_{r=1}^\ell {\rm J}_r^{\ell}.
     \end{aligned}
 \end{equation}
We first consider the terms ${\rm I}_r^\ell.$   
   By applying Lemma \ref{index 2} to $(\alpha_0',\dots,\alpha_\ell')=(0,\alpha_1,\dots,\alpha_\ell)$ and arguing as in Remark \ref{independence}, one may find $p_k,k=0,\dots,\ell$ as defined in \eqref{p0,..pk2} so that Proposition \ref{Lp boundedness} and \eqref{Lp bound es 3} imply that 
 \begin{align}\label{Ir}
     {\rm I}_r^\ell \lesssim_{s,p,q}\|F\|_X^{\|u\|_\fz,\,\loc}\|u\|_{B_{p_0,q}^{\{s\}}}\prod_{k=1}^\ell\|\partial_i^{\alpha_k}u\|_{p_k}
     \lesssim \|F\|_X^{\|u\|_\fz,\,\loc}\|u\|_{B_{p,q}^s}^{\ell+1},
 \end{align}
    where $X=C_b^{\lceil s\rceil}$ when $1<p<\infty$ and $X=\tilde{B}_{\infty,1}^1\cap \tilde{B}_{\infty,1}^{\lceil s\rceil}$ when $p=1,\infty.$
Similarly, for ${\rm J}_r^\ell$,  one applies Lemma \ref{index 2} to $(\alpha_0',\dots,\alpha_{\ell-1}')=(\alpha_r,\alpha_1,\dots,\alpha_{r-1},\alpha_{r+1},\dots,\alpha_{\ell})$ and deduces that
\begin{align}\label{Jr}
 {\rm J}_r^\ell\lesssim_{s,p,q}\|F\|_X^{\|u\|_\fz,\,\loc}\|u\|_{B_{p_0,q}^{\alpha_0'+\{s\}}}\prod_{k=1}^{\ell-1}\|\partial_i^{\alpha_k'}u\|_{p_k}\lesssim \|F\|_X^{\|u\|_\fz,\,\loc}\|u\|_{B_{p,q}^s}^{\ell}.
\end{align}
 Hence by taking sums over $r,\ell,$
 we  obtain \eqref{bound es} and \eqref{bound es 1}.
\end{proof}

We also have the following local Lipschitz estimates when $s\notin \bN.$
\begin{prop}\label{main re 2}
    Let $1\le q\le\infty,s>\frac dp,s\notin \bN$ and $u,v\in B_{p,q}^s(\rd)$ be self-adjoint. If $F\in C^{\left \lceil s \right \rceil+1}(\bR),$ then for $1<p<\infty,$  we have \begin{align*}
        \|F(u)-F(v)\|_{B_{p,q}^s}\lesssim_{s,p,q} \|F\|_{C_b^{\lceil s\rceil+1}}^{\max(\|u\|_\fz,\|v\|_\fz),\mathrm{loc}}\lf(1+\max(\|u\|_{B_{p,q}^s},\|v\|_{B_{p,q}^s})^{[s]+1}\r)\|u-v\|_{B_{p,q}^s}.
    \end{align*}
    If $F \in_{\loc} \tilde{B}_{\infty,1}^{1}(\mathbb{R})\cap \tilde{B}_{\infty,1}^{\left \lceil s \right \rceil+1}(\mathbb{R}),$ then for $1\le p\le\infty,$ and $u,v\in B_{p,q}^s(\rd)$ self-adjoint we have \begin{align*}
        \|F(u)-F(v)\|_{B_{p,q}^s}\lesssim_{s,p,q} \|F\|_{\tilde B_{\fz,1}^1\cap \tilde B_{\fz,1}^{\lceil s\rceil+1}}^{\max(\|u\|_\fz,\|v\|_\fz),\mathrm{loc}}\lf(1+\max(\|u\|_{B_{p,q}^s},\|v\|_{B_{p,q}^s})^{[s]+1}\r)\|u-v\|_{B_{p,q}^s}.
    \end{align*}
\end{prop}
\begin{proof}
Since $u,v\in L_\fz(\rd),$ we may assume that $F\in X(\bR)$ with norm $\|F\|_{X}^{\max(\|u\|_\fz,\|v\|_\fz),\,\loc},$ where $X=C_b^{\lceil s\rceil+1}$ when $1<p<\fz$ and $X=\tilde B_{\fz,1}^1\cap \tilde B_{\fz,1}^{\lceil s\rceil+1}$ when $p=1,\fz.$ As in the above proof of Theorem \ref{main re 1} for $s\notin \bN$, we arrive at the step to  estimate the terms similar to that appeared in \eqref{analyze0}: for $\ell=1,\dots,[s],$
    \begin{align*}
        &\ \ \ \lf(\int_0^{\infty}\lf(t^{-\left \{ s \right \} }\sup_{|h|\le t}\left \|\Delta_h^1 T_{F^{[\ell]}}^{u,\dots,u}(\partial_i^{\alpha_1}u,\dots,\partial_i^{\alpha_\ell}u)-\Delta_h^1 T_{F^{[\ell]}}^{v,\dots,v}(\partial_i^{\alpha_1}v,\dots,\partial_i^{\alpha_\ell}v)\right \|_p\r)^q \frac{dt}{t}\r)^{1/q}\\
&\le\sum_{r=1}^{\ell+1}\lf(\int_0^{\infty}\lf(t^{-\left \{ s \right \} }\sup_{|h|\le t}\lf \|\lf(T_{F^{[\ell+1]}}^{d^r(u)}(e^r(u))-T_{F^{[\ell+1]}}^{d^r(v)}(e^r(v))\r)\r\|_p\r)^q \frac{dt}{t}\r)^{1/q}\\
&\quad+\sum_{r=1}^{\ell}\lf(\int_0^{\infty}\lf(t^{-\left \{ s \right \} }\sup_{|h|\le t}\lf \|\lf(T_{F^{[\ell]}}^{u,\dots,u}(f^r(u))-T_{F^{[\ell]}}^{v,\dots,v}(f^r(v))\r)\r\|_p\r)^q \frac{dt}{t}\r)^{1/q}\\
&=:\sum_{r=1}^{\ell+1}{\rm K}_r^\ell+\sum_{r=1}^{\ell}{\rm L}_r^{\ell},
    \end{align*}
where $d^r,e^r,f^r$ is defined as in \eqref{def}.
We only need to consider ${\rm K}_r^\ell$ since ${\rm L}_r^\ell$ can be estimated similarly. 
Let $\mathsf T_m$ be defined as in \eqref{Tq}, then similar to \eqref{sum and minus} we have
\begin{align*}
    T_{F^{[\ell+1]}}^{d^r(u)}(e^r(u))-T_{F^{[\ell+1]}}^{d^r(v)}(e^r(v))&=\sum_{m=1}^{\ell+2}\lf(T_{F^{[\ell+1]}}^{\mathsf T_m(d^r(u),d^r(v))}(e^r(u))-T_{F^{[\ell+1]}}^{\mathsf T_{m-1}(d^r(u),d^r(v))}(e^r(u))\r)\\
&\quad+\sum_{m=1}^{\ell+1}\lf(T_{F^{[\ell+1]}}^{d^r(v)}(\mathsf T_m(e^r(u),e^r(v)))-T_{F^{[\ell+1]}}^{d^r(v)}(\mathsf T_{m-1}(e^r(u),e^r(v)))\r).
\end{align*}
By Remark \ref{depend}, one may write $T_{F^{[n]}}^{A_0,\dots,A_n}$ as $T_{F^{[n]}}$ to simplify the notation.
Then similar to the proof of \eqref{diff1} and \eqref{diff2}, by Lemma \ref{homomorphism acting}, Lemma \ref{index pick} and Proposition \ref{prop: perturbation formula} we have
 \begin{align*}
     T_{F^{[\ell+1]}}^{\mathsf T_m(d^r(u),d^r(v))}(e^r(u))-T_{F^{[\ell+1]}}^{\mathsf T_{m-1}(d^r(u),d^r(v))}(e^r(u))=T_{F^{[\ell+2]}}(g^{m,r}(u,v))=:{\rm I}_{\ell,r,m}
 \end{align*}
 and
 \begin{align*}
     T_{F^{[\ell+1]}}^{d^r(v)}(\mathsf T_m(e^r(u),e^r(v)))-T_{F^{[\ell+1]}}^{d^r(v)}(\mathsf T_{m-1}(e^r(u),e^r(v)))=T_{F^{[\ell+1]}}(h^{m,r}(u,v))=:{\rm II}_{\ell,r,m},
 \end{align*}
 where \begin{align*}
     &g^{m,r}(u,v)=(T_h\partial_i^{\alpha_1}u,\dots,T_h\partial_i^{\alpha_m}u,T_h(u-v),T_h\partial_i^{\alpha_{m+1}}u,\dots,T_h\partial_i^{\alpha_{r-1}}u,\Delta_h^1u,\partial_i^{\alpha_{r}} u,\dots,\partial_i^{\alpha_\ell}u), \ \ \ m\le r-1\\
      &g^{m,r}(u,v)=(T_h\partial_i^{\alpha_1}u,\dots,T_h\partial_i^{\alpha_{r-1}}u,\Delta_h^1u,\partial_i^{\alpha_{r}} u,\dots,\partial_i^{\alpha_m}u,u-v,\partial_i^{\alpha_{m+1}}u,\dots,\partial_i^{\alpha_\ell}u), \ \ \ m\ge r\\
     &h^{m,r}(u,v)=(T_h\partial_i^{\alpha_1}u,\dots,T_h\partial_i^{\alpha_{m-1}}u,T_h\partial_i^{\alpha_m}(u-v),T_h\partial_i^{\alpha_{m+1}}v,\dots,T_h\partial_i^{\alpha_{r-1}}v,\Delta_h^1v,\partial_i^{\alpha_{r}} v,\dots,\partial_i^{\alpha_\ell}v), \ \ \ m\le r-1\\
     &h^{r,r}(u,v)=(T_h\partial_i^{\alpha_1}u,\dots,T_h\partial_i^{\alpha_{r-1}}u,\Delta_h^1(u-v),\partial_i^{\alpha_{r}}v,\dots,\partial_i^{\alpha_\ell}v)\\
     &h^{m,r}(u,v)=(T_h\partial_i^{\alpha_1}u,\dots,T_h\partial_i^{\alpha_{r-1}}u,\Delta_h^1u,\partial_i^{\alpha_{r}} u,\dots,\partial_i^{\alpha_{m-2}}u,\partial_i^{\alpha_m-1}(u-v),\partial_i^{\alpha_{m}}v,\dots,\partial_i^{\alpha_\ell}v), \ \ \ m\ge r+1.
 \end{align*}
 For ${\rm I}_{\ell,r,q},$ 
 by applying Lemma \ref{index 3} to ${\rm I}_{\ell,r,m}$ with $(\alpha_0',\alpha_1',\dots,\alpha_{\ell}')=(0,\alpha_1,\dots,\alpha_{\ell})$ and arguing as in Remark \ref{independence}, one has
 \begin{align*}
      \lf(\int_0^{\infty}\lf(t^{-\left \{ s \right \} }\sup_{|h|\le t}\lf \|{\rm I}_{\ell,r,m}\r\|_p\r)^q \frac{dt}{t}\r)^{1/q}&\lesssim_{s,p,q}\|F\|_{X}^{\max(\|u\|_\fz,\|v\|_\fz),\,\loc}\|u-v\|_{p_{-1}}\|u\|_{B_{p_0,q}^{\{s\}}}\prod_{k=1}^\ell\|\partial^{\alpha_k}_iu\|_{p_k}\\
      &\lesssim_{s,p,q} \|F\|_{X}^{\max(\|u\|_\fz,\|v\|_\fz),\,\loc}\|u\|_{B_{p,q}^s}^{\ell+1}\|u-v\|_{B_{p,q}^s}.
\end{align*}
 For ${\rm II}_{\ell,r,m}$, by applying Lemma \ref{index 2} to $(\alpha_0',\dots,\alpha_{\ell}')=(0,\alpha_1,\dots,\alpha_\ell)$, we similarly have
  \begin{align*}
      \lf(\int_0^{\infty}\lf(t^{-\left \{ s \right \} }\sup_{|h|\le t}\lf \|{\rm II}_{\ell,r,m}\r\|_p\r)^q \frac{dt}{t}\r)^{1/q}&\lesssim_{s,p,q} \|F\|_{X}^{\max(\|u\|_\fz,\|v\|_\fz),\,\loc}\|u\|_{B_{p,q}^s}^{m-1}\|v\|_{B_{p,q}^s}^{\ell-m+1}\|u-v\|_{B_{p,q}^s}\\
&\le \|F\|_{X}^{\max(\|u\|_\fz,\|v\|_\fz),\,\loc}\max(\|u\|_{B_{p,q}^s},\|v\|_{B_{p,q}^s})^{\ell}\|u-v\|_{B_{p,q}^s}.
  \end{align*}
 
Summing over $m,$ we have
\begin{align*}
    {\rm K}^\ell_r\lesssim_{s,p,q} \|F\|_{X}^{\max(\|u\|_\fz,\|v\|_\fz),\mathrm{loc}}\lf(1+\max(\|u\|_{B_{p,q}^s},\|v\|_{B_{p,q}^s})^{[s]+1}\r)\|u-v\|_{B_{p,q}^s}.
\end{align*}
A similar conclusion holds for ${\rm L}^\ell_r$. We then conclude the proof by summing over $r,\ell$.

\end{proof}

In the case $s\in \bN,$ it is difficult to proceed as in the case $s\notin \bN,$ since by \eqref{new eq norm} we have to estimate the second order difference, which cannot be controlled directly by $\|u\|_{B_{p,q}^s}$. Instead, we will rely on the following nonlinear interpolation given in \cite[Theorem 2]{Tartar}. 

\begin{lemma}\label{nonlinear interpolation}
    Let $A_1\subset A_0,B_1\subset B_0$ be Banach spaces, $T$ maps $A_i$ to $B_i,i=0,1.$ Suppose that there are positive non-decreasing continuous functions $f,g:\bR_+\to \bR_+$ such that
    \begin{align}\label{A0}
       \|Ta-Tb\|_{B_0}\le  f\lf(\max(\|a\|_{A_0},\|b\|_{A_0})\r)\|a-b\|_{A_0},\ \ \ a,b\in A_0
    \end{align}
    and
    \begin{align}\label{A1}
       \|Ta\|_{B_1}\le  g\lf(\|a\|_{A_0}\r)\|a\|_{A_1},\ \ \ a\in A_1,
    \end{align}
    then there is $C>0$ and a positive non-decreasing continuous function $h:\bR_+\to\bR_+$ such that $h(t)\le f(2t)^\theta g(2t)^{1-\theta}$ and for every $0<\theta<1,1\le q\le\infty,$ we have
    \begin{align*}
        \|Ta\|_{B_{\theta,q}}\le  h\lf(\|a\|_{A_0}\r)\|a\|_{A_{\theta,q}},\ \ \ a\in A_{\theta,q}.
    \end{align*}
    The definition of interpolation spaces $A_{\theta,q},B_{\theta,q}$ can be seen in \cite{M1989}.
\end{lemma}

By appealing to the above nonlinear interpolation techniques, in addition to Proposition \ref{main re 2}, the following lemma will play an instrumental role in the proof of Theorem \ref{main re 1} in the case $s\in \mathbb N$.

\begin{lemma}\label{gap}
    With the same assumptions as in Theorem \ref{main re 1}, if $[s]>\frac dp,$ one may find $\frac dp<s_0<[s]$ with $s_0\notin \bN$ such that \begin{align*}
\|F(u)\|_{B_{p,q}^s}\lesssim_{s,p,q}\|F\|_X^{\|u\|_\fz\,\loc}\lf(1+\|u\|_{B_{p,q}^{s_0}}^{[s]}\r)\|u\|_{B_{p,q}^s},
    \end{align*}
    where $X=C_b^{\lceil s\rceil}$ when $1<p<\infty$ and $X=\tilde B_{\infty,1}^1\cap \tilde B_{\infty,1}^{\lceil s\rceil}$ when $p=1,\infty.$
\end{lemma}
\begin{proof}
First we assume that $p\ne\infty.$ Via a similar argument as that in the proof of Theorem \ref{main re 1}, it suffices to estimate the terms like ${\rm I}_r^\ell$ and ${\rm J}_r^{\ell}$ in \eqref{analyze0}. Moreover, without loss of generality, we only consider the terms ${\rm J}_r^{\ell}$ and the estimate of ${\rm I}_r^\ell$ is similar. Let us recall that 
$${\rm J}_r^{\ell}=\lf(\int_0^{\infty}\lf(t^{-\left \{ s \right \} }\sup_{|h|\le t}\left \|T_{F^{[\ell]}}^{u,\dots,u}(f^r(u)) \right \|_p\r)^q \frac{dt}{t}\r)^{1/q}$$
for every $\ell=1,\dots,[s]$ and $r=1,\dots,\ell,$ where
$$f^r(u)=(T_h\partial_i^{\alpha_1}u,\dots,T_h\partial_i^{\alpha_{r-1}}u,\partial_i^{\alpha_r}\Delta_h^1 u,\partial_i^{\alpha_{r+1}}u,\dots,\partial_i^{\alpha_\ell}u)$$ with $\alpha_k\in \bN^+$ for $k=1,\dotsm, \ell$ and $\sum_{k=1}^\ell \alpha_k=[s].$
 In contrast to the previous proof, we distinguish the cases $\ell=1$ and $\ell\ge 2,$ and we will show
\begin{align}\label{case1}
    {\rm J}_1^1\lesssim_{s,p,q}\|F\|_{X}^{\|u\|_\fz,\,\loc} \|u\|_{B_{p,q}^s},
\end{align}
and for $\ell\geq2,$
\begin{align}\label{case3}
     {\rm J}_r^{\ell}\lesssim_{s,p,q}\|F\|_X^{\|u\|_\fz,\,\loc}\|u\|_{B_{p,q}^s}\prod_{k=1}^{\ell-1}\|u\|_{B_{p,q}^{s_k^{(\ell)}}}, \  \text{for some} \ \frac dp<s_k^{(\ell)}<[s],s_k^{(\ell)}\notin \bN,
\end{align}
where the case $\ell=1$ coincides with \eqref{Jr} and the case $\ell\geq2$ is stronger than \eqref{Jr}.
Indeed, if we deduce \eqref{case1} and \eqref{case3}, let $s_0'=\max_{\substack{k=1,\dots,\ell-1\\ \ell=2,\dots,[s]}} s_k^{(\ell)}$ and then we have $\frac dp<s_0'<[s],s_0'\notin \bN$ with $$ {\rm J}_r^{\ell}\lesssim_{s,p,q}\|F\|_X^{\|u\|_\fz,\,\loc} \|u\|_{B_{p,q}^{s_0'}}^{\ell-1}\|u\|_{B_{p,q}^s}, \quad \ell=1,\dots,[s],\quad k=1,\dots,\ell-1.$$   
    Similarly we can find  $\frac dp<s_0^{''}<s,s_0^{''}\notin \bN$ such that 
    $$ {\rm I}_r^{\ell}\lesssim_{s,p,q}\|F\|_X^{\|u\|_\fz,\loc} \|u\|_{B_{p,q}^{s_0^{''}}}^{\ell}\|u\|_{B_{p,q}^s}\quad \ell=1,\dots,[s],\quad k=1,\dots,\ell.$$ 
    Taking $s_0=\max(s_0',s_0^{''})$ yields the desired estimate. 
   
   \textbf{Case $\ell=1$}: Proposition \ref{Lp boundedness} and \eqref{Lp bound es 3} yields
    that $${\rm J}_1^1=\lf(\int_0^{\infty}\lf(t^{-\left \{ s \right \} }\sup_{|h|\le t}\left \|T_{F^{[1]}}^{u,u}(\partial_i^{[s]}\Delta_h^1u)\right \|_p\r)^q \frac{dt}{t}\r)^{1/q}\lesssim_{s,p,q}\|F\|_{X}^{\|u\|_\fz,\,\loc} \|u\|_{B_{p,q}^s},$$ which implies \eqref{case1}.

    \textbf{Case $\ell\ge 2$}: Let $(\alpha_0',\dots,\alpha_{\ell-1}')=(\alpha_r,\alpha_1,\dots,\alpha_{r-1},\alpha_{r+1},\dots,\alpha_{\ell})$. As in the proof of Lemma \ref{index 2}, for any $0<\delta^{(\ell)}_k<\alpha_k',$ set $\delta^{(\ell)}:=\sum_{k=1}^{\ell-1}\delta_k^{(\ell)}$ and $$p_0^{(\ell)}:=\frac{[s]p}{\alpha_0'+\delta^{(\ell)}},\, p_k^{(\ell)}:=\frac{[s]p}{\alpha_k'-\delta_k^{(\ell)}}, \ \ \ k=1,\dots,\ell-1.$$ 
The new observation is that the condition $[s]>\frac dp$ allows us to choose $\delta_k^{(\ell)}$ such that  
    \begin{align}\label{delta6}
       \delta_k^{(\ell)}<\min\lf(\alpha_k'+\frac{([s]-\alpha_k')[s]p}{d}-[s],\alpha_k'\r).
    \end{align} 
    and
    \begin{align}\label{delta5}
        \delta^{(\ell)}>\max\lf(\lf(1-\frac{([s]-\alpha_0')p}d\r)[s]-\alpha_0',0\r).
    \end{align}
(Note: \eqref{delta6} is more restrictive than \eqref{delta2} while \eqref{delta5} is the same as \eqref{delta1}.)
 Indeed, similar to the proof of \eqref{alphak'} and \eqref{alpha0}, we can then choose $\varepsilon_k^{(\ell)}>0$ sufficiently small such that
\begin{align*}
    \frac dp<s_k^{(\ell)}:=\alpha_k'+\frac dp-\frac d{p_k^{(\ell)}}+\varepsilon_k^\ell<[s], \ \ \ s_{k}^{\ell}\notin \bN
\end{align*}
and 
\begin{align*}
    s_0^{\ell}=\alpha_0'+\{s\}+\frac dp-\frac d{p_0^{(\ell)}}<s.
\end{align*}
Since $\delta_k^{(\ell)},\varepsilon_k^{(\ell)}$ depend only on $s,p,\alpha_k',$ arguing as in Remark \ref{independence}, we then have
\begin{align*}
        {\rm J}_r^{\ell}&\lesssim_{s,p,q}\|F\|_X^{\|u\|_\fz,\,\loc} \|\partial^{\alpha_0'}u\|_{B_{p_0^{(\ell)},q}^{\{s\}}}\prod_{k=1}^{\ell-1}\|\partial^{\alpha_k'}u\|_{p_k^{(\ell)}}\\
&\lesssim_{s,p,q}\|F\|_X^{\|u\|_\fz,\,\loc}\|u\|_{B_{p,q}^s}\prod_{k=1}^{\ell-1}\|u\|_{B_{p,q}^{s_k^{(\ell)}}}
    \end{align*} 
    and hence obtain \eqref{case3}.

Now we show that we can choose $\delta_k^{(\ell)}$ properly such that \eqref{delta6} and \eqref{delta5} hold simultaneously. Note that $\alpha_k'\le [s]-1$ since $\sum_{j\ne k}\alpha_j'\ge 1$ for any $k=1,\dots,\ell-1.$ Hence by $[s]>\frac dp$ we have
 \begin{align*}
     \min\lf(\alpha_k'+\frac{([s]-\alpha_k')[s]p}{d}-[s],\alpha_k'\r)\ge \min\lf(\frac{[s]p}d-1,\alpha_k'\r)>0.
 \end{align*}
On the other hand, we have
\begin{align*}
    \lf(1-\frac{([s]-\alpha_0')p}d\r)[s]-\alpha_0'<0,
\end{align*}
which implies that
\begin{align*}
   \max\lf(\lf(1-\frac{([s]-\alpha_0')p}d\r)[s]-\alpha_0',0\r)=0.
\end{align*}
Therefore, for any fixed $0<\delta_k^{(\ell)}<\min\lf(\frac{[s]p}d-1,\alpha_k'\r),$ \eqref{delta6} and \eqref{delta5} hold simultaneously.


    For $p=\infty$, the proof is easier. Indeed, set $p_k^{(\ell)}=\infty$ for $\ell=1,\dots,[s]$ and $k=0,\dots,\ell-1,$ the proof is straightforward. When $\ell\ge 2,$ note that $\alpha_k'\le [s]-1,k=0,\dots,\ell-1.$ Hence we have \begin{align*}
        {\rm J}_r^{\ell} &\lesssim_{s,p,q}\|F\|_X^{\|u\|_\fz,\loc}\|\partial^{\alpha_0'}u\|_{B_{\infty,q}^{\{s\}}}\prod_{k=1}^{\ell-1}\|\partial^{\alpha_k'}u\|_\infty\\
        &\lesssim_{s,p,q} \|F\|_X^{\|u\|_\fz,\,\loc}  \|u\|^{\ell-1}_{B_{\infty,q}^{s-1}}\|u\|_{B_{\infty,q}^s},
    \end{align*} 
    and similarly
    \begin{align*}
        {\rm I}_r^\ell \lesssim_{s,p,q} \|F\|_X^{\|u\|_\fz,\,\loc}  \|u\|^{\ell}_{B_{\infty,q}^{s-1}}\|u\|_{B_{\infty,q}^s},
    \end{align*} 
    which concludes the proof.
\end{proof}

\begin{proof}[Proof of Theorem \ref{main re 1} when $s\in\bN$]
For $s\in \bN,$ fix one $s_1$ such that  $s<s_1<s+1$. Then applying Lemma \ref{gap}, one finds $\frac dp<s_0<[s_1]=s,s_0\notin \bN$ and $\kappa(s,p,q)>0$ such that \begin{align*}
\|F(u)\|_{B_{p,q}^{s_1}}&\le \kappa(s,p,q)\|F\|_{X}^{\|u\|_\fz,\,\loc}\lf(1+\|u\|_{B_{p,q}^{s_0}}^{[s_1]}\r)\|u\|_{B_{p,q}^{s_1}}\\
&\le \kappa(s,p,q)\|F\|_{X}^{C'(s_0,p,q)\|u\|_{B_{p,q}^{s_0}},\,\loc}\lf(1+\|u\|_{B_{p,q}^{s_0}}^{s}\r)\|u\|_{B_{p,q}^{s_1}},
\end{align*} 
where $C'(s_0,p,q)$ is the least constant $C>0$ such that $\|\cdot\|_\fz\le C\|\cdot\|_{B_{p,q}^{s_0}}.$
Since $s_0>\frac dp,s_0\notin \bN$ and $\lceil s_0\rceil=s,$ by Proposition \ref{main re 2}, there exists $\kappa'(s,p,q)>0$ such that for any $u,v\in B_{p,q}^{s_0}(\rd)$ we have \begin{align*}
    \|F(u)-F(v)\|_{B_{p,q}^{s_0}}&\lesssim \kappa'(s,p,q)\|F\|_{X}^{\max(\|u\|_\fz,\|v\|_\fz),\,\loc}\lf(1+\max(\|u\|_{B_{p,q}^{s_0}},\|v\|_{B_{p,q}^{s_0}})^{[s_0]+1}\r)\|u-v\|_{B_{p,q}^{s_0}}\\
    &\le \kappa'(s,p,q) \|F\|_{X}^{C'(s_0,p,q)\max(\|u\|_{B_{p,q}^{s_0}},\|v\|_{B_{p,q}^{s_0}}),\,\loc}\lf(1+\max(\|u\|_{B_{p,q}^{s_0}},\|v\|_{B_{p,q}^{s_0}})^{s}\r)\|u-v\|_{B_{p,q}^{s_0}}.
\end{align*} Finally, with Proposition \ref{interpolation} (\rm{iv}) in mind, we apply Lemma \ref{nonlinear interpolation} to $\theta=\frac{s-s_0}{s_1-s_0}$ and
 \begin{align*}
      &f:t\mapsto \kappa'(s,p,q)\|F\|_{X}^{C'(s_0,p,q)t,\,\loc}(1+t^s)\\
      &g:t\mapsto \kappa(s,p,q)\|F\|_{X}^{C'(s_0,p,q)t,\,\loc}(1+t^s),
 \end{align*}
 and conclude the desired estimate \eqref{bound es} by the existence of $h$ such that 
 \begin{align*}
      h(t)\le f(2t)^\theta g(2t)^{1-\theta}=\kappa(s,p,q)^{1-\theta}\kappa'(s,p,q)^\theta \|F\|_{X}^{2C'(s_0,p,q)t,\,\loc}\lf(1+(2t)^{s}\r), \ \ \ \forall t\in\bR_+.
 \end{align*}
 Finally, since $s_0$ is determined only by $s,p,$ the constant $C'(s_0,p,q)$ depends only on $s,p,q,$ hence we set $C(s,p,q):=2C'(s_0,p,q)$ and we then obtain \eqref{hdepend}.

\end{proof}

Finally we prove Theorem \ref{main re 4}. We need the following result, which is a slight improvement of \cite[Section 4]{McNLE}. 
\begin{lemma}\label{Mc}
    Consider $a=(a_j)_{j\in \bN},b=(b_j)_{j\in \bN}\subset L_\infty(\rd)$ such that for any $\alpha\in\bN^d,$ we have
    $$\|\partial^\alpha a_j\|_\infty\lesssim 2^{j|\alpha|}, \ \ \ \|\partial^\alpha b_j\|_\infty\lesssim 2^{j|\alpha|}.$$
    Let $T_{a,b}$ denote the map 
$$T_{a,b}:\lambda_\theta(C^\infty_c(\bR^d))\to \cS'(\rd), \,u\mapsto \sum_{j\in\bN}a_j\triangle_jub_j$$
    and define the quantities
    $$M_k(a):=\sup_{|\alpha|\le k,j\in \bN}2^{-j|\alpha|}\|\partial^\alpha a_j\|_\infty, \ \ \ M_k(b):=\sup_{|\alpha|\le k,j\in \bN}2^{-j|\alpha|}\|\partial^\alpha b_j\|_\infty.$$ Then we can extend $T_{a,b}$ to a map $T_{a,b}:B_{p,q}^s(\rd)\to B_{p,q}^s(\rd)$ such that
    \begin{align*}
        \|T_{a,b}\|_{B_{p,q}^s\to B_{p,q}^s}\lesssim_{s,p,q} M_{\lceil s\rceil}(a)M_{\lceil s\rceil}(b)
    \end{align*}
    for any $p,q\in [1,\infty]$ and $s>0,$
\end{lemma}

\begin{remark}
    In \cite[Theorem 4.3]{McNLE}, McDonald proved that \begin{align*}
        \|T_{a,b}\|_{B_{p,q}^s\to B_{p,q}^s}\lesssim_{s,p,q} M_{s+2}(a)M_{s+2}(b),
    \end{align*}
    and we observe that the bound can be improved to $M_{\lceil s\rceil}(a)M_{\lceil s\rceil}(b)$. Explicitly, one first easily verify that the bound $M_{s+1}(a)M_{s+1}(b)$ in \cite[Lemma 4.5]{McNLE} can be improved to $M_{\tilde s}(a)M_{\tilde s}(b),$ where $\tilde s=\left\{\begin{matrix} 
   \lceil s\rceil, \ \ \ s\notin \bN\\  
  s, \ \ \ \ \ s\in \bN 
\end{matrix}\right.$ This implies that
    \begin{align*}
        &\|T_{a,b}\|_{B_{p,1}^0\to L_p}\lesssim_p M_0(a)M_0(b),\\
        &\|T_{a,b}\|_{B_{p,1}^{s+\epsilon}\to B_{p,1}^{s+\epsilon}}\lesssim_{s,p} M_{\lceil s\rceil}(a)M_{\lceil s\rceil}(b)
    \end{align*}
    for every $\epsilon\in (0,1)$ such that $s+\epsilon<\lceil s\rceil.$ Together with Proposition \ref{interpolation} \rm{(iv)},  we prove Lemma \ref{Mc}. We omit the details of the arguments here.
\end{remark}

\begin{proof}[Proof of Theorem \ref{main re 4}]
    Note that when $F\in W_0(\bR),$ we have $$F(x)=(2\pi)^{-d}\int_{\bR}\cF F(\xi)e^{\ri x\xi}d\xi.$$ Hence by the condition $F(0)=0,$  $$F(u)=(2\pi)^{-d}\int_{\bR}\cF F(\xi)(e^{\ri \xi u}-1)\,d\xi.$$ The well-known Duhamel formula implies that
  \begin{align*}
      e^{i\xi u}-1=G(S_0u)S_0u+\int_0^1\sum_{j\in\bN^+}e^{\ri \theta \xi S_{j-1}u}(\triangle_j u)e^{\ri (1-\theta) \xi S_{j-1}u}\,d\theta.
  \end{align*}
  Here we set $S_j:=\sum_{k=0}^j\triangle_k$ and $G:\eta\mapsto\sum_{n\in\bN^+}\frac{(\ri\eta)^n}{n!}\xi^{n-1}.$ 
  Consider the following operator $$T_\theta: v\mapsto \sum_{j\in\bN^+}e^{\ri \theta \xi S_{j-1}u}(\triangle_j v)e^{\ri (1-\theta) \xi S_{j-1}u}.$$ By Lemma \ref{Mc} and \cite[Lemma 6.7]{McNLE}, we obtain the inequality
    \begin{align*}
        \|e^{i\xi u}-1\|_{B_{p,q}^s}&\lesssim_{s,p,\|u\|_\fz}\int_0^1 \sup_{|\alpha|\le \lceil s\rceil}\|\partial^\alpha e^{\ri \theta \xi u}\|_\fz \sup_{|\beta|\le \lceil s\rceil}\|\partial^\beta e^{\ri (1-\theta) \xi u}\|_\fz \, d\theta \,\|u\|_{B_{p,q}^s}\\
        &\lesssim_{s,\|u\|_\infty} (1+\xi)^{\lceil s\rceil}\,\|u\|_{B_{p,q}^s}.
    \end{align*}
    Note that $$\|F\|_{W_0\cap W_{\lceil s\rceil}}\sim_s \int_{\bR}|\cF F(\xi)|\,(1+\xi)^{\lceil s\rceil}\,d\xi,$$ hence we conclude the proof.
\end{proof}
\begin{remark}
    Our method is also valid for the nonlinear estimate on the quantum Sobolev space. Indeed, from the definition of quantum Sobolev space and the quantum chain rule \eqref{chain rule formula}, one may obtain, by an argument analogous to that in the quantum Besov case, the estimate
$$\|F(u)\|_{W_p^m}\lesssim_{s,p,q,F,u} \|u\|_{W_p^m}$$ for any $m>\frac dp$ under the same assumptions.  We omit the details here.
\end{remark}

\section{Well-posedness theory for noncommutative Allen-Cahn equations}\label{s5}
In this section, we apply Theorem \ref{s smaller than 1}, Theorem \ref{main re 1},  and the contraction mapping principle to prove Theorem \ref{posedness}.

\begin{proof}[Proof of Theorem \ref{posedness}]
We will prove Theorem \ref{posedness} through six steps. Without loss of generality, we may assume that $s\notin \bN,$ otherwise we replace $s$ by $s-\varepsilon$ for a sufficiently small $\varepsilon>0$. 
    
     \textbf{Step 1.} We first show the existence of a local mild solution. Fix $0<\delta\le 1$ and let $T$ be a constant to be determined later. Consider the set $$\cD_\delta:=\{u\in L_\infty([0,T];B_{p,2}^s(\rd)):u(0)=u_0,\, \sup_{t\in [0,T]}\|u(t)\|_{B_{p,2}^s}\le \|u_0\|_{B_{p,2}^s}+\delta\}$$ equipped with metric $d(u,v)=\sup_{t\in [0,T]}\|u(t)-v(t)\|_{p}.$ It is straightforward to see $(\cD_\delta,d)$ is a complete metric space. Let $\Psi$ be the solution map defined by $$\Psi(u)(t)=e^{t\Delta}u_0+\int_0^te^{(t-\tau)\Delta}F(u(\tau))\,d\tau, \ \ \ u\in \mathcal D_\delta.$$ We aim to choose $T>0$ such that $\Psi$ is a contraction on $(\cD_\delta,d).$

     Note that by Theorem \ref{main re 1}, there exists a continuous non-decreasing function $h$ depending only on $s,p,q,F$ such that for any $u\in \cD_\delta,$
     \begin{align*}
         \|F(u(t))\|_{B_{p,2}^s}\le h(\|u_0\|_{B_{p,2}^s}+1)(\|u_0\|_{B_{p,2}^s}+1),
     \end{align*} 
   Hence by Proposition \ref{heat group} \rm{(i)}, for any $0\le t\le T$ we have 
     \begin{align*}
         \|\Psi u\|_{B_{p,2}^s}&\le \|e^{t\Delta}u_0\|_{B_{p,2}^s}+\int_0^t \|e^{(t-\tau)\Delta}F(u(\tau))\|_{B_{p,2}^s}d\tau\\
         &\le \|u_0\|_{B_{p,2}^s}+Th(\|u_0\|_{B_{p,2}^s}+1)(\|u_0\|_{B_{p,2}^s}+1).
     \end{align*}
     On the other hand, by Corollary \ref{Lipes}, for any $x,y\in L_p(\rd),$ we have $$\|F(x)-F(y)\|_p\le C_p\|F\|_{\mathrm{Lip}}^{\|u_0\|_{B_{p,2}^s}+1,\,\mathrm{loc}}\|x-y\|_p.$$ Hence we have 
     \begin{align*}
         d(\Psi u,\Psi v)&\le \sup_{t\in [0,T]}\int_0^t\|F(u(\tau))-F(v(\tau))\|_p\,d\tau\\
         &\le C_p\|F\|_{C^n_b}^{\|u_0\|_{B_{p,2}^s}+1,\,\mathrm{loc}} Td(u,v)
     \end{align*}
     Take $T=T(\|u_0\|_{B_{p,2}^s}+1),$ where we set
\begin{align}\label{TM}
T(M):=\lf(\max\lf(h(M),2C_p\|F\|_{C^n_b}^{M,\,\mathrm{loc}}\r)\r)^{-1},
\end{align} then $\Psi$ is a contraction on $(\cD_\delta,d).$ By the contraction mapping principle, there exists a unique $u\in \cD_\delta$ such that $u$ is a mild solution of \eqref{integral AC}. 

    \textbf{Step 2.} We then show that $u\in C([0,T];B_{p,2}^s(\rd))$ is a unique mild solution. For any $0\le t_1<t_2\le T,$ fix a $t_1$ and we have \begin{align*}
        u(t_2)-u(t_1)&=(e^{t_2\Delta}-e^{t_1\Delta})u_0+\int_{t_1}^{t_2}e^{(t_2-\tau)\Delta}F(u(\tau))\,d\tau+\int_0^{t_1}(e^{(t_2-\tau)\Delta}-e^{(t_1-\tau)\Delta})F(u(\tau))\,d\tau\\
        &=(e^{(t_2-t_1)\Delta}-1)e^{t_1\Delta}u_0+\int_{t_1}^{t_2}e^{(t_2-\tau)\Delta}F(u(\tau))\,d\tau+(e^{(t_2-t_1)\Delta}-1)\int_0^{t_1}e^{(t_1-\tau)\Delta}F(u(\tau))\,d\tau\\
        &=(e^{(t_2-t_1)\Delta}-1)u(t_1)+\int_{t_1}^{t_2}e^{(t_2-\tau)\Delta}F(u(\tau))\,d\tau.
    \end{align*}
Taking $B_{p,2}^{s}(\rd)$ norm at both sides and applying Proposition \ref{heat group} \rm{(i)}, we have
\begin{align*}
    \lim_{t_2\searrow t_1}\|u(t_2)-u(t_1)\|_{B_{p,2}^{s}(\rd)}&\le \lim_{t_2\searrow t_1}\lf(\|(e^{(t_2-t_1)\Delta}-1)u(t_1)\|_{B_{p,2}^{s}(\rd)}+\lf\| \int_{t_1}^{t_2}e^{(t_2-\tau)\Delta}F(u(\tau))\,d\tau
 \r\|_{B_{p,2}^{s}(\rd)}\r)\\
    &\le \lim_{t_2\searrow t_1}\lf(\|(e^{(t_2-t_1)\Delta}-1)u(t_1)\|_{B_{p,2}^{s}(\rd)}+C(t_2-t_1)(\|u_0\|_{B_{p,2}^s}+1)\r)\\
    &=0.
\end{align*}
Hence we deduce that $u\in C([0,T];B_{p,2}^{s}(\rd)).$ 

Now we show that $u$ is a unique mild solution of \eqref{Allen Cahn} in $C([0,T];B_{p,2}^s(\rd)).$ Suppose $v$ is another mild solution in $C([0,T];B_{p,2}^s(\rd)).$ By the continuity, there exists $0<T_1<T$ such that $$\sup_{t\in [0,T_1]}\|v(t)\|_{B_{p,2}^s}\le \|u_0\|_{B_{p,2}^s}+\delta.$$ Utilizing the results from {\bf Step 1}, we deduce that
\begin{align*}
v(t)=u(t), \ \ \ \mathrm{when} \ \ 0\leq t \leq T_1.
\end{align*}
Thus by continuity of $u$ and $v,$ there exists a maximal time $T^*\in[0,T]$ such that
\begin{align*}
v(t)=u(t), \ \ \ \mathrm{for \ \ all} \ \ 0\leq t \le T^*.
\end{align*}
If $T^*=T$, we conclude the proof. Otherwise
consider the following Allen-Cahn equation starting with the initial data
$u(T^*)$.
\begin{align}\label{e3.aa}
w(t)=e^{(t-T^*)\Delta}u(T^*)+\int^t_{T^*}e^{(t-\tau)\Delta}F(u(\tau))\,d\tau,\ \ \ w(T^*)=u(T^*).
\end{align}
Repeating the process above, there exists a small positive number $\epsilon$ such that $T^*<T^*+\epsilon <T$ and the equation \eqref{e3.aa} has a unique mild solution $w$ in
$C([T^*,T^*+\epsilon];B_{p,2}^s(\rd))$.
Moreover, the solutions $u,v$ satisfy
$$\sup_{t\in [T^*,T^*+\epsilon]}\|u(t)\|_{B_{p,2}^s}\le \|u(T^*)\|_{B_{p,2}^s}+\delta, \ \ \
 \sup_{t\in [T^*,T^*+\epsilon]}\|v(t)\|_{B_{p,2}^s}\le \|u(T^*)\|_{B_{p,2}^s}+\delta.$$
Therefore,
\begin{align*}
v(t)=u(t)=w(t), \ \ \ \mathrm{when} \ \ t\in [T^*,T^*+\epsilon],
\end{align*}
which contradicts the maximality of $T^*$. This completes the proof of the above claim.

\textbf{Step 3.} Now we prove the existence of maximal existence time $T_{u_0}.$ We consider the new equation \begin{align*}
v(t)=e^{(t-T)\Delta}u(T)+\int^t_{T}e^{(t-\tau)\Delta}F(u(\tau))\,d\tau,\ \ \ v(T)=u(T).
\end{align*} Since $\|u(T)\|_{B_{p,2}^s}\le \|u_0\|_{B_{p,2}^s}+\delta,$ repeating the procedure in \textbf{Step 1}, we can find $T'>T$ and a unique solution $v\in C([T,T'];B_{p,2}^s(\rd)).$ Defining \begin{eqnarray*}\tilde{u}(t):=
\lf\{\begin{array}{ll}
u(t),  \ \ \mathrm{when} \ \ t\le T;\\[5pt]
v(t), \ \ \mathrm{when} \ \  T\leq t\le T'.
\end{array}\r.,
\end{eqnarray*}
then we extend the solution to $[0,T_1].$ Step by step, we will extend the solution to $[0,T_{u_0})$ for a maximal time $T_{u_0},$ which is also denoted by $u.$ Hence we have $u\in C([0,T_{u_0});B_{p,2}^s(\rd)).$ Moreover, if $T_{u_0}<\infty,$ then we should have $$\limsup_{t\nearrow T_{u_0}}\|u(t)\|_{B_{p,2}^s}=\infty.$$ Otherwise there exists $M>0$ such that $\sup_{0\le t<T_{u_0}}\|u(t)\|_{B_{p,2}^s}\le M.$ Take $\rho=T(M)$ with $T(M)$ defined in \eqref{TM}. Since $\|u(T_{u_0}-\frac\rho 2)\|_{B_{p,2}^s}\le M,$ we can extend $u$ to $[T_{u_0}-\frac \rho 2,T_{u_0}+\frac \rho 2]$ similarly as in \textbf{Step 1}, which contradicts the maximality of $T_{u_0}.$

    \textbf{Step 4.}  We show that $u\in \bigcap_{0<\alpha<n+1} C((0,T_{u_0});B_{p,2}^\alpha(\rd))$ by an iteration argument. Let $K\in \bN$ be the largest integer such that $s+K<n+1.$ We may assume that $K\in \bN^+,$ since $K=0$ is simpler. Note that for any $0<T_2<T_3<T_{u_0}$ and $t\in [T_2,T_3],$ from \eqref{heatdiff} we have
\begin{align*}
    \left \| \int_0^te^{(t-\tau)\Delta}F(u(\tau))d\tau \right \|_{B_{p,2}^{s+1}}&\le  \int_0^t\left \|e^{(t-\tau)\Delta}F(u(\tau))\right \|_{B_{p,2}^{s+1}}d\tau \\
&\lesssim_{s,p} \int_0^t\lf(1+(t-\tau)^{-\frac {1}2}\r)\left \|F(u(\tau))\right \|_{B_{p,2}^{s}}\,d\tau,
\end{align*}
  hence by Theorem \ref{main re 1} we have \begin{align*}
         \sup_{t\in [T_2,T_3]}\left \|u(t)\right \|_{B_{p,2}^{s+1}}&\leq \sup_{t\in [T_2,T_3]}\left \|e^{t\Delta}u_0\right \|_{B_{p,2}^{s+1}}+\sup_{t\in [T_2,T_3]}\left \| \int_0^te^{(t-\tau)\Delta}F(u(\tau))\,d\tau \right \|_{B_{p,2}^{s+1}}\\
        &\lesssim_{s,p,F} (1+T_2^{-\frac 12})\left \|u_0\right \|_{B_{p,2}^{s}}+\lf(T_3+T_3^{\frac 12}\r)\sup_{t\in [T_2,T_3]}h(\|u(t)\|_{B_{p,2}^s})\|u(t)\|_{B_{p,2}^{s}}.
    \end{align*}
    This implies that $u\in L_{\infty}([T_2,T_3];B_{p,2}^s(\rd)).$ As in \textbf{Step 2}, one may show $u\in C([T_2,T_3];B_{p,2}^{s+1}(\rd)).$ Since $T_2,T_3$ is arbitrary, $u\in C((0,T_{u_0});B_{p,2}^{s+1}(\rd)).$ Step by step, one may show $u\in C((0,T_{u_0});B_{p,2}^{s+K}(\rd)).$

    For $s+K<\alpha<n+1,$ note that by Theorem \ref{main re 1} and \eqref{heatdiff},  
    \begin{align*}
         \sup_{t\in [T_2,T_3]}\left \|u(t)\right \|_{B_{p,2}^{\alpha}}&\leq \sup_{t\in [T_2,T_3]}\left \|e^{t\Delta}u_0\right \|_{B_{p,2}^{\alpha}}+\sup_{t\in [T_2,T_3]}\left \| \int_0^te^{(t-\tau)\Delta}F(u(\tau))\,d\tau \right \|_{B_{p,2}^{\alpha}}\\
        &\lesssim_{s,p,F} (1+T_2^{-\frac {\alpha-s}2})\left \|u_0\right \|_{B_{p,2}^{s}}+\lf(T_3+T_3^{\frac{s+K+1-\alpha}{2}}\r)h(\|u(t)\|_{B_{p,2}^{s+K}})\sup_{t\in [T_2,T_3]}\|u(t)\|_{B_{p,2}^{s+K}}.
    \end{align*}
Hence we similarly have $u\in C((0,T_{u_0});B_{p,2}^{\alpha}(\rd)).$

\textbf{Step 5.} We will show $u$ is indeed a strong solution when $n\ge 2$. That is, fix an arbitrary $0<t<T_{u_0},$ then we have 
\begin{align*}
    \lim_{h\to 0}\lf\|\frac{u(t+h)-u(t)}{h}-\Delta u(t)-F(u(t))\r\|_{p}=0.
\end{align*} 
Set $f:=F(u),$ then for any $t\in (T_1,T_2)$ and $T_1-t<h<T_2-t$ we have 
\begin{align*}
    \frac{u(t+h)-u(t)}{h}
    =&\frac{e^{(t-T_1+h)\Delta}u(T_1)-e^{(t-T_1)\Delta}u(T_1)}{h}+\int_{T_1}^t \frac{e^{(t-\tau+h)\Delta}f(\tau)-e^{(t-\tau)\Delta}f(\tau)}{h}\, d\tau \\
  &\quad +\frac 1h\int_t^{t+h} e^{(t-\tau+h)\Delta}f(\tau)\, d\tau
    \\
    =&\frac{e^{h\Delta}-1}{h}\lf(e^{(t-T_1)\Delta}u(T_1)-\int_{T_1}^t e^{(t-\tau)\Delta}f(\tau)\, d\tau\r)+\frac 1h\int_t^{t+h}e^{(t-\tau+h)\Delta}f(\tau)\, d\tau
    \\
    =&\frac{e^{h\Delta}-1}{h}u(t)+\frac 1h\int_t^{t+h}e^{(t-\tau+h)\Delta}f(\tau)\, d\tau.
\end{align*}
For the first term, note that since $u\in \bigcap_{0<\alpha<n+1}L_{\fz}([T_1,T_2];B_{p,2}^\alpha(\rd))\subset L_{\fz}([T_1,T_2];B_{p,2}^2(\rd)),$ from Proposition \ref{heat group} \rm{(ii)} and \cite[Proposition 2.1.4 \rm{(iii)}]{Lunardi}, we have
\begin{align*}
    \lim_{h\to 0}\lf\|\frac{e^{h\Delta}-1}{h}u(t)-\Delta u(t)\r\|_{p}=0.
\end{align*}
For the second term, we rewrite
\begin{align*}
    \frac 1h\int_t^{t+h}e^{(t-\tau+h)\Delta}f(\tau)\, d\tau&=\frac 1h\lf(\int_t^{t+h}e^{(t-\tau+h)\Delta}f(\tau)-f(t)\, d\tau\r)+f(t).
\end{align*}
Then Theorem \ref{main re 1} and Corollary \ref{Lipes} imply that $f\in \bigcap_{0<\alpha<n}L_{\fz}([T_1,T_2];B_{p,2}^\alpha(\rd))$ and $f\in C((0,T_{u_0});L_p(\rd)).$ Moreover,
\begin{align*}
   &\lf\|\frac 1h\int_t^{t+h}e^{(t-\tau+h)\Delta}f(\tau)-f(t)\, d\tau\r\|_{p}\\
   \le& \frac 1h\int_t^{t+h}\big\|e^{(t-\tau+h)\Delta}f(\tau)-f(t)\big\|_{p}\, d\tau
   \\ \le& \frac 1h\int_t^{t+h} \big\|e^{(t-\tau+h)\Delta}f(\tau)-e^{(t-\tau+h)\Delta}f(t)\big\|_{p}\, d\tau\\
   &\ \ \ +\frac 1h\int_t^{t+h}\big\|e^{(t-\tau+h)\Delta}f(t)-e^{h\Delta}f(t)\big\|_{p}\, d\tau\\
   &\ \ \ +\frac 1h\int_t^{t+h}\big\|e^{h\Delta}f(t)-f(t)\big\|_{p}\, d\tau\\
   \le& \frac{1}{h}\int_t^{t+h}\big\|f(\tau)-f(t)\big\|_{p}\,d\tau
    +\frac{1}{h}\int_t^{t+h}\lf\|\Big(e^{(t-\tau)\Delta}-I\Big)f(t)\r\|_{p}\,d\tau\\
   &\ \ \ +\lf\|\lf(e^{h\Delta}-I\r)f(t)\r\|_{p}.
\end{align*}
Thus we deduce that
\begin{align*}
   &\lim_{h\rightarrow0}\lf\|\frac 1h\int_t^{t+h}e^{(t-\tau+h)\Delta}f(\tau)-f(t)\, d\tau\r\|_{p}\\
   \le& \lim_{h\rightarrow0}\frac{1}{h}\int_t^{t+h}\big\|f(\tau)-f(t)\big\|_{p}ds
+\lim_{h\rightarrow0}\frac{1}{h}\int_t^{t+h}\lf\|\Big(e^{(t-\tau)\Delta}-I\Big)f(t)
\r\|_{p}d\tau\\&\quad+\lim_{h\rightarrow 0}\lf\|\lf(e^{h\Delta}-I\r)f(t)\r\|_{p}=0.
\end{align*}
Altogether, we have 
 \begin{equation*}
   \partial_t u\in L_{\fz}([T_1,T_2];L_p(\rd)),
 \ \ \ \mathrm{for\ all} \ 0<T_1<T_2<T_{u_0},
 \end{equation*}
and $$\partial_t u(t)=\Delta u(t)+ F(u(t)).$$
Since $\Delta u,F(u)\in C([0,T_{u_0});L_p(\rd)),$ we have $\partial_t u\in C((0,T_{u_0});L_p(\rd)),$ which implies that $u\in C^1((0,T_{u_0});L_p(\rd)).$

\textbf{Step 6.} Finally, we show that $u$ is a global mild solution when $F$ is Lipschitz. Assume that $T_{u_0}<\fz.$
Since $0<s<n$ and $s\notin \bN,$ we have $0<\{s\}<1$ and $u\in C([0,T_{u_0});B_{p,2}^{\{s\}}(\rd)).$ If $T_{u_0}<\infty,$ for any $0\le t<T_{u_0},$ by Theorem \ref{s smaller than 1} we have 
\begin{align*}
    \|u(t)\|_{B_{p,2}^{\{s\}}}\le \|u_0\|_{B_{p,2}^{\{s\}}}+C(s,p,q)\|F\|_{\mathrm{Lip}}\int_0^t\|u(\tau)\|_{B_{p,2}^{\{s\}}}d\tau.
\end{align*}
Hence we apply the Gronwall inequality to $t\mapsto \|u(t)\|_{B_{p,2}^{\{s\}}}$ and obtain that 
\begin{align*}
    \|u(t)\|_{B_{p,2}^{\{s\}}}\le \|u_0\|_{B_{p,2}^{\{s\}}}\exp\lf(C(s,p,q)\|F\|_{\mathrm{Lip}}T_{u_0}\r).
\end{align*}
Hence $u\in L_\infty([0,T_{u_0});B_{p,2}^{\{s\}}(\rd)).$ By an iteration argument as in \textbf{Step 4}, we have $u\in L_\infty([0,T_{u_0});B_{p,2}^{s}(\rd)).$ However, we also have $$\limsup_{t\nearrow T_{u_0}}\|u(t)\|_{B_{p,2}^s}=\infty,$$ which yields a contradiction.

\end{proof}

\textbf{Acknowledgements.} This work is partially supported by National Natural Science Foundation of China (No. 12071355, No. 12325105, No. 12031004 and No. W2441002).


\bigskip

\end{document}